\documentclass[11pt]{amsart}

\usepackage[colorlinks=true, pdfstartview=FitV, linkcolor=blue, citecolor=blue, urlcolor=blue]{hyperref}
 
\usepackage[letterpaper,margin=1in]{geometry}
\usepackage{enumerate, amsmath, amsfonts, amssymb, amsthm, mathtools, thmtools, wasysym, graphics,
xcolor, frcursive,xparse,comment,bbm}
\usepackage{graphicx, xspace, ifthen, rotating, pict2e}
\usepackage{multirow,multicol,mathtools}
\usepackage{tikz}
\usepackage{ytableau}
\usepackage[enableskew]{youngtab}
\usepackage[cmtip,all]{xy}
\usepackage{young}
\usepackage{subcaption}

\definecolor{darkblue}{rgb}{0.0,0,0.7} 
\definecolor{darkred}{rgb}{0.7,0,0} 
\definecolor{darkgreen}{rgb}{0, .6, 0} 

\newcommand{\defn}[1]{{\color{darkred}\emph{#1}}} 

\usepackage[colorinlistoftodos]{todonotes}

\newcommand{\wt}{\operatorname{wt}}

\newtheorem{theorem}{Theorem}[section]
\newtheorem{proposition}[theorem]{Proposition}
\newtheorem{lemma}[theorem]{Lemma}
\newtheorem{corollary}[theorem]{Corollary}

\newtheorem{claim}[theorem]{Claim}

\theoremstyle{definition}

\newtheorem{definition}[theorem]{Definition}
\newtheorem{remark}[theorem]{Remark}
\newtheorem{example}[theorem]{Example}

\numberwithin{equation}{section}

\title{Characterization of queer supercrystals}

\author[M.~Gillespie]{Maria Gillespie}
\author[G.~Hawkes]{Graham Hawkes}
\author[W.~Poh]{Wencin Poh}
\author[A.~Schilling]{Anne Schilling}
\address[M. Gillespie, G. Hawkes, W. Poh, A. Schilling]{Department of Mathematics, University of California, One Shields
Avenue, Davis, CA 95616-8633, U.S.A.}
\email{anne@math.ucdavis.edu}
\urladdr{http://www.math.ucdavis.edu/\~{}anne}


\begin{document}

\begin{abstract}
We provide a characterization of the crystal bases for the quantum queer superalgebra recently
introduced by Grantcharov et al.. This characterization is a combination of local queer axioms generalizing Stembridge's local 
axioms for crystal bases for simply-laced root systems, which were recently introduced by Assaf and Oguz, with further axioms
and a new graph $G$ characterizing the relations of the type $A$ components of the queer crystal. We provide a 
counterexample to Assaf's and Oguz' conjecture that the local queer axioms uniquely characterize the queer supercrystal. 
We obtain a combinatorial description of the graph $G$ on the type $A$ components by providing explicit combinatorial rules 
for the odd queer operators on certain highest weight elements.
\end{abstract}

\maketitle{}

\section{Introduction}

The representation theory of Lie algebras is of fundamental importance, and hence combinatorial
models for representations, especially those amenable to computation, are of great use.
In the 1990's, Kashiwara~\cite{Kashiwara.1991} showed that integrable highest weight representations of the 
Drinfeld--Jimbo quantum groups $U_q(\mathfrak{g})$, where $\mathfrak{g}$ is a symmetrizable Kac--Moody Lie algebra, 
in the $q \to 0$ limit result in a combinatorial skeleton of the integrable representation. He coined the term crystal bases, 
reflecting the fact that $q$ corresponds to the temperature of the underlying physical system. Since then, crystal bases 
have appeared in many areas of mathematics, including algebraic geometry, combinatorics, mathematical physics, 
representation theory, and number theory. One of the major advances in the theory of crystals for simply-laced Lie 
algebras was the discovery by Stembridge~\cite{Stembridge.2003} of local axioms that uniquely characterize 
the crystal graphs corresponding to Lie algebra representations. These local axioms provide a completely combinatorial 
approach to the theory of crystals; this viewpoint was taken in~\cite{BumpSchilling.2017}.

Lie superalgebras~\cite{Kac.1977} arose in physics in theories that unify bosons and fermions. They are essential in 
modern string theories~\cite{GGRS.1983} and appear in other areas of mathematics, such as the projective representations 
of the symmetric group. The crystal basis theory has been developed for various quantum 
superalgebras~\cite{BKK.2000,Grantcharov.2010a,Granthcharov.2010,Grantcharov.2014, Grantcharov.2015,Grantcharov.2017,
Kwon.2015, Kwon.2016}. In this paper, we are in particular interested in the queer superalgebra $\mathfrak{q}(n)$
(see for example~\cite{ChengWang.2012}). A theory of highest weight crystals for the queer superalgebra $\mathfrak{q}(n)$
was recently developed by Grantcharov et al.~\cite{Granthcharov.2010,Grantcharov.2014, Grantcharov.2015}. They provide 
an explicit combinatorial realization of the highest weight crystal bases in terms of semistandard decomposition tableaux
and show how these crystals can be derived from a tensor product rule and the vector representation. They also 
use the tensor product rule to derive a Littlewood--Richardson rule. 
Choi and Kwon~\cite{ChoiKwon.2017} provide a new characterization of Littlewood--Richardson--Stembridge tableaux 
for Schur $P$-functions by using the theory of $\mathfrak{q}(n)$-crystals. 
Independently, Hiroshima~\cite{Hiroshima.2018} and Assaf and Oguz~\cite{AssafOguz.2018,AssafOguz.2018a}
defined a queer crystal structure on semistandard shifted tableaux, extending the type $A$ crystal structure
of~\cite{HawkesParamonovSchilling.2017} on these tableaux. 

In this paper, we provide a characterization of the queer supercrystals in analogy to Stembridge's~\cite{Stembridge.2003} 
characterization of crystals associated to classical simply-laced root systems. 
Assaf and Oguz~\cite{AssafOguz.2018,AssafOguz.2018a} conjecture a local characterization of queer crystals in the 
spirit of Stembridge~\cite{Stembridge.2003}, which involves local relations between the odd crystal operator
$f_{-1}$ with the type $A_{n-1}$ crystal operators $f_i$ for $1\leqslant i <n$. However, we provide a counterexample
to~\cite[Conjecture~4.16]{AssafOguz.2018a}, which conjectures that these local axioms uniquely characterize
the queer supercrystals. Instead, we define a new graph $G(\mathcal{C})$ on the relations between the type $A$ 
components of the queer supercrystal $\mathcal{C}$, which together with Assaf's and Oguz' local queer axioms
and further new axioms uniquely fixes the queer crystal structure (see Theorem~\ref{theorem.main}).
We provide a combinatorial description of $G(\mathcal{C})$ by providing the combinatorial rules for all
odd queer crystal operators $f_{-i}$ and $e_{-i}$ on certain highest weight elements for $1\leqslant i<n$.

This paper is structured as follows. In Section~\ref{section.queer}, we review the combinatorial definition
of the queer supercrystals by~\cite{Granthcharov.2010, Grantcharov.2014, Grantcharov.2015} and prove several results 
that are needed later for the combinatorial description of the graph $G(\mathcal{C})$. In particular, Theorems~\ref{theorem.f-i}
and~\ref{theorem.e-i} provide explicit combinatorial descriptions of the odd queer crystal operators $f_{-i}$ and $e_{-i}$
on highest weight elements.
In Section~\ref{section.axioms}, we state the local queer axioms by Assaf and Oguz~\cite{AssafOguz.2018,AssafOguz.2018a}
and provide a counterexample to~\cite[Conjecture~4.16]{AssafOguz.2018a}. The graph  $G(\mathcal{C})$ is introduced in 
Section~\ref{section.G(C)}. Theorem~\ref{theorem.combinatorial G} allows us to transform $G(\mathcal{C})$ into 
combinatorial graphs $\overline{G}(\mathcal{C})$ and $\widetilde{G}(\mathcal{C})$, 
which together with the local queer axioms of Definition~\ref{definition.local queer axioms} and new connectivity
axioms of Definition~\ref{definition.connectivity axioms} uniquely characterize the queer crystals as stated in 
Theorem~\ref{theorem.main}.

\subsection*{Acknowledgments}
We are grateful to Sami Assaf, Dan Bump, Zach Hamaker, Ezgi Oguz, and Travis Scrimshaw for helpful discussions. 
We would also like to thank Dimitar Grantcharov, Ji-Hye Jung, and Masaki Kashiwara for answering our questions about 
their work. The last two authors have implemented the queer supercrystals in  {\sc SageMath}~\cite{combinat,sage}. 
This work benefited from experimentations in {\sc SageMath}.

This work was partially supported by NSF grants  DMS--1500050, DMS--1760329, DMS--1764153 and 
NSF MSPRF grant PDRF 1604262.

\section{Queer supercrystals}
\label{section.queer}

In Section~\ref{section.definition queer supercrystal}, we review the queer crystals constructed 
in~\cite{Granthcharov.2010, Grantcharov.2014, Grantcharov.2015}. In Section~\ref{section.properties}, we review
some properties of queer crystals discovered in~\cite{AssafOguz.2018,AssafOguz.2018a}.
In Section~\ref{section.f-i and e-i}, we provide new explicit combinatorial descriptions of $f_{-i}$ and $e_{-i}$
on certain highest weight elements, which will be used in Section~\ref{section.G(C)} to construct the graph
$G(\mathcal{C})$. In Section~\ref{section.proposition bypass}, we provide relations between $e_{-i}$ when acting
on certain highest weight elements, which will be used in Section~\ref{section.G(C)} to deal with ``by-pass arrows'' in 
the component graph $G(\mathcal{C})$.

\subsection{Definition of queer supercrystals}
\label{section.definition queer supercrystal}

An \defn{(abstract) crystal} of type $A_n$ is a nonempty set $B$ together with the maps
\begin{equation}
\begin{split}
	e_i, f_i &\colon B \to B \sqcup \{0\} \qquad \text{for $i \in I$,}\\
        \wt &\colon B \to \Lambda,
\end{split}
\end{equation}
where $\Lambda = \mathbb{Z}^{n+1}_{\geqslant 0}$ is the weight lattice of the root of type $A_n$ and $I=\{1,2,\ldots,n\}$
is the index set, subject to several conditions. Denote by $\alpha_i =\epsilon_i -\epsilon_{i+1}$ for $i \in I$ the simple 
roots of type $A_n$, where $\epsilon_i$ is the $i$-th standard basis vector of $\mathbb{Z}^{n+1}$. Then we require:
\begin{itemize}
\item[\textbf{A1.}]  For $b,b'\in B$, we have $f_ib=b'$ if and only if $b=e_ib'$. In this case $\wt(b') =\wt(b)-\alpha_i$.
\end{itemize}
For $b\in B$, we also define
\[
	\varphi_i(b) = \max\{k \in \mathbb{Z}_{\geqslant 0} \mid f_i^k(b) \neq 0\}
	\qquad \text{and} \qquad
	\varepsilon_i(b) = \max\{k \in \mathbb{Z}_{\geqslant 0} \mid e_i^k(b) \neq 0\}.
\]
For further details, see for example~\cite[Definition 2.13]{BumpSchilling.2017}.

There is an action of the symmetric group $S_n$ on a type $A_n$ crystal $B$ given by the operators 
\begin{equation}
\label{equation.si}
	s_i(b) = \begin{cases}
	f_i^k(b) & \text{if $k\geqslant 0$,}\\
	e_i^{-k}(b) & \text{if $k<0$,}
	\end{cases}
\end{equation}
for $b\in B$, where $k=\varphi_i(b)-\varepsilon_i(b)$.

An element $b\in B$ is called \defn{highest weight} if $e_i(b)=0$ for all $i \in I$. Similarly, $b$ is called
\defn{lowest weight} if $f_i(b)=0$ for all $i \in I$. For a subset $J \subseteq I$, we say that $b$ is $J$-highest weight
if $e_i(b)=0$ for all $i \in J$ and similarly $b$ is $J$-lowest weight if $f_i(b)=0$ for all $i \in J$.

We are now ready to define an abstract queer crystal.

\begin{definition} \cite[Definition 1.9]{Grantcharov.2014}
\label{definition.abstract queer}
An \defn{abstract $\mathfrak{q}(n+1)$-crystal} is a type $A_n$ crystal $B$ together with the maps
$e_{-1},f_{-1} \colon B \to B \sqcup \{0\}$ satisfying the following conditions:
\begin{enumerate}
\item[\textbf{Q1.}] $\wt(B) \subset \Lambda$;
\item[\textbf{Q2.}] $\wt(e_{-1} b) = \wt(b) + \alpha_1$ and $\wt(f_{-1} b) = \wt(b) - \alpha_1$;
\item[\textbf{Q3.}] for all $b,b'\in B$, $f_{-1}b = b'$ if and only if $b = e_{-1} b'$;
\item[\textbf{Q4.}] if $3\leqslant i \leqslant n$, we have
\begin{enumerate}
\item the crystal operators $e_{-1}$ and $f_{-1}$ commute with $e_i$ and $f_i$;
\item if $e_{-1}b \in B$, then $\varepsilon_i(e_{-1}b) = \varepsilon_i(b)$ and $\varphi_i(e_{-1}b) = 
\varphi_i(b)$.
\end{enumerate}
\end{enumerate}
\end{definition}

Given two $\mathfrak{q}(n+1)$-crystals $B_1$ and $B_2$, Grantcharov et al.~\cite[Theorem 1.8]{Grantcharov.2014}
provide a crystal on the tensor product $B_1 \otimes B_2$, which we state here in reverse convention.
It consists of the type $A_n$ tensor product rule (see for example~\cite[Section 2.3]{BumpSchilling.2017})
and the \defn{tensor product rule} for $b_1 \otimes b_2 \in B_1 \otimes B_2$
\begin{equation}
\label{equation.tensor product}
\begin{split}
	e_{-1} (b_1 \otimes b_2) &= \begin{cases}
	b_1 \otimes e_{-1} b_2 & \text{if $\wt(b_1)_1 = \wt(b_1)_2 =0$,}\\
	e_{-1} b_1 \otimes b_2 & \text{otherwise,}
	\end{cases}\\
	f_{-1} (b_1 \otimes b_2) &= \begin{cases}
	b_1 \otimes f_{-1} b_2 & \text{if $\wt(b_1)_1 = \wt(b_1)_2 =0$,}\\
	f_{-1} b_1 \otimes b_2 & \text{otherwise.}
	\end{cases}
\end{split}
\end{equation}

The crystals of interest are the \defn{crystals of words} $\mathcal{B}^{\otimes \ell}$, where $\mathcal{B}$ 
is the $\mathfrak{q}(n+1)$-queer crystal of letters depicted in Figure~\ref{figure.queer letter}.

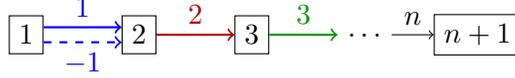
\begin{figure}
\begin{center}
\begin{tikzpicture}[node distance = 1.5cm, auto]
  \node[draw] (1)  {$1$};
  \node[draw] (2) [right of=1] {$2$};
  \node[draw] (3) [right of=2]{$3$};
  \node     (A) [right of=3]{$\ldots$};
  \node[draw] (n+1) [right of=A]{$n+1$};
  \draw[->,transform canvas={yshift=1mm},blue,thick] (1) edgenode[above,blue]{$1$} (2);
  \draw[->,dashed,transform canvas={yshift=-1mm},blue,thick] (1) edgenode[below,blue]{$-1$} (2);
  \draw[->,darkred,thick] (2) edgenode[above,darkred]{$2$} (3);
  \draw[->,darkgreen,thick] (3) edgenode[above,darkgreen]{$3$} (A);
  \draw[->] (A) edgenode[above]{$n$} (n+1);
\end{tikzpicture}
\end{center}
\caption{$\mathfrak{q}(n+1)$-queer crystal of letters $\mathcal{B}$ \label{figure.queer letter}}
\end{figure}

In addition to the queer crystal operators $f_{-1},f_1,\ldots,f_n$ and $e_{-1},e_1,\ldots,e_n$, we define the crystal operators
for $1<i\leqslant n$
\begin{equation}
\label{equation.-i}
	f_{-i} := s_{w_i^{-1}} f_{-1} s_{w_i} \qquad \text{and} \qquad e_{-i} := s_{w_i^{-1}} e_{-1} s_{w_i},
\end{equation}
where $s_{w_i}=  s_2 \cdots s_i s_1 \cdots s_{i-1}$ and $s_i$ is the reflection along the $i$-string in the crystal defined
in~\eqref{equation.si}. Furthermore for $i \in I_0:=\{1,2,\ldots,n\}$
\begin{equation}
\label{equation.-ip}
	f_{-i'} := s_{w_0} e_{-(n+1-i)} s_{w_0} \qquad \text{and} \qquad e_{-i'} := s_{w_0} f_{-(n+1-i)} s_{w_0},
\end{equation}
where $w_0$ is the long word in the symmetric group $S_{n+1}$. By~\cite[Theorem 1.14]{Grantcharov.2014}, with
all operators $e_i,f_i$ for $i\in \{-1,-2,\ldots,-n,1,2,\ldots,n\}$ each connected component of $\mathcal{B}^{\otimes \ell}$
has a unique highest weight vector and with all operators $e_i,f_i$ for $i\in \{-1',-2',\ldots,-n',1,2,\ldots,n\}$ each connected 
component of $\mathcal{B}^{\otimes \ell}$ has a unique lowest weight vector.

\subsection{Properties of queer supercrystals}
\label{section.properties}

We now review and prove several properties about the queer crystal operators.

\begin{lemma}
\label{lemma.f recursion}
For $1\leqslant i < n$, we have
\begin{equation}
\begin{split}
	f_{-(i+1)} &= (s_i s_{i+1}) \, f_{-i} \, (s_{i+1} s_i),\\
	e_{-(i+1)} &= (s_i s_{i+1}) \, e_{-i} \, (s_{i+1} s_i).
\end{split}
\end{equation}
\end{lemma}

\begin{proof}
We use the definition~\eqref{equation.-i}. Note that the following recursion holds
\begin{equation}
\label{equation.w recursion}
	s_{w_{i+1}} = (s_2 \cdots s_{i+1})(s_1\cdots s_i) = (s_2 \cdots s_i)(s_1 \cdots s_{i-1})s_{i+1} s_i = s_{w_i} s_{i+1} s_i,
\end{equation}
which implies the statement.
\end{proof}

\begin{remark}\label{rmk:signature-rule}
The operators $f_i$ for $i \in I_0$ have an easy combinatorial description on $b\in \mathcal{B}^{\otimes \ell}$
given by the \defn{signature rule}, which can be directly derived from the tensor product rule (see for
example~\cite[Section 2.4]{BumpSchilling.2017}). One can consider $b$ as a word in the alphabet $\{1,2,\ldots,n+1\}$. 
Consider the subword of $b$ consisting only of the letters $i$ and $i+1$. Pair (or bracket) any consecutive letters $i+1 , i$ 
in this order, remove this pair, and repeat. Then $f_i$ changes the rightmost unpaired $i$ to $i+1$; if there is no such letter 
$f_i(b)=0$. Similarly, $e_i$ changes the leftmost unpaired $i+1$ to $i$; if there is no such letter $e_i(b)=0$.
\end{remark}

\begin{remark}\label{rmk:combo-minus-1}
From~\eqref{equation.tensor product}, one may also derive a simple combinatorial rule for $f_{-1}$ and $e_{-1}$.
Consider the subword $v$ of $b\in \mathcal{B}^{\otimes \ell}$ consisting of the letters $1$ and $2$.
The crystal operator $f_{-1}$ on $b$ is defined if the leftmost letter of $v$ is a $1$, in which case it turns it into a $2$.
Otherwise $f_{-1}(b)=0$. Similarly, $e_{-1}$ on $b$ is defined if the leftmost letter of $v$ is a $2$, in which case it turns it 
into a $1$. Otherwise $e_{-1}(b)=0$. 
\end{remark}

Lemmas~\ref{lemma.e1} and~\ref{lemma.e2} have appeared in~\cite{AssafOguz.2018,AssafOguz.2018a}.
We provide proofs for completeness.

\begin{lemma}
\label{lemma.e1}
Let $b\in \mathcal{B}^{\otimes \ell}$. The following holds:
\begin{enumerate}
\item
If $\varphi_1(b) \geqslant 2$ and $\varphi_{-1}(b)=1$, we have $\varphi_1(b) = \varphi_1(f_{-1}(b))+2$ and
$\varepsilon_1(b) = \varepsilon_1(f_{-1}(b))$. If furthermore $\varphi_1(b) > 2$, then
\[
	f_1 f_{-1}(b) = f_{-1}f_1(b).
\]
\item 
If $\varphi_1(b)=\varphi_{-1}(b)=1$, we have
\[
	f_1(b) = f_{-1}(b).
\]
\item
If $\varepsilon_1(b),\varepsilon_{-1}(b)>0$ and $e_1(b)\neq e_{-1}(b)$, we have $\varepsilon_1(b) = \varepsilon_1(e_{-1}(b))$,
$\varphi_1(b) = \varphi_1(e_{-1}(b))-2$, and 
\[
	e_1e_{-1}(b) = e_{-1}e_1(b).
\]
\end{enumerate}
\end{lemma}

\begin{proof}
Let $p=\varphi_1(b)$ and $q=\varepsilon_1(b)$. Consider the subword $v$ consisting of all letters 1 and 2 in $b$.
After performing 1,2-bracketing onto $v$ according to the signature rule, we have  a subword of unbracketed letters in $b$ as  
\begin{equation}
\label{equation.12 word}
	v_{i_1}v_{i_2}\ldots v_{i_p}v_{j_1}\ldots v_{j_q},
\end{equation}
where $v_{i_k}=1$ for all $1\leqslant k\leqslant p$ and $v_{j_k}=2$ for all $1\leqslant k\leqslant q$. 
\begin{enumerate}
\item We assume that $\varphi_{-1}(b)>0$, so that $f_{-1}(b)$ is defined. This implies $v_1=1$. 
Since $v_1$ is necessarily unbracketed, $i_1=1$ as well. The word $b'=f_{-1}(b)$ is formed by changing the leftmost 1 in $b$, 
namely $v_{i_1}$, into 2. This introduces a new bracketed 1,2-pair formed by $v_1=2$ and $v_{i_2}=1$. 
The subword of unbracketed letters in $b'$ now becomes 
\[
	v_{i_3}\ldots v_{i_p}v_{j_1}\ldots v_{j_q}
\]
so that $\varphi_1(f_{-1}(b))=p-2=\varphi_1(b)-2$ and $\varepsilon_1(f_{-1}(b))=q=\varepsilon_1(b)$. This establishes the 
first assertion.

Now, assume in addition that $p=\varphi_1(b)>2$. Using the sequence of unbracketed letters in $b$ as in the preceding 
paragraph, $f_1$ changes the rightmost unbracketed 1 in $b$, namely $v_{i_p}$, into 2. We still have $v_1$ to be 1 after
 the change, so that $f_{-1}(f_1(b))$ is defined and the leftmost 1 in $f_1(b)$, namely $v_1$, is changed into 2 under $f_{-1}$.
On the other hand, $f_1(f_{-1}(b))$ is defined precisely because $p>2$, and the rightmost unbracketed 1 in $f_{-1}(b)$, namely 
$v_{i_p}$, is changed into 2 under $f_1$. 
As the changes introduced in $b$ to form $f_{-1}(f_1(b))$ are the same as in those of $f_1(f_{-1}(b))$, we conclude that
$f_1(f_{-1}(b))=f_{-1}(f_1(b))$, proving the second assertion.
\item We assume $\varphi_1(b)=1$, so that~\eqref{equation.12 word} is of the form $v_{i_1}v_{j_1}\ldots v_{j_q}$
Furthermore, as $\varphi_{-1}(b)=1$, $f_{-1}(b)$ is defined and $v_1=1$. As $v_1$ is necessarily unbracketed, $i_1=1$ as 
well. Therefore, we see that $f_1(b)=f_{-1}(b)$, since the rightmost unbracketed 1 in $b$ and the leftmost 1 in $b$ are the 
same, namely $v_{i_1}=v_1$. 
\item We assume that $\varepsilon_{-1}(b)>0$, so that $e_{-1}(b)$ is defined. This implies $v_1=2$. 
However, since $e_{-1}(b)\neq e_1(b)$, $e_{-1}$ and $e_1$ must change a 2 in $b$ at different locations, so we have $j_1>1$.
Consequently $v_1$ is a bracketed 2 and hence must be paired with some $v_h=1$ where $h<i_1<j_1$ (in case $p=0$, 
$h<j_1$ still holds). The word $b'=e_{-1}(b)$ is obtained by changing the leftmost 2 in $b$, namely $v_1$, to 1. This 
introduces two new unbracketed 1's, namely, $v_1$ and $v_h$. The subword of unbracketed letters in $b'$ is now
\[
	v_1 v_h v_{i_1}\ldots v_{i_p}v_{j_1}\ldots v_{j_q}
\]
so that $\varepsilon_1(b)=q=\varepsilon_1(e_{-1}(b))$ and $\varphi_1(e_{-1}(b))=p+2=\varphi_1(b)+2$. This establishes the 
first two equalities.

Now, $e_1(e_{-1}(b))$ is the word formed by changing the leftmost unbracketed 2 in $b'=e_{-1}(b)$, namely $v_{j_1}$, to 1. 
On the other hand, using the subword of $v$ in $b$ containing unbracketed letters as described in the preceding paragraph, 
$e_1(b)$ changes the leftmost unbracketed 2 in $b$, namely $v_{j_1}$, into a 1. We still have $v_1=2$ and $v_h=1$ after the 
change, so that $e_{-1}(e_1(b))$ is defined, with the leftmost 2 in $e_1(b)$, namely $v_1$, being changed into 1 under
$e_{-1}$. As the changes introduced in $b$ to form $e_{-1}(e_1(b))$ are the same as in those of $e_1(e_{-1}(b))$, we 
conclude that $e_1(e_{-1}(b))=e_{-1}(e_1(b))$, thereby proving the final relation.
\end{enumerate}
\end{proof}

\begin{lemma}
\label{lemma.e2}
Let $b\in \mathcal{B}^{\otimes \ell}$. The following holds:
\begin{enumerate}
\item If $\varphi_2(b),\varphi_{-1}(b)>0$, we have $\varphi_2(b) = \varphi_2(f_{-1}(b))-1$, 
$\varepsilon_2(b) = \varepsilon_2(f_{-1}(b))$ and
\[
	f_2 f_{-1}(b) = f_{-1} f_2(b).
\]
\item If $\varphi_2(b)=0$ and $\varphi_{-1}(b)>0$, we have either
\begin{enumerate}
\item $\varphi_2(f_{-1}(b))=1$ and $\varepsilon_2(b) = \varepsilon_2(f_{-1}(b))$, or
\item $\varphi_2(f_{-1}(b))=0$ and $\varepsilon_2(b) = \varepsilon_2(f_{-1}(b))+1$.
\end{enumerate}
\item If $\varepsilon_2(b),\varepsilon_{-1}(b)>0$, we have either
\begin{enumerate}
\item $\varepsilon_2(e_{-1}(b))=\varepsilon_2(b)+1$, $\varphi_2(b)=\varphi_2(e_{-1}(b))=0$, or
\item $\varepsilon_2(e_{-1}(b))=\varepsilon_2(b)$, $\varphi_2(b) = \varphi_2(e_{-1}(b))+1$, and
\[
	e_{-1}e_2(b) = e_2 e_{-1}(b).
\]
\end{enumerate}
\end{enumerate}
\end{lemma}

\begin{proof}
We prove each part separately.
\begin{enumerate}
\item[(1)] Assume that $\varphi_2(b),\varphi_{-1}(b)>0$, so that $f_2(b)$ and $f_{-1}(b)$ are both nonzero.  
Let $b'=f_{-1}(b)$ and $b''=f_2(b)$.

By the signature rule, $\varphi_2(b)$ is the number of unbracketed $2$ entries in the $2,3$-bracketing of $b$.  Since 
$\varphi_2(b)>0$, there exists a leftmost unbracketed $2$, say $b_j$.  As in Remark \ref{rmk:combo-minus-1} 
$b'=f_{-1}(b)$ is formed by changing the leftmost $1$, say $b_i$, to $b'_i=2$, where $b_i$ is the leftmost of all 
$1$ and $2$ entries (so in particular $i<j$).  

If every $3$ left of $b_i$ is bracketed with a $2$ that is also to the left of $b_i$, then $b_i'$ is a new unbracketed $2$ 
in the $2,3$-bracketing of $b'$, so $\varphi_2(b')=\varphi_2(b)+1$.  Otherwise, assume there is a $3$ left of $b_i$ 
bracketed with a $2$ to the right of $b_i$, and let $b_{s_1}\cdots b_{s_r}b_{t_1}\cdots b_{t_r}=3^r2^r$ be the subsequence 
of all $3$ and $2$ entries bracketed with each other for which $s_j<i$ and $i<t_j$ for all $j$.   Then in $b'$, we have that 
$b_{s_r}'$ brackets with $b_i'$ rather than $b_{t_1}'$, and $b_{s_{r-1}}'$ brackets with $b_{t_1}'$, and so on, leaving 
$b_{t_r}'$ a new unbracketed $2$.  Thus we always have $\varphi_2(b')=\varphi_2(b)+1$.  Furthermore, since the 
number of unbracketed $3$ entries remains unchanged, we have $\varepsilon_2(b)=\varepsilon_2(f_{-1}(b))$.

For the commutativity relation, note that since $j>i$, so $b_j'=2$ is still the rightmost unbracketed $2$ in $b'$ and 
$b_i''=1$ is the leftmost $1$ or $2$ in $b''$.  Thus both $f_2(f_{-1}(b))$ and $f_{-1}(f_2(b))$ are formed by changing 
$b_i$ to $2$ and $b_j$ to $3$.  Hence $$f_2(f_{-1}(b))=f_{-1}(f_2(b))$$ as desired.

\item[(2)] Assume $\varphi_2(b)=0$ and $\varphi_{-1}(b)>0$, so that $b'=f_{-1}(b)$ is defined but $f_2(b)$ is not.  
Then there is an entry $b_i=1$ with no $1$ or $2$ left of it that changes to $2$ to form $b'$.  There are also no 
unbracketed $2$ entries in the $2,3$ bracketing.  

We consider two cases.  First, suppose that every $3$ to the left of $b_i$ in $b$ is bracketed with some $2$ to its right.  
Then in $b'$ with $b_i'=2$, the bracketed pairs for the entries $b_{s_i}'=3$ to the left of $b_i'$ shift left as in part (1) above, 
leaving a new unbracketed $2$ and exactly the same number of unbracketed $3$ entries.  Thus $\varphi_2(b')=1$ and 
$\varepsilon_2(b')=\varepsilon_2(b)$ in this case.  

If instead there is an unbracketed $3$ to the left of $b_i$, then this $3$ becomes bracketed with a $2$ (after the same 
shift in bracketed pairs) and we have $\varphi_2(b')=0$ and $\varepsilon_2(b')=\varepsilon_2(b)-1$, as desired.

\item[(3)]  Suppose $\varepsilon_2(b),\varepsilon_{-1}(b)>0$.  Then the leftmost $1$ or $2$ in $b$ is $b_i=2$ for some 
$i$, and $b':=e_{-1}(b)$ is formed by changing $b_i$ to $1$.  Since $e_2(b)$ is defined, there also exists a leftmost 
unbracketed $3$, say $b_j=3$.  

We consider two cases.  First suppose $\varphi_2(b)=0$, meaning that every $2$ is bracketed in the $2,3$-bracketing 
of $b$.  Then in particular $b_i$ is bracketed; let $b_{s_1}\cdots b_{s_r}b_ib_{t_1}\cdots b_{t_{r-1}}=3^r2^r$ be the 
subsequence consisting of all bracketed $3$'s ($b_{s_i}$) to the left of $b_i$ along with the entries they are bracketed 
with ($b_{t_{r-i}}$ where $t_{0}=i$).  Then after lowering $b_i$ to $1$ to form $b'$, we have that $b'_{s_i}$ brackets 
with $b'_{t_{r-i+1}}$ for $i\geqslant 2$, and $b'_{s_{1}}$ is an unbracketed $3$.  All other bracketed pairs are the same as in 
$b$, so there is only one more $3$ among the unbracketed letters.  It follows that $\varepsilon_{2}(b')=\varepsilon_2(b)+1$ 
and $\varphi_2(b')=\varphi_2(b)=0$.

For the second case, suppose $\varphi_2(b)>0$.  Then there is some unbracketed $2$ in $b$; let $b_k$ be the leftmost 
unbracketed $2$.  Note that $k\geqslant i$ because $b_i$ is the leftmost $2$, and note also that $k<j$ because $b_j$ is 
the leftmost unbracketed $3$.  Thus $i<j$.  

Now, lowering $b_i$ to $1$ to form $b'$ results in shifting the bracketing as in the cases above, which makes $b_k'$ 
be bracketed (and all other bracketings the same).  Thus there is one less unbracketed $2$ in $b'$ as $b$, and the 
same number of unbracketed $3$'s.  It follows that $\varepsilon_2(b')=\varepsilon_2(b)$ and $\varphi_2(b')=\varphi_2(b)-1$.
Furthermore, $b_j'$ is still the leftmost unbracketed $3$ in $b'$, and so both $e_{-1}e_2(b)$ and $e_{2}e_{-1}(b)$ are 
formed by changing $b_i$ to $1$ and $b_j$ to $2$.  The result follows.
\end{enumerate} 
\end{proof}

\subsection{Explicit description of $f_{-i}$ and $e_{-i}$}
\label{section.f-i and e-i}

In this section, we give explicit descriptions of $\varphi_{-i}(b)$, $\varepsilon_{-i}(b)$, $f_{-i} b$, and $e_{-i}b$
for $J$-highest-weight elements $b\in \mathcal{B}^{\otimes \ell}$ for certain $J \subseteq I_0$
(see Proposition~\ref{proposition.phi -i} and Theorems~\ref{theorem.f-i} and~\ref{theorem.e-i}).
We will need these results in Section~\ref{section.G(C)} when we characterize certain graphs on the type $A$ components
of the queer crystal.

\begin{lemma}
\label{lemma.e-i}
Let $i \in I_0$ and $b \in \mathcal{B}^{\otimes \ell}$ be $\{1,2,\ldots, i-1\}$-highest weight. If the first letter in the
$(i,i+1)$-subword of $b$ is $i+1$, then $\varepsilon_{-i}(b)=1$.
\end{lemma}

\begin{proof}
The statement is true for $i=1$ by Remark~\ref{rmk:combo-minus-1}. Now suppose that by induction on $i$
the statement of the lemma is true for $1,2,\ldots,i-1$. By Lemma~\ref{lemma.f recursion}, we have
$e_{-i} = s_{i-1} s_i e_{-(i-1)} s_i s_{i-1}$. Let $u=i+1$ be the leftmost $i+1$ in $b$ and $v=i$ be the leftmost $i$ in $b$. 
By assumption, $u$ appears to the left of $v$ and hence $v$ is bracketed in the $(i,i+1)$-bracketing. 
Since by assumption $b$ is $\{1,2,\ldots,i-1\}$-highest weight,
in the $(i-1,i)$-bracketing there are no unbracketed $i$ and $s_{i-1}$ raises all unbracketed $i-1$ to $i$.
In particular, all $i-1$ to the left of $v$ are raised to $i$ since $v$ is the leftmost $i$. In turn, $s_i$ acts on unbracketed
$i$ and $i+1$ in the $(i,i+1)$-bracketing. Since $v$ is bracketed and there are no $i-1$ to the left of $v$,
the first letter in the $(i-1,i)$-subword of $s_i s_{i-1}(b)$ is $i$. Also, $s_i s_{i-1}(b)$ is $\{1,2,\ldots,i-2\}$-highest weight.
Hence by induction $\varepsilon_{-(i-1)}(s_i s_{i-1}(b))=1$, which proves that $\varepsilon_{-i}(b)=1$.
\end{proof}

The next definition below will be used heavily throughout this section.

\begin{definition}
The \defn{initial $k$-sequence} of a word $b=b_1 \ldots b_\ell \in \mathcal{B}^{\otimes \ell}$, if it 
exists, is the sequence of letters $b_{p_k},b_{p_{k-1}},\ldots,b_{p_1}$, where $b_{p_k}$ is the leftmost $k$ and 
$b_{p_j}$ is the leftmost $j$ to the right of $b_{p_{j+1}}$ for all $1\leqslant j<k$.
\end{definition}

Let $i \in I_0$ and $b \in \mathcal{B}^{\otimes \ell}$ be $\{1,2,\ldots, i\}$-highest weight with $\wt(b)_{i+1}>0$.  
Then note that $b$ has an initial $(i+1)$-sequence, say $b_{p_{i+1}},b_{p_i},\ldots,b_{p_1}$.  Also let 
$b_{q_{i}},b_{q_{i-1}},\ldots,b_{q_1}$ be the initial $i$-sequence of $b$.  Note that $p_{i+1}< p_i<\cdots < p_1$ and 
$q_i<q_{i-1}<\cdots<q_1$ by the definition of initial sequence. Furthermore either $q_j=p_j$ or $q_j<p_{j+1}$ for 
all $1\leqslant j \leqslant i$.

\begin{proposition}
\label{proposition.phi -i}
Let $b \in \mathcal{B}^{\otimes \ell}$ be $\{1,2,\ldots, i\}$-highest weight for $i \in I_0$. Then:
\begin{enumerate}
\item[(a)]
$\varepsilon_{-i}(b)=1$ if and only if $\wt(b)_{i+1}>0$ and $p_j = q_j$ for at least
one $j\in \{1,2,\ldots,i\}$.
\item[(b)] 
$\varphi_{-i}(b)=1$ if and only if $\wt(b)_{i}>0$ and either $\wt(b)_{i+1}=0$ or $p_j \neq q_j$ for all $j\in \{1,2,\ldots,i\}$.
\end{enumerate}
\end{proposition}

\begin{example}
\label{example.phi -i}
Take $b=1331242312111$ and $i=3$. Then $p_4=6, p_3=8, p_2=10, p_1=11$ and
$q_3=2,q_2=5,q_1=9$. We indicate the chosen letters $p_j$ by underlines and 
$q_j$ by overlines: $b=1\overline{3}31\overline{2}\underline{4}2\underline{3}\overline{1}\underline{2}\underline{1}11$.
Since no letter has a both an overline and underline (meaning $p_j\neq q_j$ for all $j$), we have
$\varphi_{-3}(b)=1$.
\end{example}

\begin{proof}[Proof of Proposition~\ref{proposition.phi -i}]
Let us first prove claim (a) for $i=1$. If $\wt(b)_2=0$, then certainly $\varepsilon_{-1}(b)=0$ since by definition $e_{-1}$ 
changes a 2 into a 1. If $\wt(b)_2>0$, then $q_1$ is the position of the leftmost 1, $p_2$ is the position of the 
leftmost 2, and $p_1$ is the position of the first 1 after this 2. If $p_1 = q_1$, there is no 1
to the left of the leftmost 2. By definition in this case $\varepsilon_{-1}(b)=1$. If on the other hand
$q_1<p_2$, the leftmost 1 is before the leftmost 2 and hence $\varepsilon_{-1}(b)=0$. This proves the claim.

Now assume by induction that claim (a) is true for up to $i-1$. If $\wt(b)_{i+1}=0$, then $\varepsilon_{-i}(b)=0$
since $e_{-i}$ changes the weight by the simple root $\alpha_i$. Otherwise assume that $\wt(b)_{i+1}>0$.

If $p_i=q_i$, the first letter $i$ or $i+1$ is the $i+1$ in position $p_{i+1}<p_i=q_i$. Hence by 
Lemma~\ref{lemma.e-i} we have $\varepsilon_{-i}(b)=1$. 

If $q_i<p_i$ (and hence automatically $q_i<p_{i+1}$),
recall that by Lemma~\ref{lemma.f recursion} we have $e_{-i} = s_{i-1} s_i e_{-(i-1)} s_i s_{i-1}$. The operator
$s_{i-1}$ leaves the letter $i-1$ in positions $q_{i-1}$ and $p_{i-1}$ unchanged since these letters are bracketed with
$i$ in positions $q_i$ and $p_i$, respectively. All $i-1$ to the left of position $q_{i-1}$ are unbracketed and since
$b$ is $\{1,2,\ldots, i\}$-highest weight, $s_{i-1}$ changes all of these $i-1$ to $i$. In $s_{i-1}b$ there are possibly
new letters $i$ between positions $p_{i+1}$ and $p_i$; the $i+1$ in position $p_{i+1}$ brackets with the leftmost
of these in position $p_{i+1}<p'_i\leqslant p_i$. The operator $s_i$ on $s_{i-1}b$ changes all letters $i$ to the left
of position $p'_i$ to $i+1$. Hence $\wt(s_i s_{i-1}b)_i>0$, $s_i s_{i-1} b$ is $\{1,2,\ldots, i-1\}$-highest weight with 
sequences with respect to $i-1$ given by $p_i'>p_{i-1}>\cdots>p_1$ and $q_{i-1}>q_{i-2}>\cdots>q_1$.
Claim (a) now follows by induction on $i$.

If $b$ is $\{1,2,\ldots,i\}$-highest weight and $\wt(b)_i>0$, we must have $\varphi_{-i}(b) + \varepsilon_{-i}(b)=1$.
Hence $\varphi_{-i}(b)=1$ precisely when $\varepsilon_{-i}(b)=0$, proving (b).
\end{proof}

Recall that in a queer crystal $B$ an element $b\in B$ is \defn{highest-weight} if $e_i(b)=0$ for all $i\in I_0 \cup I_-$,
where $I_0 = \{1,2,\ldots,n\}$ and $I_- = \{-1,-2,\ldots,-n\}$.

\begin{proposition}\cite[Proposition 1.13]{Grantcharov.2014}
\label{proposition.hw}
Let $b\in \mathcal{B}^{\otimes \ell}$ be highest weight. Then $\wt(b)$ is a strict partition.
\end{proposition}

\begin{proof}
Let $b$ be highest weight and suppose that $\wt(b)_i = \wt(b)_{i+1}$ for some $i$, meaning that $b$ contains
the same number of letters $i$ and $i+1$. Since all letters $i$ and $i+1$ must be bracketed in the $(i,i+1)$-bracketing,
this means that the first letter in the $(i,i+1)$-subword of $b$ is the letter $i+1$. Then by Lemma~\ref{lemma.e-i},
$\varepsilon_{-i}(b)=1$, which means that $b$ is not highest weight. Hence $\wt(b)_i>\wt(b)_{i+1}$ for all $i$,
implying that $\wt(b)$ is a strict partition.
\end{proof}

Next, we provide an explicit description of $f_{-i}(b)$ for $i \in I_0$, when $b$ is $\{1,2,\ldots,i\}$-highest weight.
Recall that the sequence $b_{q_i}, b_{q_{i-1}},\ldots,b_{q_1}$ is the leftmost sequence of letters $i,i-1,\ldots,1$ from left 
to right. Set $r_1=q_1$ and recursively define $r_j<r_{j-1}$ for $1<j\leqslant i$ to be maximal such that $b_{r_j}=j$.
Note that by definition $q_j \leqslant r_j$. Let $1\leqslant k\leqslant i$ be maximal such that $q_k=r_k$.

\begin{theorem}
\label{theorem.f-i}
Let $b\in\mathcal{B}^{\otimes \ell}$ be $\{1,2,\ldots,i\}$-highest weight for $i \in I_0$ and $\varphi_{-i}(b)=1$
(see Proposition~\ref{proposition.phi -i}). Then $f_{-i}(b)$ is obtained from $b$ by changing 
$b_{q_j}=j$ to $j-1$ for $j=i,i-1,\ldots,k+1$ and $b_{r_j}=j$ to $j+1$ for $j=i,i-1,\ldots,k$.
\end{theorem}

\begin{example}
Let us continue Example~\ref{example.phi -i} with $b=1331242312111$ and $i=3$. We overline $b_{q_j}$
and underline $b_{r_j}$, so that $b=1\overline{3}\underline{3}1\overline{2}4\underline{2}3\underline{\overline{1}}2111$.
From this we read off $q_3=2,q_2=5,q_1=9$, $r_3=3, r_2=7, r_1=9$, $k=1$ and 
$f_{-3}(b) = 1241143322111$.

As another example, take $b=545423321211$ in the $\mathfrak{q}(6)$-crystal $\mathcal{B}^{\otimes 12}$ and $i=5$.
Again, we overline $b_{q_j}$ and underline $b_{r_j}$, so that 
$b=\overline{5}\overline{4}\underline{54}2\overline{3}\underline{3}\underline{\overline{2}}\underline{\overline{1}}211$.
This means that $q_5=1$, $q_4=2$, $q_3=6$, $q_2=8$, $q_1=9$, $r_5=3$, $r_4=4$, $r_3=7$, $r_2=8$, $r_1=9$,
$k=2$, and $f_{-5}(b) =436522431211$.
\end{example}

\begin{proof}[Proof of Theorem~\ref{theorem.f-i}]
We prove the claim by induction on $i$. For $i=1$, since by assumption $\varphi_{-1}(b)=1$, the first letter
in the subword of $b$ of letters in $\{1,2\}$ is a 1. This 1 is in position $q_1=r_1$ and changes to 2, which proves the claim.

Now assume that the claim is true for $f_{-1},\ldots,f_{-(i-1)}$. Recall that by Lemma~\ref{lemma.f recursion} we have
$f_{-i}=s_{i-1} s_i f_{-(i-1)}s_i s_{i-1}$. Let $b\in\mathcal{B}^{\otimes \ell}$ be $\{1,2,\ldots,i\}$-highest weight.
Applying $s_{i-1}$ to $b$ changes all unbracketed $i-1$ in the $(i-1,i)$-bracketing to $i$.  Subsequently applying 
$s_i$ changes all unbracketed $i$ in the $(i,i+1)$-bracketing to $i+1$.  
It is not hard to see that the resulting word is $\{1,\ldots,i-1\}$-highest weight, so we can apply the inductive hypothesis in 
order to apply $f_{-(i-1)}$.

In the notation for Proposition~\ref{proposition.phi -i}, we have either $\wt(b)_{i+1}=0$ or $q_i<p_{i+1}$ and $q_{i-1}<p_i$ 
since $\varphi_{-i}(b)=1$. In particular this means that if $p_{i+1}$ is defined and $p_{i+1}<q_{i-1}$, no letter $i$ lies between
$p_{i+1}$ and $q_{i-1}$ since otherwise $p_i<q_{i-1}$ contradicting the requirement $q_{i-1}<p_i$.
This implies that all $i-1$ and $i$ in the positions to the left of position $q_{i-1}$ become $i+1$ when applying $s_i s_{i-1}$. 
The letter $i-1$ in position $q_{i-1}$ remains $i-1$ under $s_i s_{i-1}$ since it is bracketed with an $i$. Denote the 
sequences for $f_{-(i-1)}$ in $s_i s_{i-1} b$ by $q_{i-1}',\ldots,q_1'$ and $r_{i-1}',\ldots,r_1'$ and call $k'$ the maximal 
index such that $q'_{k'}=r'_{k'}$. By the above arguments, we have $q_{i-1}'=q_{i-1}$.
We need to distinguish three cases given by $k=i,i-1$ and $k<i-1$.

\smallskip

\noindent \textbf{Case $k=i$:} The claim is that the $i$ in position $q_i$ changes to $i+1$.
Since $q_i=r_i$ for $k=i$, there is only one $i$ to the left of the $i-1$ in position $r_{i-1}$. Since $q_{i-1}\leqslant r_{i-1}$,
this implies that all $i-1$ between positions $q_{i-1}$ and $r_{i-1}$ (and including $r_{i-1}$) change to $i+1$ when applying 
$s_i s_{i-1}$. This means that $k'=i-1$ and by induction $f_{-(i-1)}$ changes the $i-1$ in position $q_{i-1}$ to $i$. Hence 
under $s_{i-1} s_i$, the letter in position $q_i$ remains an $i+1$ and all other letters $i+1$ and $i$ return to their original 
value. This proves the claim.

\smallskip

\noindent \textbf{Case $k=i-1$:} In this case, we have at least two $i$ to the left of position $q_{i-1}=r_{i-1}$ and there is
no $i-1$ between positions $q_{i-1}$ and $r_{i-2} \geqslant q_{i-2}$. Since $s_i s_{i-1}$ lifts all $i$ to the left of position 
$q_{i-1}$ to $i+1$, but leaves the $i-1$ in position $q_{i-1}$ and possible $i-2$ in positions $q_{i-2}$ and $r_{i-2}$, we 
have $k'=i-1$. Hence by induction $f_{-(i-1)}$ changes the $i-1$ in position $q'_{i-1}=q_{i-1}$ to $i$. When applying 
$s_{i-1} s_i$ to $f_{-(i-1)} s_i s_{i-1} b$, the $i+1$ in position $r_i$ remains an $i+1$ since it is now bracketed with the $i$ in
position $q_{i-1}$ or an $i$ to its left. In addition, the $i+1$ in position $q_i$ becomes an $i-1$ since the $i$ in position 
$q_{i-1}$ is now bracketed with the previous bracketing partner of letter in position $q_i$ in $b$, causing it to drop to $i-1$. 
This proves the claim for $k=i-1$.

\smallskip

\noindent \textbf{Case $k<i-1$:} In this case $q_i<r_i$ and $q_{i-1}<r_{i-1}$, so that there are at least two $i$ to
the left of position $r_{i-1}$ and at least two $i-1$ between positions 
$q_i$ and $r_{i-2}\geqslant q_{i-2}$. By the arguments above, all
$i$ to the left of position $q_{i-1}$ become $i+1$ under $s_i s_{i-1}$, the letter $i-1$ in position $q_{i-1}$ remains 
$i-1$ and $q'_{i-1}=q_{i-1}<r_{i-1}' \leqslant r_{i-1}$. Also, since $s_i s_{i-1}$ leaves all letters $i-2$ and smaller untouched,
we have $q'_j=q_j$ and $r'_j=r_j$ for $1\leqslant j <i-1$.
Hence by induction $f_{-(i-1)}$ changes the letter in position $q_{i-1}=q'_{i-1}$ to $i-2$ and the letter in
position $r'_{i-1}$ to $i$, in addition to the letters in positions $q_j,r_j$ for $j<i-1$. Next applying $s_{i-1} s_i$
changes the letter in position $r_{i-1}$ to $i$ since it is now bracketed with the $i-1$ in position $r_{i-2}$.
The letters $i+1$ in positions $r_{i-1}'<p < r_{i-1}$ are changed back to $i-1$ since they are not bracketed.
If $r_{i-1}'<r_{i-1}$, then the letter $i$ in position $r'_{i-1}$ changes to $i-1$ since it is also not bracketed.
The letter in position $q_{i-1}=q'_{i-1}$ remains $i-2$. The letter $i+1$ in position $r_i$ is bracketed 
with the $i$ in position $r'_{i-1}$ in $f_{-(i-1)} s_i s_{i-1}b$ and hence remains $i+1$ in $s_{i-1} s_i f_{-(i-1)} s_i s_{i-1}b$.
The letters $i+1$ between positions $q_i$ and $r_i$ in $f_{-(i-1)} s_i s_{i-1}b$ return to their original value $i$ 
under $s_{i-1} s_i$ since they are bracketed with $i-1$ to the right. The letter in position $q_i$ lost its bracketing
partner since the $i-1$ in position $q_{i-1}$ became $i-2$. Hence the letter in position $q_i$ becomes $i-1$,
proving the claim.
\end{proof}

\begin{corollary}
\label{corollary.f-i}
Let $b\in\mathcal{B}^{\otimes \ell}$ be $J$-highest weight for $\{1,2,\ldots,i\} \subseteq J \subseteq I_0$
and $\varphi_{-i}(b)=1$ for some $i \in I_0$. Then:
\begin{enumerate}
\item
Either $f_{-i}(b) = f_i(b)$ or $f_{-i}(b)$ is $J$-highest weight.
\item
$f_{-i}(b)$ is $I_0$-highest weight only if $b=f_{i+1} f_{i+2} \cdots f_{h-1} u$ for some $h>i$ and $u$ a
$I_0$-highest weight element.
\end{enumerate}
\end{corollary}

\begin{proof}
We begin by proving (1).
By Theorem~\ref{theorem.f-i}, in $f_{-i}(b)$ the letters $b_{q_j}$ are changed from $j$ to $j-1$ 
for $j=i,i-1,\ldots,k+1$ and $b_{r_j}$ are changed from $j$ to $j+1$ for $j=i,i-1,\ldots,k$. Hence $f_{-i}(b)$ is
not $J$-highest weight if and only if either there is an $i+1$ to the left of position $q_i$ that is no longer bracketed 
with an $i$ or the letter $k+1$ in position $r_k$ is no longer bracketed with a $k$.

First assume that $k<i$. Since $k$ is maximal such that
$q_k=r_k$, there must be at least two $k+1$ to the left of position $q_k$ in $b$, one in position $q_{k+1}$ and
one in position $r_{k+1}$. Since $b$ is $J$-highest weight, both of these $k+1$ must be bracketed with a $k$
to their right in $b$, which implies that there is a $k$ to the right of position $q_k$ that is bracketed with the $k+1$ in 
position $q_{k+1}$ in $b$. In $f_{-i}(b)$, the letter $k+1$ in position $q_{k+1}$ changes to $k$, and hence the 
new $k+1$ in position $q_k=r_k$  is bracketed with the $k$ to its right.

Since by assumption $\varphi_{-i}(b)=1$, we have by Proposition~\ref{proposition.phi -i} that either $\wt(b)_{i+1}=0$
(in which case there cannot be an $i+1$ to the left of position $q_i$ in $b$) or $p_j\neq q_j$ for all
$j \in \{1,2,\ldots,i\}$. The condition $p_i \neq q_i$ implies that $q_i < p_{i+1}$, so that there cannot be a letter
$i+1$ to the left of position $q_i$. This proves that $f_{-i}(b)$ is $J$-highest weight when $k<i$.

Next assume that $k=i$. In this case $f_{-i}(b)$ differs from $f_i(b)$ by changing the letter $i$ in position
$q_i$ to $i+1$. If there is a letter $i$ to the right of position $q_i$ that is not bracketed with a letter $i+1$,
then the new $i+1$ in position $q_i$ will bracket with this $i$ in $f_{-i}(b)$ (or to the left of this $i$) and hence $f_{-i}(b)$ is
$J$-highest weight. Otherwise, there is no letter $i$ to the right of position $q_i$ in $b$ that is not
bracketed with an $i+1$ and therefore $f_i(b)=f_{-i}(b)$. This proves claim (1).

The above arguments also show that $f_{-i}(b)$ can only be $I_0$-highest weight if either $b$ is $I_0$-highest weight
or $\varepsilon_j(b)=0$ for $j\in I_0 \setminus \{i+1\}$ and the new letter $i+1$ in position $r_i$ in $f_{-i}(b)$
is bracketed with a letter $i+2$ in $b$. Such a $b$ is precisely of the form $b=f_{i+1} f_{i+2} \cdots f_{h-1} u$
proving claim (2).
\end{proof}

Next, we describe $e_{-i}$ on a $\{1,2,\ldots,i\}$-highest weight element $b$.  We again use the initial $(i+1)$-sequence 
$b_{p_{i+1}},b_{p_{i}},\ldots,b_{p_1}$ in $b$.

We also need the notion of \defn{cyclically scanning leftwards} for a letter $t$ starting at an entry $b_j$.
By this we mean choosing the rightmost $t$ to the left of $b_j$, if it exists, or else the rightmost $t$ in the entire word 
(i.e., ``wrapping around'' the edge of the word).

We define the \defn{$k$-bracketed entries} of a word $b$ as follows.  Every $k$ in $b$ is $k$-bracketed, and for 
$j=k-1,k-2,\ldots,1$, we recursively determine which $j$'s in $b$ are $k$-bracketed by considering the subword of 
only the $k$-bracketed $(j+1)$'s and all $j$'s, and performing an ordinary crystal bracketing on this subword.  The $j$'s 
that are bracketed in this process are the $k$-bracketed $j$'s.

\begin{example}
In the word $$142334122311322111,$$ to obtain the $4$-bracketed letters we first mark all $4$'s as $4$-bracketed: 
$$1\mathbf{4}233\mathbf{4}122311322111$$
and then bracket these with $3$'s and mark the bracketed $3$'s as being $4$-bracketed:
$$1\mathbf{4}2\mathbf{3}3\mathbf{4}122\mathbf{3}11322111.$$
We then consider only the boldface $3$'s and all the $2$'s and bracket them to obtain the $4$-bracketed $2$'s:
$$1\mathbf{4}2\mathbf{3}3\mathbf{4}1\mathbf{2}2\mathbf{3}113\mathbf{2}2111$$
Finally we bracket these boldface $2$'s with the $1$'s to obtain: 
$$1\mathbf{4}2\mathbf{3}3\mathbf{4}1\mathbf{2}2\mathbf{3}\mathbf{1}13\mathbf{2}2\mathbf{1}11$$
The boldface letters above are precisely the $4$-bracketed letters in this word.
\end{example}

We now have the tools to describe the application of $e_{-i}$ to an $\{1,2,\ldots,i\}$-highest weight word.

\begin{theorem}\label{theorem.e-i}
Let $b\in \mathcal{B}^{\otimes \ell}$ be $\{1,2,\ldots,i\}$-highest weight for $i\in I_0$ and $\varepsilon_{-i}(b)=1$ (see 
Proposition \ref{proposition.phi -i}).  Let $b_{p_{i+1}},\ldots,b_{p_1}$ be the initial $(i+1)$-sequence of $b$.  Then $e_{-i}(b)$ 
is obtained from $b$ by the following algorithm:
  \begin{itemize}
  \item Change $b_{p_j}$ from $j$ to $j-1$ for 
$j=i+1,i,\ldots,3,2$ to form a word $c^{(1)}$.
  \item Cyclically scan left in $c^{(1)}$ starting just to the left of position $p_1$ for a $1$ that is not $i$-bracketed in $c^{(1)}$. 
   Change that $1$ to $2$ to form a word $c^{(2)}$.  In $c^{(2)}$, continue cyclically scanning from just to the left of the 
   previously changed entry for a $2$ that is not $i$-bracketed in $c^{(2)}$, and change it to $3$.  Continue this process 
   until an $i-1$ changes into an $i$; the resulting word $c^{(i)}$ is $e_{-i}(b)$.
  \end{itemize}
\end{theorem}

\begin{proof}
  We will prove this by induction on $i$.  For $i=1$ the algorithm simply changes the leftmost $2$ to a $1$ as required, 
  since the second step is vacuous in this case.  
  
  Assume the statement is true for $i$ and let $b\in \mathcal{B}^{\otimes \ell}$ be $\{1,2,\ldots,i+1\}$-highest weight.  
  Recall that $e_{-(i+1)}=s_is_{i+1}e_{-i}s_{i+1}s_i$ by Lemma~\ref{lemma.f recursion}. We will analyze each step of 
  applying $s_is_{i+1}e_{-i}s_{i+1}s_i$ to $b$ and show that it matches the desired algorithm.
  
  Let $b_{p_{i+2}},b_{p_{i+1}},b_{p_i},\ldots,b_{p_2},b_{p_1}$ be the initial $(i+2)$-sequence of $b$.  Since $e_i b=0$, 
  applying $s_i$ to $b$ simply changes all unbracketed $i$ entries in the $(i,i+1)$-pairing to $i+1$. Note that $b_{p_i}$ 
  itself must be bracketed with an $i+1$ in $b$, for if it is not then $b_{p_{i+1}}$ is paired with an earlier $i$ to its right, 
  contradicting the definition of $b_{p_i}$.  Thus $b_{p_i}$ is still $i$ in $s_i b$.  Note also that $s_i b$ 
  still satisfies $e_{i+1} s_i b=0$.
  
  Let $b'=s_{i+1}s_i b$.  Note that any $i+1$ to the left of $b_{p_{i+2}}$ in $s_i b$ is not bracketed with an $i+2$ since 
  $b_{p_{i+2}}$ is the leftmost $i+2$.  Thus every $i+1$ left of $b_{p_{i+2}}$ (including those $i$'s that changed to $i+1$ 
  from $b$) changes to $i+2$ to form $b'$, along with any other unpaired $i+1$.  Let $b_{t_{i+1}}$ be the leftmost $i+1$ 
  between $b_{p_{i+2}}$ and $b_{p_{i+1}}$ in $s_{i} b$.  Then $b_{t_{i+1}}$ is either equal to $b_{p_{i+1}}$ or was an 
  $i$ in $b$.  Furthermore, $b_{t_{i+1}}$ is still $i+1$ in $b'=s_{i+1}s_i b$ since it must be paired with either $b_{p_{i+2}}$ 
  itself or some $i+2$ to the right of $b_{p_{i+2}}$.
  
  Now consider $e_{-i} b'$.  By the induction hypothesis, this can be computed by first lowering the entries of the initial 
  $(i+1)$-sequence $b'_{p'_{i+1}},b'_{p'_{i}},\ldots,b'_{p'_1}$ appropriately to form a word $c'^{(1)}$, then cyclically raising 
  some non-$i$-bracketed entries $1,2,3,\ldots,i-1$ in order to form words $c'^{(2)},\ldots,c'^{(i)}$.  We will show that 
  $p'_j=p_j$ for $j\leqslant i$, and that the same entries $1,2,\ldots,i-1$ are changed as would be changed in the $e_{-(i+1)}$ 
  algorithm applied to $b$.
  
  For the first claim, it suffices to show that $p'_i=p_i$.  Note that $b'_{p'_{i+1}}$ may be to the left of $b_{p_{i+1}}$, but it is 
  to the right of $b_{p_{i+2}}$ by the above analysis.  If $p'_{i+1}=p_{i+1}$ we are done, so suppose $p_{i+2}<p'_{i+1}<p_{i+1}$.
  Assume by contradiction that there is an entry $b'_a=i$ between positions $p'_{i+1}$ and $p_i$ in $b'$.  Then we further 
  have $p'_{i+1}<a<p_{i+1}$ by the definition of $b_{p_i}$ and $b'$.  It follows that $b_a$ is an $i$ in $b$ that is bracketed 
  with an $i+1$, since applying $s_i$ kept it an $i$.  But then by the definition of $p_{i+1}$, the entry $b_c=i+1$ that brackets 
  with $b_a$ in $b$ is to the left of position $p_{i+2}$.  Thus $b_{p'_{i+1}}$ itself was a bracketed $i$ in $b$, a contradiction. 
  Thus $p'_i=p_i$.
  
  Let $c^{(j)}$ be the word in the definition of $e_{-(i+1)}$ acting on $b$ and $c'^{(j)}$ the word in the definition 
  of $e_{-i}$ on $b'$. Similarly, let $t_j$ (resp. $t'_j$) be the position of the chosen $j$ in $c^{(j)}$ (resp. $c'^{(j)}$)
  that is raised to $j+1$. We now wish to show that, for any $j\leqslant i-1$, we have $t_j'=t_j$.  
  
  We first show this for $j=1$.  Note that since $p_2=p_2'$ (assuming $i\geqslant 2$, since otherwise we are done) the 
  same entries are equal to $1$ in both $c=c^{(1)}$ and $c'=c'^{(1)}$.  Moreover, $p_1=p_1'$, so we start searching cyclically 
  left for a $1$ in the same position in both.  It therefore suffices to show that an entry $c_x=1$ is $(i+1)$-bracketed in $c$ if 
  and only if $c_x'=1$ is $i$-bracketed in $c'$.  Note that the $i$'s in $c$ that are bracketed with $i+1$'s are precisely either:
  \begin{itemize}
     \item $c_{p'_{i+1}}$, or
     \item an $i$ that was bracketed with an $i+1$ in $b$.
  \end{itemize}
  But since $c'$ is formed by applying $s_i$ to $b$ (which changes all unbracketed $i$'s to $i+1$'s), then $s_{i+1}$ 
  (which does not change any $i$'s), then lowering certain entries, where $b_{p'_{i+1}}$ is the only one that becomes a 
  new $i$, the above characterization gives precisely \textit{all} $i$'s in $c'$.  Since the $1,2,\ldots,i-1$ entries are the same 
  in both $c$ and $c'$, it follows that an entry is $(i+1)$-bracketed in $c$ if and only if it is $i$-bracketed in $c'$.
  
  It now follows that $t_1=t_1'$, and inductively we can conclude that $t_j=t_j'$ for all $j\leqslant i-1$.  Thus if we apply 
  $s_{i}s_{i+1}$ to $c'^{(i)}$ to obtain $e_{-(i+1)}b$, the entries less than or equal to $i-1$ match those of $c^{(i+1)}$, the 
  result of the algorithm applied to $b$.  Furthermore, since $s_i$, $s_{i+1}$, and $e_{-i}$ only change letters less than 
  or equal to $i+2$, the entries larger than $i+2$ also match.
  
  It remains to consider the entries equal to $i$, $i+1$, and $i+2$.  For $i+2$, the application of $s_{i+1}$ to $s_i b$ 
  changes all unbracketed $i+1$ entries in $s_i b$ to $i+2$, and $e_{-i}$ changes the single entry $b'_{p'_{i+1}}=i+1$ to 
  $i$ and otherwise does not affect the $i+1$ or $i+2$ entries.  In the $(i+1,i+2)$-bracketing in $b'$, $b'_{p_{i+2}}$ is the 
  leftmost bracketed $i+2$, and $b'_{p'_{i+1}}$ is the first $i+1$ after it, so removing $b'_{p'_{i+1}}$ from the 
  $(i+1,i+2)$-subword leaves the $i+2$ in position $p_{i+2}$ unbracketed, with all other bracketed $(i+2)$'s remaining 
  bracketed.  It follows that applying $s_{i+1}$ to $e_{-i} s_{i+1}s_i b$ lowers the $i+2$ in position $p_{i+2}$ to $i+1$, 
  along with any $i+2$ that was raised in the first $s_{i+1}$ step.  Therefore, the $i+2$ entries in $s_{i+1}e_{-i} b'$, and 
  hence in $s_is_{i+1}e_{-i}b'=e_{-(i+1)}b$, match those in the output of the algorithm.
  
  Finally, we consider the $(i,i+1)$-subwords of the words in question.  We first analyze how the $(i,i+1)$-subword of 
  $w:=s_i b$ differs from that of $w':=s_{i+1}e_{-i}s_{i+1}s_i b$.  By inspecting the above analysis, we see that $w'$ 
  differs from $w$ in the following four ways:
  \begin{itemize}
    \item $w'_{p_{i+2}}=i+1$ is a new $i+1$ in the $(i,i+1)$-subword in $w'$ whereas $w_{p_{i+2}}=i+2$ was not in 
    the subword in $w$.
    \item $w'_{p'_{i+1}}=i$ whereas $w_{p'_{(i+1)}}=i+1$.
    \item $w'_{p_i}=i-1$ is no longer in the subword whereas $w_{p_i}=i$ was an $i$ in the subword.
    \item $w'_{t_{i-1}}=i$ is a new $i$ in the subword, whereas $w_{t_{i-1}}=i-1$.  
  \end{itemize}
  Note that the last two items above may coincide and cancel each other out if $t_{i-1}=p_i$.
  
  We now apply $s_i$ to both subwords, and analyze how $s_i w'=e_{-(i+1)} b$ differs from $s_i w=b$ in the 
  $(i,i+1)$-subword.  In particular, we will show it is the same as how $c^{(i+1)}$ differs from $b$.  Note that the 
  $(i,i+1)$-subword in $c^{(i+1)}$ is formed from that of $b$ by making the following changes:
  \begin{itemize}
    \item A new $i+1$ is inserted in position $p_{i+2}$ ($b_{p_{i+2}}=i+2$ whereas $c^{(i+1)}_{p_{i+2}}=i+1$).
    \item The $i+1$ in position $p_{i+1}$ is lowered to $i$.
    \item The $i$ in position $p_i$ is removed.
    \item An $i$ is inserted in position $t_{i-1}$.
    \item In the current subword, look for the first unbracketed $i$ cyclically left of position $t_{i-1}$; call this position 
    $t_i$ and change this $i$ to $i+1$.
  \end{itemize}
  
  First, note that there are no $i+1$ entries between $w'_{p_{i+2}}=i+1$ and $w'_{p'_{i+1}}=i$ in $w'$, for if there were, 
  this would contradict the definition of $b_{p_{i+1}}$.  It follows that $w'_{p_{i+2}}=i+1$ is bracketed with an $i$ to its right 
  in $w'$, so in $s_i w'=e_{-(i+1)} b$, the entry in position $p_{i+2}$ remains $i+1$.  So this is one position in which it differs 
  from $b$, since $b_{p_{i+2}}=i+2$, so it matches $c^{(i+1)}$ in this position.
  
  Note also that in $w$, all $i$'s are bracketed with $(i+1)$'s.  Applying $s_i$ to $w$ simply changes the unbracketed 
  $i+1$'s back to $i$'s to form $b$.
  We now consider two cases.
  
  \noindent
  \textbf{Case 1:} Suppose $p_{i+1}'\neq p_{i+1}$.
    
  We know that $s_i w$ and $s_i w'$ match $b$ and $c^{(i+1)}$, respectively, in position $p_{i+2}$ by the above analysis.  
  For position $p'_{i+1}$, note that it is an unbracketed $i+1$ in $w$, so it changes to $i$ in $s_i w$.  It is a bracketed $i$ 
  in $w'$ since it was the first unbracketed $i+1$ to the right of position $p_{i+1}$ in $w$, so it stays $i$ in $s_i w'$.  
  Thus they are both equal to $i$ in the results, matching $b$ and $c^{(i+1)}$, which do not differ in this entry.
    
   We now wish to show that the $i+1$ in position $p_{i+1}$ is unbracketed in $w'$ unless it is bracketed via the insertion 
   of the $i$ in position $t_{i-1}$.  In other words, if we make all the changes that define $w'$ from $w$ besides the $i$ 
   in position $t_{i-1}$, we claim that position $p_{i+1}$ is an unbracketed $i+1$.  Indeed, before removing $i$ in position 
   $p_i$, this $i+1$ in position $p_{i+1}$ is the leftmost $i+1$ that is bracketed with an entry weakly right of position $p_i$, 
   since the position $p_{i+2}$ entry is bracketed with some $i$ weakly left of position $p'_{i+1}$.  It follows that removing 
   the $i$ in position $p_i$ leaves $b_{p_{i+1}}$ unbracketed, and otherwise all other $i+1$'s are bracketed if and only if 
   they are bracketed in $w$.  
   
   Furthermore, the combination of lowering both $p_{i+2}$ and $p'_{i+1}$ to $i+1$ and $i$ and removing the $i$ in 
   position $p_i$ leaves all $i$'s still bracketed, as they are in $w$.
   
   Finally, when we put back the new $i$ in position $t_{i-1}$ to form $w'$, there are two subcases: first suppose inserting 
   this $i$ makes some unbracketed $i+1$ to its left become bracketed.  Then by the above analysis, this must have been 
   the position of the first unbracketed $i$ in $c^{(i)}$  to the left of $t_{i-1}$, and this is position $t_{i}$, which remains $i+1$ 
   in $s_i w'$.  Applying $s_i$ to $w'$ then turns the remaining unbracketed $i+1$ entries back to $i$ and matches 
   $c^{(i+1)}$.  Otherwise, if inserting the $i$ in position $t_{i-1}$ does not bracket any $i+1$ to the left, it creates an 
   unbracketed $i$ in the word, and so the rightmost unbracketed $i+1$ also will not change under applying $s_i$ to $w'$. 
   This corresponds to the first unbracketed $i$ cyclically left of position $t_{i-1}$ in $c^{(i)}$, and we are done as before.
   
   \noindent
   \textbf{Case 2:}  Suppose $p_{i+1}'=p_{i+1}$.  
  
  In this case, the analysis matches the above except for the following steps: first, since position $p_{i+1}$ contains a 
  bracketed $i+1$ in $w$, lowering it to $i$ may make some $i$ to its right become unbracketed.  (The new $i$ in 
  position $p_{i+1}$ itself is bracketed due to the new $i+1$ in position $p_{i+2}$ as before.)  
  
  Then, removing the $i$ in position $p_i$ will make all $i$'s bracketed once again, since $b_{p_i}$ was the first $i$ to 
  the right of position $p_{i+1}$ in $b$ and hence in $w$.  So once again, at the step before inserting $t_{i-1}$, all 
  $i$'s are bracketed, and an $i+1$ in that matches one in $w$ is bracketed if and only if it is bracketed in the modified
  word.  Thus inserting $t_{i-1}$ has the same effect as above, and we are done.
\end{proof}

We now show that the output of $e_{-i}$ on a $\{1,2,\ldots,i\}$-highest weight element is itself $\{1,2,\ldots,i\}$-highest 
weight if and only if there is no ``cycling around the edge'' in the cycling step of Theorem~\ref{theorem.e-i}.

\begin{proposition}\label{proposition.e-output-hw}
  Let $b\in \mathcal{B}^{\otimes \ell}$ be $\{1,2,\ldots,i\}$-highest weight for $i\in I_0$, with $\varepsilon_{-i}(b)=1$.  
  Let $t_1,\ldots,t_{i-1}$ be the positions of the $1,2,\ldots,i-1$ that change to $2,3,\ldots,i$ respectively in the second step 
  of the computation of $e_{-i}(b)$ (see Theorem \ref{theorem.e-i}). Then $e_{-i}(b)$ is 
  $\{1,2,\ldots,i\}$-highest weight if and only if $t_{i-1}<t_{i-2}<\cdots<t_{1}$.  
\end{proposition}
\begin{proof}
  First, suppose that it is not the case that $t_{i-1}<t_{i-2}<\cdots<t_{1}$; let $1\leqslant k<i$ be the smallest index for 
  which $t_{k-1}\leqslant t_{k}$, where $t_0=p_1$.  Then in the algorithm for computing $e_{-i}(b)$, after changing a $k-1$ 
  to $k$ in position 
  $t_{k-1}$, we search cyclically left for a $k$ that is not $i$-bracketed to find position $t_{k}$.  Since $t_{k-1}\leqslant t_{k}$, 
  we cycle around the end of the word, so $t_k$ is the position of the rightmost $k$ that is not $i$-bracketed.
  
  Any $k$ to the right of $t_k$ is $i$-bracketed, and we claim that the $k+1$'s that they bracket with in the $i$-bracketing are 
  all to the right of position $t_k$ as well.  Indeed, if one such $k+1$ was to the left of $t_k$ then it should bracket with the 
  $k$ in position $t_k$ instead, a contradiction.  Thus the suffix starting at position $t_{k}+1$ has at least as many $k+1$'s 
  as $k$'s.
  
  In particular, just after changing each $b_{p_r}$ to $r-1$ in the first step of the algorithm, the resulting word $c$ is still 
  highest weight.  It follows that, just after raising $t_{k-1}$ to $k$, the resulting word is still $\{k\}$-highest weight.  It follows 
  that the suffix starting at position $t_{k}+1$ at this step has exactly as many $k+1$'s as $k$'s.
  
  Now, if $t_{k+1}< t_k$, changing $t_k$ to $k+1$ and then changing $t_{k+1}$ to $k+2$ leaves the suffix starting at 
  $t_k$ being not $\{k\}$-highest weight in the final word.  Thus we are done in this case.  
  
  Otherwise, suppose $t_{k+1}$ also cycles, so that $t_{k+1}\geqslant t_k$ and $t_{k+1}$ is the new position of the rightmost 
  $k+1$ that is not $i$-bracketed after changing $t_k$ to $k+1$.  Changing $t_{k+1}$ to $k+2$ could potentially make the 
  word $\{k\}$-highest weight again. In fact, suppose for contradiction that, just after 
  changing $t_{k-1}$ to $k$, there were a $k+1$ between position $t_{k-1}$ and $t_k$ that makes its suffix not 
  $\{k\}$-highest weight.  Then some entry $k+1$ in position $p<t_k$ brackets with the $k$ in position $t_k$, and since 
  position $t_k$ is not $i$-bracketed, this $k+1$ is not $i$-bracketed either.    Thus after changing $t_k$ to $k+1$,  the 
  $k+1$ in position $p$ is still not $i$-bracketed and it would be picked up in the search for $t_{k+1}$, a contradiction to 
  the assumption that $t_{k+1}\geqslant t_k$.
  
  We now, however, can repeat the argument with $t_{k+1}$ and the $(k+1,k+2)$-subword, and so on until we either reach 
  the last step or a non-cycling step, say with index $\ell$. At this point we conclude that the final word $e_{-i}(b)$ is not 
  $\{\ell\}$-highest weight.
  
  It follows that if $t_{k-1}\leqslant t_{k}$ for some $k$, then $e_{-i}(b)$ is not $\{1,2,\ldots,i\}$-highest weight.

  For the converse, we wish to show that if $t_{i-1}<t_{i-2}<\cdots<t_1$ then $e_{-i}(b)$ remains highest weight.  
  Notice that by construction we must have $t_{k-1}\leqslant p_k$ for all $k\leqslant i$.  
  
  We first show that the $(1,2)$-subword remains highest weight in $e_{-i}(b)$ if $t_2<t_1$.  If $i=1$, then the first $2$ 
  simply changes to a $1$ and so it is still $\{1\}$-highest weight.  So suppose $i\geqslant 2$.
  
  The changes that affect the $(1,2)$-subword are that $b_{p_3}$ changes from $3$ to $2$, $b_{p_2}$ changes from $2$ 
  to $1$, $b_{t_1}$ changes from $1$ to $2$, and (if $i\geqslant 3$) $b_{t_2}$ changes from $2$ to $3$.  Note that after the 
  first two of these changes, any suffix of the word starting between positions $p_3$ and $p_2$ has at least two more $1$'s 
  than $2$'s (due to the change in $b_{p_2}$ starting from a highest weight word) and any suffix starting weakly before position 
  $p_3$ has at least one more $1$ than $2$.  
  
  If $i=2$, $b_{t_1}$ is an unbracketed $1$, so the suffixes before it must in fact have at least two more $1$'s than $2$'s 
  even if $t_1<p_3$.  Thus changing $b_{t_1}$ to $2$ leaves the word highest weight, and we are done in this case.
  
  If $i\geqslant 3$, $b_{t_1}$ is a $1$ that is not $i$-bracketed to the left of $b_{p_2}$, and $b_{t_2}$ is the first $2$ that is not 
  $i$-bracketed to the left of $t_1$ (and necessarily to the left of $b_{p_3}$).  It follows that, after changing them to $2$ and 
  $3$ respectively, the suffixes all have at least as many $1$'s as $2$'s except possibly those starting between position 
  $t_2$ and $t_1$.  Assume to the contrary that there is a suffix with more $2$'s than $1$'s starting between $t_2$ and 
  $t_1$; the rightmost such starts at another entry $b_a=2$ between $t_2$ and $t_1$, and this $2$ must be $i$-bracketed 
  by the definition of $t_2$.  But then since $b_{t_1}$ is not $i$-bracketed, $b_a$ must be bracketed with a $1$ between 
  $b_a$ and $b_{t_1}$; hence the suffix starting at $b_a$ cannot have a higher difference between $2$'s and $1$'s than 
  the suffix starting at $b_{t_1}$ after its change, a contradiction. It follows that the $(1,2)$-subword remains highest weight.
  
  Now consider the $(k,k+1)$-subword for some $k\leqslant i-1$.  This is changed by $b_{p_{k+2}}, b_{p_{k+1}},b_{p_{k}}$ 
  changing from $k+2$ to $k+1$, $k+1$ to $k$, and $k$ to $k-1$ respectively, and then $b_{t_{k-1}}, b_{t_{k}},b_{t_{k+1}}$ 
  changing from $k-1$ to $k$, $k$ to $k+1$, $k+1$ to $k+2$ respectively.  
  
  If we first change $b_{p_k}$ to $k-1$, then we have removed a $k$ from the subword, but since there are no $k$ entries 
  between $b_{p_{k+1}}$ and $b_{p_k}$, the rightmost suffix that may become not highest weight for $k$ starts at 
  $b_{p_{k+1}}$ itself.  Thus changing $b_{p_{k+1}}$ from $k+1$ to $k$ afterwards keeps the $(k,k+1)$-subword being 
  $\{k\}$-highest weight, and in fact any suffix starting to the left of $b_{p_{k+1}}$ at this point has at least one more $k$ 
  than $k+1$.  Finally if we change $b_{p_{k+2}}$ to $k+1$, this adds a single $k+1$ to any suffix starting left of this position, 
  so again the word remains $\{k\}$-highest weight.  Next, we change $b_{t_{k-1}}$ from $k-1$ to $k$, which means any 
  suffix starting left of $t_{k-1}$ has at least one more $k$ than $k+1$.  The argument for what happens after changing 
  $t_{k}$ and $t_{k+1}$ now is identical to that of the $(1,2)$-subword above.
  
  Finally, consider the $(i,i+1)$-subword.  This is only affected by the changes to $b_{p_{i+1}}$, $b_{p_i}$, and $b_{t_{i-1}}$.  
  The same argument as above shows that it stays $\{i\}$-highest weight after changing $b_{p_{i+1}}$ and $b_{p_i}$, and 
  then changing $b_{t_{i-1}}$ to $i$ certainly keeps it $\{i\}$-highest weight as well.  This completes the proof.
\end{proof}

From the above proof, we immediately obtain the following corollary.

\begin{corollary}\label{cor:which-hw}
Let $b\in \mathcal{B}^{\otimes \ell}$ be $\{1,2,\ldots,i\}$-highest weight for $i\in I_0$, with $\varepsilon_{-i}(b)=1$.  
Let $t_1,\ldots,t_{i-1}$ be the positions of the $1,2,\ldots,i-1$ that change to $2,3,\ldots,i$ respectively in the second step 
of the computation of $e_{-i}(b)$ (see Theorem \ref{theorem.e-i}).
Then if $e_{-i}(b)$ is not $\{1,2,\ldots,i\}$-highest weight, the smallest index $\ell$ for which $e_{-i}(b)$ is not 
$\{\ell\}$-highest weight is precisely the smallest index for which $t_{\ell-1}\leqslant t_\ell$ and $t_{\ell+1}<t_\ell$ (these 
inequalities being vacuously true if $\ell=1$ or $\ell=i-1$, respectively).

In other words,
$\ell$ is the smallest index for which one needs to cycle to get from $t_{\ell-1}$ to $t_\ell$, but one does not
need to cycle to get from $t_\ell$ to $t_{\ell+1}$.
\end{corollary}

\begin{proof}
The proof of Lemma \ref{proposition.e-output-hw} shows that $e_{-i}(b)$ is not $\{\ell\}$-highest weight, 
and that it is $\{k\}$-highest weight for $k<\ell$ if $t_{k-1}\leqslant t_k\leqslant t_{k+1}$ (i.e., if $t_k$ and $t_{k+1}$ both cycle). 
\end{proof}

\subsection{Relation among $e_{-i}$}
\label{section.proposition bypass}

The main result of this section is Proposition~\ref{proposition.by-pass}, which provides relations between
$e_{-i}$ that do and do not yield a $\{1,2,\ldots,i\}$-highest weight element when acting on an $I_0$-highest weight
element. This proposition will be used in Section~\ref{section.G(C)} to deal with ``by-pass arrows'' in the component 
graph $G(\mathcal{C})$.

We require several technical lemmas about $k$-bracketed entries and the $e_{-i}$ operation on highest weight words.

\begin{lemma}\label{lem:k-bracketed}
 Suppose $b\in \mathcal{B}^{\otimes \ell}$ is $\{1,2,\ldots,i\}$-highest weight and $1\leqslant k\leqslant i$.  
 If a letter $b_r=a$ in $b=b_1 b_2 \ldots b_\ell$ is $k$-bracketed, then $b_r$ is $j$-bracketed for all $a<j\leqslant k$.
\end{lemma}

\begin{proof}
  We first show that if an entry $a$ in $b$ is $(a+2)$-bracketed, then it is $(a+1)$-bracketed; for simplicity we set $a=1$. 
  Let $v$ be the subword of $b$ consisting of only the $2$'s that are bracketed with a $3$ along with all the $1$'s, and 
  let $v'$ be the subword consisting of all the $1$'s and $2$'s.  Then $v'$ can be formed from $v$ by inserting some $2$ 
  letters.  It therefore suffices to show that any $1$ that was bracketed in $v$ is still bracketed after inserting a  single $2$.

  Indeed, let $v_s=2$ and $v_r=1$ be a bracketed pair in $v$.  Note that by the definition of the ordinary crystal bracketing
   rule, the subword $v_s\ldots v_r$ has exactly the same number of $2$'s as $1$'s, all of them bracketed with some other 
   letter in $v_s\ldots v_r$.  Therefore, if we insert a $2$ to the left or right of this pair, then the pair $(v_s,v_r)$ remains 
   bracketed.  If instead we insert it between $v_s$ and $v_r$, then the interval between $v_s$ and $v_r$  contains  
   strictly more $2$'s than $1$'s, and so there is some entry $v_t$ between $v_s$ and $v_r$ for which the subword 
   $v_t\cdots v_r$ is tied; in other words, $v_r$ is now bracketed with some $2$ to the right of $v_s$.  Thus $v_r$ 
   stays bracketed after inserting a $2$, as desired.

 Now, if $b_r=a$ is $k$-bracketed, then by the above reasoning it is also $(k-1)$-bracketed, since there are weakly more 
 $(k-1)$'s available in this bracketing, and hence weakly more $(k-2)$'s available, and so on.  The conclusion follows
  by induction.
\end{proof}

\begin{lemma}
\label{lemma.sequence}
Let $b\in \mathcal{B}^{\otimes \ell}$ be $\{1,2,\ldots,i\}$-highest weight and $\varepsilon_{-i}(b)=1$.
Let $b_{p_{i+1}},\ldots,b_{p_1}$ be the initial $(i+1)$-sequence of $b$ and $c$ the word obtained by changing
$b_{p_j}$ from $j$ to $j-1$. Let $k \leqslant i' \leqslant i$.
If $b$ contains a sequence of letters $k-1,k-2,\ldots,1$ before position $p_1$ that is not $i'$-bracketed,
then $c$ contains a sequence of letters $k-1,k-2,\ldots,1$ before position $p_1$ that is not $i'$-bracketed.
\end{lemma}

\begin{proof}
Suppose that $b$ contains a sequence $S$ of letters $k-1,k-2,\ldots,1$ in positions $s_{k-1},\ldots,s_1$ respectively, 
before position $p_1$, that are not $i'$-bracketed; take $S$ to be the rightmost such sequence in the sense that it 
contains the rightmost $1$ left of $p_1$ that is not $i'$-bracketed, then the rightmost $2$ that is not $i'$-bracketed 
before that, and so on.  Note that $s_1<p_1$ implies that $s_1<p_2$ by the definition of $p_1$.  Thus $s_2<s_1<p_2$ 
and so $s_2<p_3$, and so on, showing that $s_j<p_{j+1}$ for all $j$.  Also note that the initial $(i+1)$-sequence 
$b_{p_{i+1}},\ldots,b_{p_1}$ is $(i+1)$-bracketed, so that the letters $b_{p_k},\ldots,b_{p_1}$ must also be $i'$-bracketed 
by Lemma \ref{lem:k-bracketed}. Since $k\leqslant i' \leqslant i$, this means that the initial $(i+1)$-sequence is disjoint from 
$S$ and hence $S$ remains unchanged in $c$.  

We now form a sequence $S'$ from $S$ that is not $i'$-bracketed in $c$ as follows.  Consider the largest entry 
$j\leqslant i'$ for which there exists a $j$ between $p_{j+2}$ and $p_{j+1}$. Then all bracketing with higher letters remains 
the same in $c$, but the letter $j$ between positions $p_{j+2}$ and $p_{j+1}$ becomes bracketed with the letter
$j+1$ in position $p_{j+2}$ in the $i'$-bracketing in $c$, leaving the letter $j$ in position $p_{j+1}$ to be an $i'$-unbracketed 
$j$. If $s_j<p_{j+2}$ (or otherwise $c_{s_j}$ does not become bracketed) we keep it in $S'$, and if 
$p_{j+2}<s_j<p_{j+1}$ and it becomes bracketed, we replace $s_j$ with the first $i'$-unbracketed position $s'_j$ of a 
$j$ in $c$ to the right of $s_j$, to choose the $j$ for $S'$.  

We now show that we can choose a $j-1$ after this step to be in $S'$.  If the $j$ on the previous step did not change, 
then we repeat this process for $j-1$.  If it did change, from $s_j$ to an index $s_j'$, note that if $s_{j-1}<s'_j$ then 
the previous $j-1$ is now $i'$-bracketed with $s_{j}$ in $c$ as well, so we also have to choose the next $j-1$ to the right.  
Either way we replace $s_{j-1}$ with the next $i'$-unbracketed $j-1$, in position $s_{j-1}'$, if the $j-1$ became bracketed, 
and we see that $s_{j}'<s_{j-1}'$.  Furthermore, $s_{j-1}'\leqslant p_{j}$ since we know that $p_j$ becomes an $i'$-unbracketed 
$j-1$ as in the case of $j$ above.  Continuing in this manner we can form a sequence $S'$ of elements of $c$ that are not 
$i'$-bracketed, all weakly to the left of $p_2$ (and hence strictly before $p_1$).
\end{proof}

\begin{lemma}\label{lem:e-ie-k}
Let $b \in \mathcal{B}^{\otimes \ell}$ be $I_0$-highest weight such that $\varepsilon_{-i}(b)>0$ for some $i\in I_0$
and $e_{-i}(b)$ is not $\{1,2,\ldots,i\}$-highest weight.  Let $k$ be the smallest index for which $t_{k-1}\leqslant t_k$, where 
$t_0=p_1$ and $t_j$ for $j=1,\ldots,i-1$ are the indices that are raised in the second step of the computation of $e_{-i}(b)$ 
(such a $k$ exists by Proposition \ref{proposition.e-output-hw}).  Then we have that $\varepsilon_{-k}(b)=1$ and  $e_{-k}(b)$ 
is $\{1,2,\ldots,k\}$-highest weight.
\end{lemma}

\begin{proof}
Let $b_{p_{i+1}},b_{p_{i}},\ldots,b_{p_1}$ be the initial $(i+1)$-sequence, 
$b_{q_{i}},b_{q_{i-1}},\ldots,b_{q_1}$ be the initial $i$-sequence,
$b_{p'_{k+1}},\ldots,b_{p'_1}$ the initial $(k+1)$-sequence, and
$b_{q'_k},\ldots,b_{q'_{1}}$ the initial $k$-sequence of $b$.  Also define $c$ and $c'$ respectively to be the words 
formed by lowering the entries in the sequences $\{b_{p_j}\}$ or $\{b_{p'_j}\}$ by one, respectively.

Since $\varepsilon_{-i}(b)>0$, we have by Proposition~\ref{proposition.phi -i} that
$q_a = p_a$ for some $1\leqslant a \leqslant i$. If $a$ is maximal with this property, then in
fact $q_j=p_j$ for all $j\leqslant a$ by the definition of the initial sequences.
Assume by contradiction that $\varepsilon_{-k}(b)=0$.  Then again by Proposition \ref{proposition.phi -i},
$q'_j< p'_j$ for all $j\in \{1,\ldots,k\}$. Furthermore, $p'_j\leqslant p_j$ for all $j\leqslant k$ so $q'_j<p_j$ as well.  

Suppose that $q'_{a'}=q_{a'}$ for some $1\leqslant a'\leqslant k$. Then $q'_j = q_j$ for all $j\leqslant a'$ and hence 
$q'_j=q_j=p_j$ for $j\leqslant \min(a,a')$, contradicting the fact that $q'_j<p_j$ for all $j$.  Hence $q'_j < q_j$ for all 
$1\leqslant j \leqslant k$.  Thus we also have $q'_j<q_{j+1}$ for all $1\leqslant j \leqslant k$, for otherwise $b_{q'_j}$ 
would be the first $j$ after $q_{j+1}$ and we would have $q'_j=q_j$.  

The sequence of letters $k,k-1,\ldots,1$ in positions $q'_k,\ldots,q'_1$ in $b$ is not $i$-bracketed since
the first bracketed $k+1$ in $b$ must be weakly right of position $q_{k+1}>q'_{k}$. Hence by 
Lemma~\ref{lemma.sequence}, the word $c$ also contains a sequence $k,k-1,\ldots,1$ of letters that are not 
$i$-bracketed before position $p_1$, contradicting the fact that $t_{k-1}\leqslant t_k$. It follows that $\varepsilon_{-k}(b)=1$.

Next we show that $e_{-k}(b)$ is $\{1,2,\ldots,k\}$-highest weight. Note that by the definition of 
the initial sequences $q'_j \leqslant p'_j \leqslant q_j \leqslant p_j$. Since $\varepsilon_{-i}(b)=1$
and $\varepsilon_{-k}(b)=1$, we also have $q'_j=p_j'$ for $j\leqslant a'$ and
$q_j=p_j$ for $j\leqslant a$ for some $a',a$. Suppose $p'_j<q_j$ for all $j$.  Then by a similar argument to that above, in 
the word $c$ there exists a sequence of positions $t_k<t_{k-1}<\cdots<t_1<t_0=p_1$ such that $c_{t_j}=j$ which are not 
$i$-bracketed in $c$.  This contradicts the fact that $t_{k-1}\leqslant t_k$.  Hence we must have $p'_j=q_j$ for some $j$ 
and hence $q'_j=p'_j=q_j=p_j$ for $j\leqslant x$ for some $x\geqslant 1$.
We claim that $t_j<q'_j$ for all $1\leqslant j<k$. Indeed, $t_1$ is to the left of position $p_1=q'_1$, so that $t_1<q'_1$.  
By the definition of $p_1$ we also cannot have $p_2<t_1<p_1$ so in fact $t_1\leqslant p_2$.
The letter in position $q'_j=p_j$ for $1<j\leqslant x$ in $c$ is $j-1$, so that also $t_j<q'_j$ for $1<j\leqslant x$.
For $j>x$, the letter in position $q_j'<p_j$ in $c$ as well as in $b$ is $j$. It is $k$-bracketed in $c$ and $b$ since
the first letter $k$ in $c$ and $b$ is in position $q'_k$.  If $t_j\geqslant q_{j}'$ then since the sequence of entries $q_r'$ 
for $r\geqslant j$ is $k$-bracketed but not $i$-bracketed, we would have $t_k<t_{k-1}$, a contradiction. Thus $t_j<q'_j$.  

It follows that the $t_j$ entries are not $k$-bracketed, so 
 $b$ contains a sequence $k-1,k-2,\ldots,1$ that is not $k$-bracketed.  By Lemma~\ref{lemma.sequence} this means that
$c'$ has a sequence $k-1,\ldots,1$ in positions $t'_{k-1}<\cdots<t'_1$ that is not $k$-bracketed,
proving that $e_{-k}(b)$ is $\{1,2,\ldots,k\}$-highest weight by Proposition~\ref{proposition.e-output-hw}.
\end{proof}

For an element $b \in \mathcal{B}^{\otimes \ell}$, denote by $\uparrow b$ the unique $I_0$-highest weight element in 
the same component as $b$.  The next lemma describes the action of $\uparrow$ after an application of $e_{-i}$.

\begin{lemma}\label{lem:uparrowe-i}
  Let $b \in \mathcal{B}^{\otimes \ell}$ be $I_0$-highest weight such that $\varepsilon_{-i}(b)>0$ for some $i\in I_0$
and $e_{-i}(b)$ is not $\{1,2,\ldots,i\}$-highest weight.  Let $k$ be as in Lemma \ref{lem:e-ie-k} and let the sequences 
$p_j$ and $t_j$ be as in Theorem \ref{theorem.e-i}.  Then $\uparrow e_{-i} (b)$ can be obtained from $b$ by 
changing $j$ in position $p_j$ to $j-1$ for $1<j \leqslant i+1$ and $j$ in position $t_j$ for $1\leqslant j<k$ to $j+1$, and 
lowering some letters larger than $i+1$.  In particular, the changes in positions $t_j$ for $j\geqslant k$ in $e_{-i}$ are 
undone by the application of $\uparrow$.
\end{lemma}

\begin{proof}
    By Corollary \ref{cor:which-hw}, the smallest index $\ell$ for which $e_{\ell}(e_{-i}(b))$ is defined is the first $\ell$ for 
    which $t_\ell$ cycled but $t_{\ell+1}$ did not (or does not exist).  In particular $\ell\geqslant k$ and all $t_j$ with 
    $k\leqslant j\leqslant \ell$ cycle around the end of the word.  
    
    Note that $t_\ell$ was chosen as the rightmost $\ell$ that is not $i$-bracketed (after raising $t_1,\ldots,t_{\ell-1}$).  
    Also recall that the word $c$ formed by lowering the $b_{p_j}$ entries is $\{1,2,\ldots,i\}$-highest weight, so just before 
    changing $t_\ell$ the word is still $\{\ell\}$-highest weight.  Finally, by assumption $t_\ell$ is weakly right of $t_{\ell-1}$ 
    (which is the only new $\ell$ since starting at the word $c$).  Thus, after changing $t_\ell$ to $\ell+1$, if it bracketed 
    with an $\ell$ to its right (in the ordinary crystal bracketing) then in fact that $\ell$ is also not $i$-bracketed on the 
    previous step, a contradiction since $t_{\ell-1}\leqslant t_\ell$.  
    
    Therefore $t_\ell$ is an unbracketed $\ell+1$ in $e_{-i}(b)$, and since all other $(\ell+1)$'s before it are bracketed with 
    some $\ell$, we know that $e_{\ell}$ changes it back to an $\ell$.  After doing so, by the same argument we see that 
    position $t_{\ell-1}$ is now an unbracketed $\ell$, so applying $e_{\ell-1}$ changes it back to $\ell-1$, and so on down 
    to $t_k$.  At this point the resulting word
    \[
    	w:=e_k\cdots e_{\ell-1}e_\ell(e_{-i}b)
    \]
    is $\{1,2,\ldots,\ell\}$-highest weight, since $t_{k-1}$ did not cycle and so changing $t_k$ back to $k$ leaves 
    $w$ highest weight at that step.
    
    Now suppose $t_{\ell+1}$ exists (that is, $\ell\leqslant i-2$); then $t_{\ell+1}<t_\ell$, and in $w$ the position $t_\ell$ is 
    changed back to $\ell$.  We claim that $e_{\ell+1}$ is defined on $w$ and applying it changes $t_{\ell+1}$ from $\ell+2$ 
    back to $\ell+1$. Indeed, if $t_{\ell+1}$ is bracketed with an $\ell+1$ in $w$ then this $\ell+1$ must be to the right of $t_\ell$ 
    (since otherwise it would have been a preferred non-$i$-bracketed choice of $t_{\ell+1}$ in the $e_{-i}$ algorithm).  But
    then this $\ell+1$ is bracketed with an $\ell$ to its right since $w$ is $\{\ell\}$-highest weight, and then this $\ell$ similarly 
    contradicts the choice of $t_\ell$.  Thus $t_{\ell+1}$ is an $\ell+2$ that is not bracketed with an $\ell+1$ after lowering 
    $t_\ell$ back to $\ell$.  By the weight changes it must be the only such $\ell+2$ and so applying $e_{\ell+1}$ changes 
    $t_{\ell+1}$ back to $\ell+1$.
    Continuing in this fashion, we can apply $e_{\ell+2},e_{\ell+3}$, and so on in that order to change the next entries 
    $t_{\ell+2}$, $t_{\ell+3}$, and so on back to their original values, until some $t_{\ell+r}$ cycles again. Let $t_m$ be the 
    next entry for which $t_{m+1}$ does not cycle (the end of the next block of cycling entries); by the same arguments as  
    above we can now apply $e_{m}$, then $e_{m-1}$, and so on down to $e_{\ell+r}$.  Repeating this process on 
    every block of cycling and non-cycling entries yields a $\{1,\ldots,i\}$-highest weight word formed by changing 
    $t_k,\ldots,t_{i-1}$ back to $k,k+1,\ldots,i-1$ respectively.  Finally, to finish forming $\uparrow e_{-i}(b)$, only entries 
    larger than $i+1$ may be changed, and the conclusion follows.
\end{proof}

The next proposition will be used in Section~\ref{section.G(C)} to deal with ``by-pass arrows'' in the component graph 
$G(\mathcal{C})$.

\begin{proposition}
\label{proposition.by-pass}
Let $b \in \mathcal{B}^{\otimes \ell}$ be $I_0$-highest weight such that $\varepsilon_{-i}(b)>0$ for some $i\in I_0$
and $e_{-i}(b)$ is not $\{1,2,\ldots,i\}$-highest weight. Then there exists $1\leqslant k<i$ such that 
$\varepsilon_{-k}(b)=1$, $e_{-k}(b)$ is $\{1,2,\ldots,k\}$-highest weight and
\begin{equation}
\label{equation.e-i e-k com}
	\uparrow e_{-i}(b) = \uparrow e_{-i} \uparrow e_{-k}(b) \quad \text{or} \quad
	\uparrow e_{-i}(b) = \uparrow e_{-k}(b).
\end{equation}
\end{proposition}

\begin{example}
Take $b=343212211 \in \mathcal{B}^{\otimes 9}$, which satisfies $\varepsilon_{-3}(b)>0$. Then
\[
	\uparrow e_{-3} b = e_2 e_1 e_{-3} b = 332112211 = e_2 e_{-3} e_{-1} b  = \uparrow e_{-3} \uparrow e_{-1} b.
\]
Furthermore, $e_{-1} b = 343112211$ is $\{1\}$-highest weight.

Take $b=4321321 \in \mathcal{B}^{\otimes 7}$, which satisfies $\varepsilon_{-3}(b)>0$. Then
\[
	\uparrow e_{-3} b = e_1 e_2 e_{-3} b = 3211321 = e_{-3} e_2 e_{-1} b 
	= \uparrow e_{-3} \uparrow e_{-1}b.
\]
Furthermore, $e_{-1}b=4311321$ is $\{1\}$-highest weight.

Take $b=2154321 \in \mathcal{B}^{\otimes 7}$, which satisfies $\varepsilon_{-4}(b)>0$. Then
\[
 	\uparrow e_{-4} b = e_3 e_{-4} b = 3243211 = e_4 e_{-3} b = \uparrow e_{-3} b.
\]
\end{example}

\begin{proof}[Proof of Proposition~\ref{proposition.by-pass}]
Let $k$ be as in Lemma \ref{lem:e-ie-k}.  Then the first statements hold for $k$ by Lemma \ref{lem:e-ie-k} and it only 
remains to prove~\eqref{equation.e-i e-k com}.  By Lemma \ref{lem:uparrowe-i}, $\uparrow e_{-i} b$ changes $j$ in 
position $p_j$ to $j-1$ for $1<j \leqslant i+1$ and $j$ in position $t_j$ for $1\leqslant j<k$ to $j+1$. The changes in 
positions $t_j$ for $j\geqslant k$ in $e_{-i}$ are undone by $\uparrow$. Some letters bigger than $i+1$ might also be 
lowered by $\uparrow$. 

We use the same notation as in the proof of Lemma~\ref{lem:e-ie-k}. There we proved that $t_j<q'_j$ for all $1\leqslant j<k$. 
Since $q'_j \leqslant p_j$ and there is no letter $j$ between positions $p_{j+1}$ and $p_j$ in $b$, it follows that 
$t_j \leqslant p_{j+1}$ for all $1\leqslant j < k$.
Now suppose that $t_j=p_{j+1}$ for some $1\leqslant j<k$. We claim that then $t_{j-1}=p_j$ as well.
Let $d-1$ be maximal such that $t_{d-1}=p_d$. Then there has to be a letter $d-1$ in position $p$ in $b$ with
$p_{d+1}<p<p_d$, so that the letter $d-1$ in position $p_d$ in $c$ is not $i$-bracketed. 
Suppose that there is no letter $d-2$ between positions $p$ and $p_{d-1}$ in $b$. In this case the letter $d-2$ in 
position $p_{d-1}$ in $c$ is $i$-bracketed, so that $t_{d-2}>p_{d-1}$, which contradicts $t_{d-2}\leqslant p_{d-1}$.
Continuing this argument, there has to be a sequence of letters $d-1,d-2,\ldots,1$ between positions $p_{d+1}$ and $p_2$
that is not $i$-bracketed. Moreover, letter $j$ in this sequence has to appear before position $p_{j+1}$.
But this means that the letter $j$ in position $p_{j+1}$ for $1\leqslant j<d$ is not $i$-bracketed, so that $t_j=p_{j+1}$
for all $1\leqslant j<d$.

By the arguments above, we have that $t_j=p_{j+1}$ for $1\leqslant j <d$ for some $d$ and $t_j$ for $j\geqslant d$ is part 
of a sequence of non $k$-bracketed letters in $b$ (by the definition of $k$ and the sequence $q_j'$). Similarly, we have 
$t'_j=p'_{j+1}$ for $1\leqslant j<d'$ for some $d'$ and $t_j'$ for $j\geqslant d'$ is part of the same sequence of non 
$k$-bracketed letters in $b$ as $t_j$. Also, $d'\geqslant d$ since $p'_j\leqslant p_j$ for all $1\leqslant j\leqslant k+1$. 
In particular, this implies $t_j=t'_j$ for $d'\leqslant j<k$. 

Furthermore, before applying the $\uparrow$ operator the entries that change are:
\begin{align*}
	\text{In $\uparrow e_{-i} b$:} \qquad & b_{p_j}\colon j\mapsto j-1 \quad \text{for }d<j\leqslant i+1\\
        &b_{t_j} \colon j\mapsto j+1 \quad \text{for }d\leqslant j<i\\        
        	\text{In $\uparrow e_{-k} b$:} \qquad & b_{p'_j}\colon j\mapsto j-1 \quad \text{for }d'<j\leqslant k+1\\
        &b_{t'_j} \colon j\mapsto j+1 \quad \text{for }d'\leqslant j<k.     
\end{align*}

Recall also that $p'_j=p_j$ for $1\leqslant j\leqslant x$ for some $x\geqslant 1$.
Denote by $\overline{t}_j$ and $\overline{p}_j$ the selected positions by $e_{-i}$ on the element $\uparrow e_{-k} b$.

First assume that $x=k+1$, so that $p'_j=p_j$ for all $1\leqslant j \leqslant k+1$. In this case $t'_j=t_j$ for
$1\leqslant j <k$. Furthermore, if in $e_{-k}(b)$ the letter $k+2$ in position $p_{k+2}$ is unbracketed, then
in $\uparrow e_{-k}(b)$, the letter $k+2$ in position $p_{k+2}$, then the letter $k+3$ in position $p_{k+3}$ etc
will be lowered. These are the same changes as in $\uparrow e_{-i}(b)$, so that $\uparrow e_{-i}(b) = \uparrow e_{-k}(b)$.

Next assume that $d'<x\leqslant k$ or $x=k+1$ but the letter $k+2$ in position $p_{k+2}$ in $e_{-k}(b)$ is bracketed.  
We first show that in this case $\overline{p}_j=p_j$ for $x<j\leqslant i+1$.  
Note that to form $\uparrow e_{-k}(b)$, since $e_{-k}(b)$ is $\{1,2,\ldots,k\}$-highest weight, we apply $e_{k+1}, e_{k+2},
\ldots,e_r$ in order for some $r$, so that we lower a $k+2$ to a $k+1$, $k+3$ to $k+2$, and so on until we reach an 
$I_0$-highest weight word.   Note also that $b_{p_{k+1}'}$ was the entry that lowered from $k+1$ to $k$, so the $k+2$ 
that gets lowered, if it exists, is to the left of $p_{k+1}'<p_{k+1}$.  Similarly the $k+3$ that gets lowered is left of 
$p_{k+2}'<p_{k+2}$, and so on, and hence $r<i$ since $p_{i+1}$ is the leftmost $i+1$.  It follows that no $i+1$ lowers
to an $i$, and so $\overline{p}_{i+1}=p_{i+1}$.  Since the entries lowered by $\uparrow$ are left of $p_j$ for each $j>x$,
it follows that $\overline{p}_{j}=p_j$ for $x<j\leqslant i+1$.

For the sequence $\overline{t}_j$, note that the entries $\overline{p}_j$ that we lower for $j\leqslant x$ cannot be 
$i$-bracketed in $\overline{c}$ due to the condition $\overline{p}_{i+1}=p_{i+1}$ shown above, and because 
$t_{x-1}=t_{x-1}'$, so that $t_{x-1}'$ cannot be between $p_{x+1}$ and $p_x$.  
Furthermore, for $x\leqslant j<k$ the letters in positions $\overline{p}_{j+1}$ are all $i$-bracketed in $\overline{c}$
and $t_j=t_j'<p_{j+1}'< p_{j+1}=\overline{p}_{j+1}$.
Also note that $d=d'$ since $p_{j+1}=p_{j+1}'=t'_j$ for $d\leqslant j<d'<x$ and the letter $j$ in position
$p_{j+1}=p_{j+1}'$ in $c'$ is not $k$-bracketed and hence not $i$-bracketed in $c'$ and $c$. It follows that 
\[
	\overline{t}_j = \begin{cases}
	\overline{p}_{j+1} & \text{for $1\leqslant j<x$,}\\
	p_{j+1}' & \text{for $x\leqslant j\leqslant k$,} \end{cases}
\]
and for $k<j\leqslant r$, we have that $\overline{t}_j$  is equal to the position of letter $j+1$ that is lowered when applying
$\uparrow$ to $e_{-k}(b)$. Hence $\uparrow e_{-i}(b) = \uparrow e_{-i} \uparrow e_{-k}(b)$.

Finally, assume that $x\leqslant d'$. In this case, by a similar argument, we have $\overline{p}_j=p_j$ for 
$1\leqslant j\leqslant i+1$ and
\[
	\overline{t}_j = \begin{cases}
	\overline{p}_{j+1} & \text{for $1\leqslant j<d$,}\\
	t_j & \text{for $d\leqslant j <d'$,}\\
	p_{j+1}' & \text{for $d'\leqslant j\leqslant k$,} \end{cases}
\]
and for $k<j\leqslant r$, we have that $\overline{t}_j$ is equal to the position of letter $j+1$ that is lowered when applying
$\uparrow$ to $e_{-k}(b)$. Again, we have $\uparrow e_{-i}(b) = \uparrow e_{-i} \uparrow e_{-k}(b)$.
\end{proof}

\section{Local axioms}
\label{section.axioms}

In~\cite[Definition 4.11]{AssafOguz.2018a}, Assaf and Oguz give a definition of regular queer crystals.
In essence, their axioms are rephrased in the following definition, where $\tilde{I}:=I_0 \cup \{-1\}$.

\begin{definition}[\defn{Local queer axioms}]
\label{definition.local queer axioms}
Let $\mathcal{C}$ be a graph with labeled directed edges given by $f_i$ for $i \in I_0$ and $f_{-1}$.
If $b'=f_jb$ for $j \in \tilde{I}$ define $e_j$ by $b=e_j b'$.
\begin{enumerate}
\item[{\bf LQ1.}] The subgraph with all vertices but only edges labeled by $i\in I_0$ is a type $A_n$ Stembridge crystal.
\item[{\bf LQ2.}] $\varphi_{-1}(b),\varepsilon_{-1}(b) \in \{0,1\}$ for all $b\in \mathcal{C}$.
\item[{\bf LQ3.}] $\varphi_{-1}(b)+\varepsilon_{-1}(b)>0$ if $\wt(b)_1 + \wt(b)_2>0$.
\item[{\bf LQ4.}] Assume $\varphi_{-1}(b)=1$ for $b\in \mathcal{C}$.
\begin{enumerate}
\item If $\varphi_1(b)>2$, we have
\begin{equation*}
\begin{split}
	f_1 f_{-1}(b) &= f_{-1} f_1(b),\\
	\varphi_1(b) &= \varphi_1(f_{-1}(b))+2,\\
	\varepsilon_1(b) &= \varepsilon_1(f_{-1}(b)).
\end{split}
\end{equation*}
\item If $\varphi_1(b)=1$, we have
\[
	f_1(b) = f_{-1}(b).
\]
\end{enumerate}
\item[{\bf LQ5.}] Assume $\varphi_{-1}(b)=1$ for $b\in \mathcal{C}$.
\begin{enumerate}
\item If $\varphi_2(b)>0$, we have
\begin{equation*}
\begin{split}
	f_2 f_{-1}(b) &= f_{-1} f_2(b),\\
	\varphi_2(b) &= \varphi_2(f_{-1}(b))-1,\\
	\varepsilon_2(b) &= \varepsilon_2(f_{-1}(b)).
\end{split}
\end{equation*}
\item If $\varphi_2(b)=0$, we have
\begin{equation*}
\begin{aligned}
	\varphi_2(b) &= \varphi_2(f_{-1}(b))-1=0, & \text{or} \quad \varphi_2(b) &= \varphi_2(f_{-1}(b))=0,\\
	\varepsilon_2(b) &= \varepsilon_2(f_{-1}(b)), & \varepsilon_2(b) &= \varepsilon_2(f_{-1}(b))+1.
\end{aligned}
\end{equation*}
\end{enumerate}
\item[{\bf LQ6.}] Assume that $\varphi_{-1}(b)=1$ and $\varphi_i(b)>0$ with $i\geqslant 3$ for $b\in \mathcal{C}$. Then
\begin{equation*}
\begin{split}
	f_i f_{-1}(b) &= f_{-1} f_i(b),\\
	\varphi_i(b) & = \varphi_i(f_{-1}(b)),\\
	\varepsilon_i(b) &= \varepsilon_i(f_{-1}(b)).
\end{split}
\end{equation*}
\end{enumerate}
\end{definition}

Axioms {\bf LQ4} and {\bf LQ5} are illustrated in Figure~\ref{figure.local queer}.

\begin{figure}
\centering
\begin{tikzpicture}[node distance = 1.5cm, auto]
  \tikzstyle{place}=[circle,fill=black,scale=1]
  \tikzstyle{minusone}=[->,dashed,blue,thick]
  \tikzstyle{one}=[->,blue,thick]
  \tikzstyle{two}=[->,darkred,thick]
  \node [place] (A1) {};  
  \node [place] (B1) [right= 1cm of A1] {};
  \draw [minusone] (A1) edgenode{$-1$} (B1);
  \node [place] (A2) [below= 1cm of A1] {};  
  \node [place] (B2) [right= 1cm of A2] {};
  \draw [minusone] (A2) edgenode{$-1$} (B2);
  \draw [one] (A1) edgenode[left]{$1$} (A2);
  \draw [one] (B1) edgenode[right]{$1$} (B2);
  \node [place] (A3) [below= 1cm of A2] {};  
  \node [place] (B3) [right= 1cm of A3] {};
  \draw [minusone] (A3) edgenode{$-1$} (B3);
  \draw [one] (A2) edgenode[left]{$1$} (A3);
  \draw [one] (B2) edgenode[right]{$1$} (B3);
  \node [place] (A4) [below= 1cm of A3] {};  
  \node [place] (B4) [right= 1cm of A4] {};
  \draw [minusone] (A4) edgenode{$-1$} (B4);
  
  \draw [loosely dotted,blue,thick] (A3) -- (A4);
  \draw [loosely dotted,blue,thick] (B3) -- (B4);
  
  \node [place] (A5) [below= 1cm of A4] {};  
  \node [place] (B5) [right= 1cm of A5] {};
  \draw [minusone] (A5) edgenode{$-1$} (B5);
  \draw [one] (A4) edgenode[left]{$1$} (A5);
  \draw [one] (B4) edgenode[right]{$1$} (B5);
  
  \node [place] (A6) [below= 1cm of A5] {};
  \draw [one] (A5) edgenode[left]{$1$} (A6);
  \node [place] (A7) [below= 1cm of A6] {};
  \draw [minusone,transform canvas={xshift=-1mm}] (A6) edgenode[left]{$-1$} (A7);
  \draw [one,transform canvas={xshift=1mm}] (A6) edgenode[right]{$1$} (A7);

  \node [place] (C1) [right= 2cm of B1] {};  
  \node [place] (D1) [right= 1cm of C1] {};
  \draw [minusone] (C1) edgenode{$-1$} (D1);
  \node [place] (C2) [below= 1cm of C1] {};  
  \node [place] (D2) [right= 1cm of C2] {};
  \draw [minusone] (C2) edgenode{$-1$} (D2);
  \draw [two] (C1) edgenode[left]{$2$} (C2);
  \draw [two] (D1) edgenode[right]{$2$} (D2);
  \node [place] (C3) [below= 1cm of C2] {};  
  \node [place] (D3) [right= 1cm of C3] {};
  \draw [minusone] (C3) edgenode{$-1$} (D3);
  \draw [two] (C2) edgenode[left]{$2$} (C3);
  \draw [two] (D2) edgenode[right]{$2$} (D3);
  \node [place] (C4) [below= 1cm of C3] {};  
  \node [place] (D4) [right= 1cm of C4] {};
  \draw [minusone] (C4) edgenode{$-1$} (D4);
  
  \draw [loosely dotted,red,thick] (C3) -- (C4);
  \draw [loosely dotted,red,thick] (D3) -- (D4);
  
  \node [place] (C5) [below= 1cm of C4] {};  
  \node [place] (D5) [right= 1cm of C5] {};
  \draw [minusone] (C5) edgenode{$-1$} (D5);
  \draw [two] (C4) edgenode[left]{$2$} (C5);
  \draw [two] (D4) edgenode[right]{$2$} (D5);
  
  \node [place] (D6) [below= 1cm of D5] {};
  \draw [two] (D5) edgenode[left]{$2$} (D6);

  \node [place] (E2) [right= 1cm of D2] {};  
  \node [place] (F2) [right= 1cm of E2] {};
  \draw [minusone] (E2) edgenode{$-1$} (F2);
  \node [place] (E3) [below= 1cm of E2] {};  
  \node [place] (F3) [right= 1cm of E3] {};
  \draw [minusone] (E3) edgenode{$-1$} (F3);
  \draw [two] (E2) edgenode[left]{$2$} (E3);
  \draw [two] (F2) edgenode[right]{$2$} (F3);
  \node [place] (E4) [below= 1cm of E3] {};  
  \node [place] (F4) [right= 1cm of E4] {};
  \draw [minusone] (E4) edgenode{$-1$} (F4);
  
  \draw [loosely dotted,red,thick] (E3) -- (E4);
  \draw [loosely dotted,red,thick] (F3) -- (F4);
  
  \node [place] (E5) [below= 1cm of E4] {};  
  \node [place] (F5) [right= 1cm of E5] {};
  \draw [minusone] (E5) edgenode{$-1$} (F5);
  \draw [two] (E4) edgenode[left]{$2$} (E5);
  \draw [two] (F4) edgenode[right]{$2$} (F5);
  
  \node [place] (F6) [below= 1cm of F5] {};
  \draw [two] (F5) edgenode[left]{$2$} (F6);
  
  \draw [minusone] (D6) edgenode{$-1$} (F6);
\end{tikzpicture}
\caption{Illustration of axioms \textbf{LQ4} (left) and \textbf{LQ5} (right). The $(-1)$-arrow at the bottom of the right
figure might or might not be there.}
\label{figure.local queer}
\end{figure}
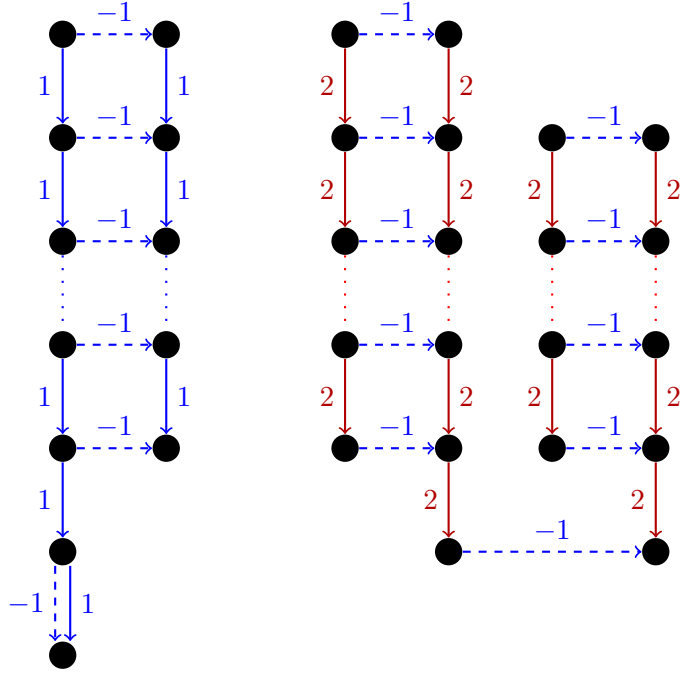

\begin{proposition}[\cite{AssafOguz.2018a}]
\label{proposition.local queer axioms}
The queer crystal of words $\mathcal{B}^{\otimes \ell}$ satisfies the axioms in Definition~\ref{definition.local queer axioms}.
\end{proposition}

\begin{proof}
{\bf LQ1} follows by definition. {\bf LQ2} and {\bf LQ3} follow from Remark~\ref{rmk:combo-minus-1}.
{\bf LQ4} follows from Lemma~\ref{lemma.e1} and {\bf LQ5} follows from Lemma~\ref{lemma.e2}.
Finally, {\bf LQ6} is {\bf Q4}.
\end{proof}

In~\cite[Conjecture 4.16]{AssafOguz.2018a}, Assaf and Oguz conjecture that every regular queer crystal is a 
normal queer crystal. In other words, every connected graph satisfying the local queer axioms
of Definition~\ref{definition.local queer axioms} is isomorphic to a connected component in some
$\mathcal{B}^{\otimes \ell}$. We provide a counterexample to this claim in Figure~\ref{figure.counterexample}.
In the figure, the $I_0$-components of the $\mathfrak{q}(3)$-crystal of highest weight $(4,2,0)$ are shown.
Some of the $f_{-1}$-arrows are drawn in green. The remaining arrows can be filled in using the axioms
of Figure~\ref{figure.local queer} in a consistent manner. If the dashed green arrow from $331131$ to $332131$ and the 
dashed green arrow from $331132$ to $332132$ are replaced by the dashed purple arrow from $331131$ to
$331231$ and the dashed purple arrow from $331132$ to $332231$, respectively, all axioms of 
Definition~\ref{definition.local queer axioms} are still satisfied with the remaining $f_{-1}$-arrows filled in. 
However, the $I_0$-component with highest weight element $132121$ has become disconnected and hence the two 
crystals are not isomorphic.

\begin{figure}
\rotatebox{270}{\scalebox{0.38}{
\begin{tikzpicture}[>=latex,line join=bevel,]
\node (node_32) at (1159.0bp,497.0bp) [draw,draw=none] {$1 \otimes 3 \otimes 1 \otimes 1 \otimes 2 \otimes 1$};
  \node (node_63) at (162.0bp,287.0bp) [draw,draw=none] {$2 \otimes 2 \otimes 1 \otimes 2 \otimes 3 \otimes 2$};
  \node (node_46) at (1347.0bp,147.0bp) [draw,draw=none] {$3 \otimes 3 \otimes 2 \otimes 2 \otimes 3 \otimes 1$};
  \node (node_35) at (390.0bp,287.0bp) [draw,draw=none] {$1 \otimes 3 \otimes 1 \otimes 3 \otimes 3 \otimes 1$};
  \node (node_47) at (376.0bp,77.0bp) [draw,draw=none] {$2 \otimes 3 \otimes 2 \otimes 3 \otimes 3 \otimes 2$};
  \node (node_42) at (1159.0bp,427.0bp) [draw,draw=none] {$1 \otimes 3 \otimes 1 \otimes 2 \otimes 2 \otimes 1$};
  \node (node_34) at (960.0bp,217.0bp) [draw,draw=none] {$3 \otimes 3 \otimes 1 \otimes 1 \otimes 3 \otimes 2$};
  \node (node_26) at (562.0bp,427.0bp) [draw,draw=none] {$1 \otimes 3 \otimes 1 \otimes 1 \otimes 3 \otimes 1$};
  \node (node_27) at (48.0bp,287.0bp) [draw,draw=none] {$2 \otimes 2 \otimes 1 \otimes 3 \otimes 3 \otimes 1$};
  \node (node_24) at (324.0bp,427.0bp) [draw,draw=none] {$1 \otimes 2 \otimes 1 \otimes 2 \otimes 3 \otimes 1$};
  \node (node_25) at (704.0bp,497.0bp) [draw,draw=none] {$2 \otimes 2 \otimes 1 \otimes 1 \otimes 2 \otimes 1$};
  \node (node_22) at (921.0bp,147.0bp) [draw,draw=none] {$3 \otimes 3 \otimes 1 \otimes 2 \otimes 3 \otimes 2$};
  \node (node_23) at (1486.0bp,357.0bp) [draw,draw=none] {$2 \otimes 3 \otimes 1 \otimes 2 \otimes 2 \otimes 1$};
  \node (node_20) at (732.0bp,357.0bp) [draw,draw=none] {$2 \otimes 3 \otimes 2 \otimes 1 \otimes 2 \otimes 1$};
  \node (node_21) at (1074.0bp,77.0bp) [draw,draw=none] {$3 \otimes 3 \otimes 2 \otimes 3 \otimes 3 \otimes 1$};
  \node (node_48) at (618.0bp,287.0bp) [draw,draw=none] {$1 \otimes 3 \otimes 2 \otimes 1 \otimes 3 \otimes 2$};
  \node (node_49) at (1352.0bp,357.0bp) [draw,draw=none] {$3 \otimes 3 \otimes 1 \otimes 1 \otimes 2 \otimes 1$};
  \node (node_60) at (276.0bp,287.0bp) [draw,draw=none] {$1 \otimes 2 \otimes 1 \otimes 3 \otimes 3 \otimes 2$};
  \node (node_36) at (141.0bp,147.0bp) [draw,draw=none] {$3 \otimes 3 \otimes 1 \otimes 3 \otimes 3 \otimes 1$};
  \node (node_28) at (490.0bp,147.0bp) [draw,draw=none] {$2 \otimes 3 \otimes 2 \otimes 2 \otimes 3 \otimes 2$};
  \node (node_29) at (141.0bp,217.0bp) [draw,draw=none] {$2 \otimes 3 \otimes 1 \otimes 3 \otimes 3 \otimes 1$};
  \node (node_37) at (921.0bp,77.0bp) [draw,draw=none] {$3 \otimes 3 \otimes 2 \otimes 2 \otimes 3 \otimes 2$};
  \node (node_61) at (376.0bp,147.0bp) [draw,draw=none] {$1 \otimes 3 \otimes 2 \otimes 3 \otimes 3 \otimes 2$};
  \node (node_62) at (504.0bp,357.0bp) [draw,draw=none] {$1 \otimes 3 \otimes 1 \otimes 1 \otimes 3 \otimes 2$};
  \node (node_9) at (255.0bp,77.0bp) [draw,draw=none] {$3 \otimes 3 \otimes 1 \otimes 3 \otimes 3 \otimes 2$};
  \node (node_8) at (1074.0bp,147.0bp) [draw,draw=none] {$2 \otimes 3 \otimes 2 \otimes 3 \otimes 3 \otimes 1$};
  \node (node_7) at (1074.0bp,357.0bp) [draw,draw=none] {$1 \otimes 3 \otimes 2 \otimes 2 \otimes 2 \otimes 1$};
  \node (node_6) at (846.0bp,287.0bp) [draw,draw=none] {$2 \otimes 3 \otimes 1 \otimes 1 \otimes 3 \otimes 2$};
  \node (node_5) at (846.0bp,217.0bp) [draw,draw=none] {$2 \otimes 3 \otimes 1 \otimes 2 \otimes 3 \otimes 2$};
  \node (node_4) at (504.0bp,287.0bp) [draw,draw=none] {$1 \otimes 3 \otimes 1 \otimes 2 \otimes 3 \otimes 2$};
  \node (node_3) at (732.0bp,217.0bp) [draw,draw=none] {$3 \otimes 3 \otimes 2 \otimes 1 \otimes 3 \otimes 1$};
  \node (node_2) at (618.0bp,217.0bp) [draw,draw=none] {$2 \otimes 3 \otimes 2 \otimes 1 \otimes 3 \otimes 2$};
  \node (node_1) at (1074.0bp,287.0bp) [draw,draw=none] {$1 \otimes 3 \otimes 2 \otimes 2 \otimes 3 \otimes 1$};
  \node (node_0) at (562.0bp,497.0bp) [draw,draw=none] {$1 \otimes 2 \otimes 1 \otimes 1 \otimes 3 \otimes 1$};
  \node (node_53) at (325.0bp,7.0bp) [draw,draw=none] {$3 \otimes 3 \otimes 2 \otimes 3 \otimes 3 \otimes 2$};
  \node (node_52) at (1352.0bp,427.0bp) [draw,draw=none] {$2 \otimes 3 \otimes 1 \otimes 1 \otimes 2 \otimes 1$};
  \node (node_40) at (1416.0bp,287.0bp) [draw,draw=none] {$3 \otimes 3 \otimes 2 \otimes 1 \otimes 2 \otimes 1$};
  \node (node_50) at (1188.0bp,287.0bp) [draw,draw=none] {$2 \otimes 3 \otimes 2 \otimes 2 \otimes 2 \otimes 1$};
  \node (node_57) at (846.0bp,357.0bp) [draw,draw=none] {$2 \otimes 2 \otimes 1 \otimes 1 \otimes 3 \otimes 2$};
  \node (node_56) at (1444.0bp,217.0bp) [draw,draw=none] {$3 \otimes 3 \otimes 1 \otimes 2 \otimes 3 \otimes 1$};
  \node (node_55) at (1188.0bp,217.0bp) [draw,draw=none] {$2 \otimes 3 \otimes 2 \otimes 2 \otimes 3 \otimes 1$};
  \node (node_54) at (255.0bp,147.0bp) [draw,draw=none] {$2 \otimes 3 \otimes 1 \otimes 3 \otimes 3 \otimes 2$};
  \node (node_59) at (632.0bp,567.0bp) [draw,draw=none] {$1 \otimes 2 \otimes 1 \otimes 1 \otimes 2 \otimes 1$};
  \node (node_58) at (960.0bp,357.0bp) [draw,draw=none] {$2 \otimes 3 \otimes 1 \otimes 1 \otimes 3 \otimes 1$};
  \node (node_39) at (390.0bp,357.0bp) [draw,draw=none] {$1 \otimes 2 \otimes 1 \otimes 3 \otimes 3 \otimes 1$};
  \node (node_38) at (661.0bp,147.0bp) [draw,draw=none] {$3 \otimes 3 \otimes 2 \otimes 1 \otimes 3 \otimes 2$};
  \node (node_19) at (618.0bp,357.0bp) [draw,draw=none] {$1 \otimes 3 \otimes 2 \otimes 1 \otimes 3 \otimes 1$};
  \node (node_18) at (376.0bp,217.0bp) [draw,draw=none] {$1 \otimes 3 \otimes 1 \otimes 3 \otimes 3 \otimes 2$};
  \node (node_17) at (846.0bp,427.0bp) [draw,draw=none] {$2 \otimes 2 \otimes 1 \otimes 1 \otimes 3 \otimes 1$};
  \node (node_16) at (1530.0bp,287.0bp) [draw,draw=none] {$2 \otimes 3 \otimes 1 \otimes 2 \otimes 3 \otimes 1$};
  \node (node_15) at (732.0bp,287.0bp) [draw,draw=none] {$2 \otimes 3 \otimes 2 \otimes 1 \otimes 3 \otimes 1$};
  \node (node_14) at (704.0bp,427.0bp) [draw,draw=none] {$1 \otimes 3 \otimes 2 \otimes 1 \otimes 2 \otimes 1$};
  \node (node_13) at (1302.0bp,287.0bp) [draw,draw=none] {$3 \otimes 3 \otimes 1 \otimes 2 \otimes 2 \otimes 1$};
  \node (node_12) at (255.0bp,217.0bp) [draw,draw=none] {$2 \otimes 2 \otimes 1 \otimes 3 \otimes 3 \otimes 2$};
  \node (node_11) at (448.0bp,427.0bp) [draw,draw=none] {$1 \otimes 2 \otimes 1 \otimes 1 \otimes 3 \otimes 2$};
  \node (node_10) at (276.0bp,357.0bp) [draw,draw=none] {$1 \otimes 2 \otimes 1 \otimes 2 \otimes 3 \otimes 2$};
  \node (node_41) at (162.0bp,357.0bp) [draw,draw=none] {$2 \otimes 2 \otimes 1 \otimes 2 \otimes 3 \otimes 1$};
  \node (node_31) at (1074.0bp,217.0bp) [draw,draw=none] {$1 \otimes 3 \otimes 2 \otimes 3 \otimes 3 \otimes 1$};
  \node (node_43) at (490.0bp,217.0bp) [draw,draw=none] {$1 \otimes 3 \otimes 2 \otimes 2 \otimes 3 \otimes 2$};
  \node (node_30) at (324.0bp,497.0bp) [draw,draw=none] {$1 \otimes 2 \otimes 1 \otimes 2 \otimes 2 \otimes 1$};
  \node (node_44) at (1188.0bp,357.0bp) [draw,draw=none] {$1 \otimes 3 \otimes 1 \otimes 2 \otimes 3 \otimes 1$};
  \node (node_51) at (1302.0bp,217.0bp) [draw,draw=none] {$3 \otimes 3 \otimes 2 \otimes 2 \otimes 2 \otimes 1$};
  \node (node_33) at (189.0bp,427.0bp) [draw,draw=none] {$2 \otimes 2 \otimes 1 \otimes 2 \otimes 2 \otimes 1$};
  \node (node_45) at (960.0bp,287.0bp) [draw,draw=none] {$3 \otimes 3 \otimes 1 \otimes 1 \otimes 3 \otimes 1$};
  \draw [red,->] (node_2) ..controls (629.2bp,198.77bp) and (642.24bp,177.54bp)  .. (node_38);
  \definecolor{strokecol}{rgb}{0.0,0.0,0.0};
  \pgfsetstrokecolor{strokecol}
  \draw (653.0bp,182.0bp) node {$2$};
  \draw [blue,->] (node_34) ..controls (949.9bp,198.87bp) and (938.24bp,177.95bp)  .. (node_22);
  \draw (954.0bp,182.0bp) node {$1$};
  \draw [red,->] (node_12) ..controls (255.0bp,199.19bp) and (255.0bp,179.15bp)  .. (node_54);
  \draw (264.0bp,182.0bp) node {$2$};
  \draw [red,->] (node_47) ..controls (362.64bp,58.664bp) and (346.96bp,37.14bp)  .. (node_53);
  \draw (365.0bp,42.0bp) node {$2$};
  \draw [red,->] (node_6) ..controls (877.16bp,267.87bp) and (916.02bp,244.0bp)  .. (node_34);
  \draw (925.0bp,252.0bp) node {$2$};
  \draw [darkgreen,thick,->] (node_49) edgenode[above]{$-1$} (node_40);
  \draw [blue,->] (node_14) ..controls (711.21bp,408.98bp) and (719.46bp,388.35bp)  .. (node_20);
  \draw (730.0bp,392.0bp) node {$1$};
  \draw [red,->] (node_39) ..controls (390.0bp,339.19bp) and (390.0bp,319.15bp)  .. (node_35);
  \draw (399.0bp,322.0bp) node {$2$};
  \draw [red,->] (node_63) ..controls (187.14bp,268.08bp) and (218.01bp,244.85bp)  .. (node_12);
  \draw (228.0bp,252.0bp) node {$2$};
  \draw [blue,->] (node_7) ..controls (1117.0bp,345.92bp) and (1132.5bp,340.08bp)  .. (1145.0bp,332.0bp) .. controls (1157.4bp,324.01bp) and (1168.8bp,311.78bp)  .. (node_50);
  \draw (1176.0bp,322.0bp) node {$1$};
  \draw [darkgreen,thick,->] (node_32) edgenode[above]{$-1$} (node_52);
  \draw [blue,->] (node_11) ..controls (376.89bp,419.46bp) and (349.55bp,413.27bp)  .. (327.0bp,402.0bp) .. controls (312.15bp,394.58bp) and (298.22bp,381.73bp)  .. (node_10);
  \draw (336.0bp,392.0bp) node {$1$};
  \draw [blue,->] (node_59) ..controls (545.13bp,547.26bp) and (427.26bp,520.47bp)  .. (node_30);
  \draw (523.0bp,532.0bp) node {$1$};
  \draw [red,->] (node_50) ..controls (1188.0bp,269.19bp) and (1188.0bp,249.15bp)  .. (node_55);
  \draw (1197.0bp,252.0bp) node {$2$};
  \draw [red,->] (node_17) ..controls (877.16bp,407.87bp) and (916.02bp,384.0bp)  .. (node_58);
  \draw (925.0bp,392.0bp) node {$2$};
  \draw [red,->] (node_20) ..controls (732.0bp,339.19bp) and (732.0bp,319.15bp)  .. (node_15);
  \draw (741.0bp,322.0bp) node {$2$};
  \draw [blue,->] (node_32) ..controls (1159.0bp,479.19bp) and (1159.0bp,459.15bp)  .. (node_42);
  \draw (1168.0bp,462.0bp) node {$1$};
  \draw [blue,->] (node_1) ..controls (1105.2bp,267.87bp) and (1144.0bp,244.0bp)  .. (node_55);
  \draw (1153.0bp,252.0bp) node {$1$};
  \draw [red,->] (node_59) ..controls (613.45bp,548.45bp) and (591.32bp,526.32bp)  .. (node_0);
  \draw (614.0bp,532.0bp) node {$2$};
  \draw [red,->] (node_7) ..controls (1074.0bp,339.19bp) and (1074.0bp,319.15bp)  .. (node_1);
  \draw (1083.0bp,322.0bp) node {$2$};
  \draw [blue,->] (node_3) ..controls (713.19bp,198.45bp) and (690.74bp,176.32bp)  .. (node_38);
  \draw (713.0bp,182.0bp) node {$1$};
  \draw [blue,->] (node_10) ..controls (244.84bp,337.87bp) and (205.98bp,314.0bp)  .. (node_63);
  \draw (241.0bp,322.0bp) node {$1$};
  \draw [blue,->] (node_56) ..controls (1417.7bp,198.03bp) and (1385.3bp,174.64bp)  .. (node_46);
  \draw (1416.0bp,182.0bp) node {$1$};
  \draw [blue,->] (node_18) ..controls (376.0bp,199.19bp) and (376.0bp,179.15bp)  .. (node_61);
  \draw (385.0bp,182.0bp) node {$1$};
  \draw [red,->] (node_58) ..controls (960.0bp,339.19bp) and (960.0bp,319.15bp)  .. (node_45);
  \draw (969.0bp,322.0bp) node {$2$};
  \draw [red,->] (node_28) ..controls (458.84bp,127.87bp) and (419.98bp,104.0bp)  .. (node_47);
  \draw (455.0bp,112.0bp) node {$2$};
  \draw [red,->] (node_11) ..controls (462.67bp,408.66bp) and (479.89bp,387.14bp)  .. (node_62);
  \draw (491.0bp,392.0bp) node {$2$};
  \draw [blue,->] (node_49) ..controls (1339.0bp,338.77bp) and (1323.8bp,317.54bp)  .. (node_13);
  \draw (1341.0bp,322.0bp) node {$1$};
  \draw [red,->] (node_1) ..controls (1074.0bp,269.19bp) and (1074.0bp,249.15bp)  .. (node_31);
  \draw (1083.0bp,252.0bp) node {$2$};
  \draw [blue,->] (node_52) ..controls (1388.9bp,407.71bp) and (1435.5bp,383.37bp)  .. (node_23);
  \draw (1443.0bp,392.0bp) node {$1$};
  \draw [red,->] (node_55) ..controls (1156.8bp,197.87bp) and (1118.0bp,174.0bp)  .. (node_8);
  \draw (1153.0bp,182.0bp) node {$2$};
  \draw [red,->] (node_5) ..controls (865.98bp,198.35bp) and (890.02bp,175.91bp)  .. (node_22);
  \draw (901.0bp,182.0bp) node {$2$};
  \draw [red,->] (node_57) ..controls (846.0bp,339.19bp) and (846.0bp,319.15bp)  .. (node_6);
  \draw (855.0bp,322.0bp) node {$2$};
  \draw [blue,->] (node_0) ..controls (530.84bp,477.87bp) and (491.98bp,454.0bp)  .. (node_11);
  \draw (527.0bp,462.0bp) node {$1$};
  \draw [red,->] (node_54) ..controls (255.0bp,129.19bp) and (255.0bp,109.15bp)  .. (node_9);
  \draw (264.0bp,112.0bp) node {$2$};
  \draw [red,->] (node_0) ..controls (562.0bp,479.19bp) and (562.0bp,459.15bp)  .. (node_26);
  \draw (571.0bp,462.0bp) node {$2$};
  \draw [blue,->] (node_22) ..controls (921.0bp,129.19bp) and (921.0bp,109.15bp)  .. (node_37);
  \draw (930.0bp,112.0bp) node {$1$};
  \draw [blue,->] (node_19) ..controls (618.0bp,339.19bp) and (618.0bp,319.15bp)  .. (node_48);
  \draw (627.0bp,322.0bp) node {$1$};
  \draw [red,->] (node_41) ..controls (130.84bp,337.87bp) and (91.978bp,314.0bp)  .. (node_27);
  \draw (128.0bp,322.0bp) node {$2$};
  \draw [blue,->] (node_61) ..controls (376.0bp,129.19bp) and (376.0bp,109.15bp)  .. (node_47);
  \draw (385.0bp,112.0bp) node {$1$};
  \draw [blue,->] (node_58) ..controls (928.84bp,337.87bp) and (889.98bp,314.0bp)  .. (node_6);
  \draw (925.0bp,322.0bp) node {$1$};
  \draw [red,->] (node_51) ..controls (1313.7bp,198.77bp) and (1327.4bp,177.54bp)  .. (node_46);
  \draw (1338.0bp,182.0bp) node {$2$};
  \draw [blue,->] (node_26) ..controls (546.72bp,408.56bp) and (528.64bp,386.73bp)  .. (node_62);
  \draw (548.0bp,392.0bp) node {$1$};
  \draw [blue,->] (node_39) ..controls (358.84bp,337.87bp) and (319.98bp,314.0bp)  .. (node_60);
  \draw (355.0bp,322.0bp) node {$1$};
  \draw [red,->] (node_24) ..controls (341.49bp,408.45bp) and (362.35bp,386.32bp)  .. (node_39);
  \draw (373.0bp,392.0bp) node {$2$};
  \draw [red,->] (node_16) ..controls (1507.0bp,268.25bp) and (1479.0bp,245.5bp)  .. (node_56);
  \draw (1506.0bp,252.0bp) node {$2$};
  \draw [blue,->] (node_35) ..controls (386.42bp,269.08bp) and (382.35bp,248.75bp)  .. (node_18);
  \draw (393.0bp,252.0bp) node {$1$};
  \draw [red,->] (node_15) ..controls (732.0bp,269.19bp) and (732.0bp,249.15bp)  .. (node_3);
  \draw (741.0bp,252.0bp) node {$2$};
  \draw [blue,->] (node_24) ..controls (295.38bp,414.96bp) and (279.2bp,408.11bp)  .. (265.0bp,402.0bp) .. controls (238.5bp,390.6bp) and (208.27bp,377.36bp)  .. (node_41);
  \draw (274.0bp,392.0bp) node {$1$};
  \draw [blue,->] (node_17) ..controls (846.0bp,409.19bp) and (846.0bp,389.15bp)  .. (node_57);
  \draw (855.0bp,392.0bp) node {$1$};
  \draw [blue,->] (node_6) ..controls (846.0bp,269.19bp) and (846.0bp,249.15bp)  .. (node_5);
  \draw (855.0bp,252.0bp) node {$1$};
  \draw [blue,->] (node_48) ..controls (618.0bp,269.19bp) and (618.0bp,249.15bp)  .. (node_2);
  \draw (627.0bp,252.0bp) node {$1$};
  \draw [blue,->] (node_29) ..controls (172.16bp,197.87bp) and (211.02bp,174.0bp)  .. (node_54);
  \draw (220.0bp,182.0bp) node {$1$};
  \draw [red,->] (node_43) ..controls (458.84bp,197.87bp) and (419.98bp,174.0bp)  .. (node_61);
  \draw (455.0bp,182.0bp) node {$2$};
  \draw [red,->] (node_60) ..controls (303.1bp,268.03bp) and (336.52bp,244.64bp)  .. (node_18);
  \draw (346.0bp,252.0bp) node {$2$};
  \draw [blue,->] (node_42) ..controls (1136.2bp,408.25bp) and (1108.6bp,385.5bp)  .. (node_7);
  \draw (1135.0bp,392.0bp) node {$1$};
  \draw [red,->] (node_42) ..controls (1166.5bp,408.98bp) and (1175.0bp,388.35bp)  .. (node_44);
  \draw (1185.0bp,392.0bp) node {$2$};
  \draw [red,->] (node_27) ..controls (73.137bp,268.08bp) and (104.01bp,244.85bp)  .. (node_29);
  \draw (115.0bp,252.0bp) node {$2$};
  \draw [blue,->] (node_36) ..controls (172.16bp,127.87bp) and (211.02bp,104.0bp)  .. (node_9);
  \draw (220.0bp,112.0bp) node {$1$};
  \draw [blue,->] (node_4) ..controls (500.42bp,269.08bp) and (496.35bp,248.75bp)  .. (node_43);
  \draw (507.0bp,252.0bp) node {$1$};
  \draw [red,->] (node_25) ..controls (742.85bp,477.85bp) and (792.72bp,453.27bp)  .. (node_17);
  \draw (800.0bp,462.0bp) node {$2$};
  \draw [blue,->] (node_30) ..controls (287.17bp,477.9bp) and (240.07bp,453.48bp)  .. (node_33);
  \draw (281.0bp,462.0bp) node {$1$};
  \draw [red,->] (node_23) ..controls (1497.5bp,338.77bp) and (1510.8bp,317.54bp)  .. (node_16);
  \draw (1522.0bp,322.0bp) node {$2$};
  \draw [blue,->] (node_43) ..controls (490.0bp,199.19bp) and (490.0bp,179.15bp)  .. (node_28);
  \draw (499.0bp,182.0bp) node {$1$};
  \draw [red,->] (node_33) ..controls (182.05bp,408.98bp) and (174.09bp,388.35bp)  .. (node_41);
  \draw (187.0bp,392.0bp) node {$2$};
  \draw [blue,->] (node_45) ..controls (960.0bp,269.19bp) and (960.0bp,249.15bp)  .. (node_34);
  \draw (969.0bp,252.0bp) node {$1$};
  \draw [darkgreen,thick,->] (node_59) edgenode[above]{$-1$} (node_25);
  \draw [red,->] (node_29) ..controls (141.0bp,199.19bp) and (141.0bp,179.15bp)  .. (node_36);
  \draw (150.0bp,182.0bp) node {$2$};
  \draw [red,->] (node_8) ..controls (1074.0bp,129.19bp) and (1074.0bp,109.15bp)  .. (node_21);
  \draw (1083.0bp,112.0bp) node {$2$};
  \draw [blue,->] (node_13) ..controls (1302.0bp,269.19bp) and (1302.0bp,249.15bp)  .. (node_51);
  \draw (1311.0bp,252.0bp) node {$1$};
  \draw [blue,->] (node_62) ..controls (504.0bp,339.19bp) and (504.0bp,319.15bp)  .. (node_4);
  \draw (513.0bp,322.0bp) node {$1$};
  \draw [red,->] (node_10) ..controls (276.0bp,339.19bp) and (276.0bp,319.15bp)  .. (node_60);
  \draw (285.0bp,322.0bp) node {$2$};
  \draw [blue,->] (node_31) ..controls (1074.0bp,199.19bp) and (1074.0bp,179.15bp)  .. (node_8);
  \draw (1083.0bp,182.0bp) node {$1$};
  \draw [red,->] (node_30) ..controls (324.0bp,479.19bp) and (324.0bp,459.15bp)  .. (node_24);
  \draw (333.0bp,462.0bp) node {$2$};
  \draw [blue,->] (node_60) ..controls (270.63bp,269.08bp) and (264.52bp,248.75bp)  .. (node_12);
  \draw (276.0bp,252.0bp) node {$1$};
  \draw [blue,->] (node_44) ..controls (1150.6bp,345.72bp) and (1135.4bp,339.71bp)  .. (1123.0bp,332.0bp) .. controls (1109.6bp,323.71bp) and (1096.6bp,311.52bp)  .. (node_1);
  \draw (1132.0bp,322.0bp) node {$1$};
  \draw [blue,->] (node_9) ..controls (273.55bp,58.455bp) and (295.68bp,36.324bp)  .. (node_53);
  \draw (307.0bp,42.0bp) node {$1$};
  \draw [red,->] (node_14) ..controls (680.96bp,408.25bp) and (653.01bp,385.5bp)  .. (node_19);
  \draw (680.0bp,392.0bp) node {$2$};
  \draw [red,->] (node_52) ..controls (1352.0bp,409.19bp) and (1352.0bp,389.15bp)  .. (node_49);
  \draw (1361.0bp,392.0bp) node {$2$};
  \draw [darkgreen,thick,->] (node_36) edgenode[above]{$-1$} (node_21);
  \draw [darkgreen,thick,->,dashed] (node_45) edgenode[above]{$-1$} (node_3);
  \draw [darkgreen,thick,->,dashed] (node_34) edgenode[above]{$-1$} (node_38);
  \draw [purple,thick,->,dashed] (node_45) edgenode[above]{$-1$} (node_56);
  \draw [purple,thick,->,dashed] (node_34) edgenode[above]{$-1$} (node_46);
\end{tikzpicture}}}
\caption{Counterexample to the unique characterization of the local queer axioms of 
Definition~\ref{definition.local queer axioms}.
\label{figure.counterexample}}
\end{figure}
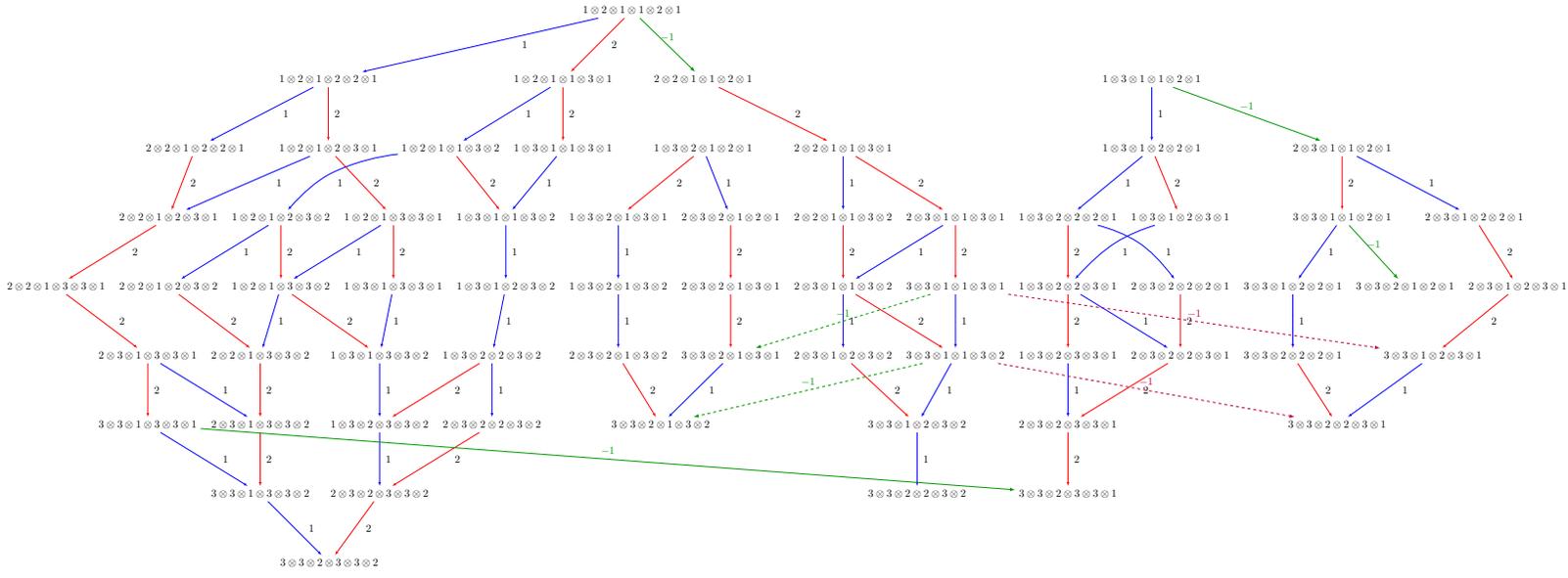

The problem with Axiom \textbf{LQ5} illustrated in Figure~\ref{figure.local queer} is that the $(-1)$-arrow at the bottom
of the 2-strings is not closed at the top. Hence, as demonstrated by the counterexample in Figure~\ref{figure.counterexample}
switching components with the same $I_0$-highest weights can cause non-uniqueness. 
In fact, if $f_{-1}b$ is determined for all $b \in \mathcal{C}$ such that
\begin{equation}
\label{equation.almost lowest}
	\varphi_i(b)=0 \quad \text{for all $i\in I_0 \setminus \{1\}$ and} \quad 
	\varphi_1(b)=2,
\end{equation}
then, by the relations between $f_{-1}$ and $f_i$ for $i\in I_0$ of Definition~\ref{definition.local queer axioms},
$f_{-1}$ is determined on all elements in $\mathcal{C}$. 

\begin{lemma}
\label{lemma.gjk}
Let $v \in \mathcal{B}^{\otimes \ell}$ be an $I_0$-lowest weight element, that is, $\varphi_i(v)=0$ for all $i\in I_0$.
Then every $b \in \mathcal{B}^{\otimes \ell}$ satisfying~\eqref{equation.almost lowest} is of the form
\begin{equation}
\label{equation.precise almost lowest}
	g_{j,k} := (e_1 \cdots e_j) (e_1 \cdots e_k) v \qquad \text{for some $1\leqslant j \leqslant k  \leqslant n$.}
\end{equation}
Conversely, every $g_{j,k} \neq 0$ with $1\leqslant j \leqslant k  \leqslant n$ satisfies~\eqref{equation.almost lowest}.
\end{lemma}

\begin{proof}
The statement of the lemma is a statement about type $A_n$ crystals and hence can be verified by
the tableaux model for type $A_n$ crystals (see for example~\cite{BumpSchilling.2017}). The element $v$
is $I_0$-lowest weight and hence as a tableau in French notation contains the letter $n+1$ at the top of each column,
the letter $n$ in the second to top box in each column, and in general the letter $n+2-i$ in the $i$-th box from
the top in its column. If there is a letter $k+1$ in the first row of $v$, then $(e_1 \cdots e_k)$ applies to $v$
and $b'=(e_1 \cdots e_k)v$ satisfies $\varphi_i(b')=0$ for $i\in I_0 \setminus \{1\}$ and $\varphi_1(b')=1$.
The element $b'$ has several changed entries in the first row, and otherwise the entries above the first row
all have letter $n+2-i$ in the $i$-th box from the top in their column.   If $b'$ has a letter $j+1$ in the first row with 
$1\leqslant j \leqslant k$, then $(e_1 \cdots e_j)$ applies to $b'$ and $b=g_{j,k} = (e_1 \cdots e_j)b'$ 
satisfies~\eqref{equation.almost lowest}.  Note that if $j>k$, then the last $e_1$ would no longer apply and hence $b=0$. 
This proves that $g_{j,k} \neq 0$ as in~\eqref{equation.precise almost lowest} satisfies~\eqref{equation.almost lowest}. If 
conversely $b$ satisfies~\eqref{equation.almost lowest}, then as a tableau it contains two extra $1$'s in the first row that 
have a $3$ or bigger above them rather than a $2$ in their columns, and for entries higher than the first row the $i$-th box 
from the top in its column contains $n+2-i$. It is not hard to check that then $(f_k \cdots f_1)(f_j\cdots f_1)b=v$ for 
some $1\leqslant j\leqslant k \leqslant n$.  Hence $b$ is of the form~\eqref{equation.precise almost lowest}.
\end{proof}

In the next section, we introduce a new graph just on $I_0$-highest weight elements and 
new connectivity axioms (see Definition~\ref{definition.connectivity axioms}) that
uniquely characterizes queer crystals (see Theorem~\ref{theorem.main}).  

\section{Graph on type $A$ components}
\label{section.G(C)}

Let $\mathcal{C}$ be a crystal with index set $I_0 \cup \{-1\}$ that is a Stembridge crystal of type $A_n$ when
restricted to the arrows labeled $I_0$. In this section, we define a graph for $\mathcal{C}$ labeled by the type $A_n$ 
components of $\mathcal{C}$. We draw an edge from vertex $C_1$ to vertex $C_2$ in this graph if there is an element 
$b_1$ in the component $C_1$ and an element $b_2$ in the component $C_2$ such that $f_{-1} b_1 = b_2$.
We provide an easy combinatorial way to describe this graph for a queer crystal leveraging the explicit actions of $f_{-i}$ 
described in Theorem~\ref{theorem.f-i} and $e_{-i}$ described in Theorem~\ref{theorem.e-i}, respectively
(see Theorem~\ref{theorem.combinatorial G}). We also provide new axioms in 
Definition~\ref{definition.connectivity axioms} that will be used in Section~\ref{section.characterization} to provide a 
unique characterization of queer crystals.

\begin{definition}
Let $\mathcal{C}$ be a crystal with index set $I_0 \cup \{-1\}$ that is a Stembridge crystal of type $A_n$ when
restricted to the arrows labeled $I_0$. We define the \defn{component graph} of $\mathcal{C}$, denoted by 
\defn{$G(\mathcal{C})$}, as follows. The vertices of $G(\mathcal{C})$ are the type $A_n$ components of $\mathcal{C}$ 
(typically labeled by their highest weight elements). There is an edge from vertex $C_1$ to vertex $C_2$ in this graph, if 
there is an element $b_1$ in the component $C_1$ and an element $b_2$ in the component $C_2$ such that
\[
	f_{-1} b_1 = b_2.
\]
\end{definition}

\begin{example}
\label{example.G counterexample}
Let $\mathcal{C}$ be the connected component in the $\mathfrak{q}(3)$-crystal $\mathcal{B}^{\otimes 6}$ with highest 
weight element $1\otimes 2 \otimes 1 \otimes 1 \otimes 2 \otimes 1$ of highest weight $(4,2,0)$. The graph $G(\mathcal{C})$
is given in Figure~\ref{figure.queer 12 graph} on the left (disregarding the labels on the edges).
The graph $G(\mathcal{C}')$ for the counterexample $\mathcal{C}'$ in Figure~\ref{figure.counterexample}
is given in Figure~\ref{figure.queer 12 graph} on the right. Since the two graphs are not isomorphic as unlabeled graphs, 
this confirms that the purple dashed arrows in Figure~\ref{figure.counterexample} do not give the queer crystal even 
though the induced crystal satisfies the axioms in Definition~\ref{definition.local queer axioms}.
\end{example}

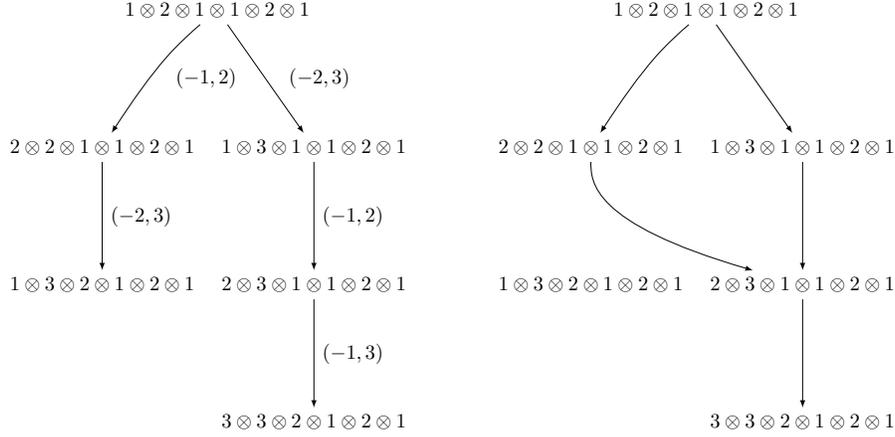
\begin{figure}
\scalebox{0.7}{
\begin{tikzpicture}[>=latex,line join=bevel,]
\node (node_5) at (162.0bp,7.0bp) [draw,draw=none] {$3 \otimes 3 \otimes 2 \otimes 1 \otimes 2 \otimes 1$};
  \node (node_4) at (162.0bp,81.0bp) [draw,draw=none] {$2 \otimes 3 \otimes 1 \otimes 1 \otimes 2 \otimes 1$};
  \node (node_3) at (48.0bp,155.0bp) [draw,draw=none] {$2 \otimes 2 \otimes 1 \otimes 1 \otimes 2 \otimes 1$};
  \node (node_2) at (48.0bp,81.0bp) [draw,draw=none] {$1 \otimes 3 \otimes 2 \otimes 1 \otimes 2 \otimes 1$};
  \node (node_1) at (162.0bp,155.0bp) [draw,draw=none] {$1 \otimes 3 \otimes 1 \otimes 1 \otimes 2 \otimes 1$};
  \node (node_0) at (110.0bp,229.0bp) [draw,draw=none] {$1 \otimes 2 \otimes 1 \otimes 1 \otimes 2 \otimes 1$};
  \draw [black,->] (node_3) ..controls (48.0bp,136.61bp) and (48.0bp,114.21bp)  .. (node_2);
  \definecolor{strokecol}{rgb}{0.0,0.0,0.0};
  \pgfsetstrokecolor{strokecol}
  \draw (69.0bp,118.0bp) node {$\left(-2, 3\right)$};
  \draw [black,->] (node_0) ..controls (96.003bp,217.17bp) and (88.742bp,210.52bp)  .. (83.0bp,204.0bp) .. controls (73.785bp,193.54bp) and (64.671bp,180.68bp)  .. (node_3);
  \draw (104.0bp,192.0bp) node {$\left(-1, 2\right)$};
  \draw [black,->] (node_4) ..controls (162.0bp,62.61bp) and (162.0bp,40.211bp)  .. (node_5);
  \draw (183.0bp,44.0bp) node {$\left(-1, 3\right)$};
  \draw [black,->] (node_1) ..controls (162.0bp,136.61bp) and (162.0bp,114.21bp)  .. (node_4);
  \draw (183.0bp,118.0bp) node {$\left(-1, 2\right)$};
  \draw [black,->] (node_0) ..controls (123.23bp,210.17bp) and (139.9bp,186.45bp)  .. (node_1);
  \draw (165.0bp,192.0bp) node {$\left(-2, 3\right)$};
\end{tikzpicture}}
\hspace{0.7cm}
\scalebox{0.7}{
\begin{tikzpicture}[>=latex,line join=bevel,]
\node (node_5) at (162.0bp,7.0bp) [draw,draw=none] {$3 \otimes 3 \otimes 2 \otimes 1 \otimes 2 \otimes 1$};
  \node (node_4) at (162.0bp,81.0bp) [draw,draw=none] {$2 \otimes 3 \otimes 1 \otimes 1 \otimes 2 \otimes 1$};
  \node (node_3) at (48.0bp,155.0bp) [draw,draw=none] {$2 \otimes 2 \otimes 1 \otimes 1 \otimes 2 \otimes 1$};
  \node (node_2) at (48.0bp,81.0bp) [draw,draw=none] {$1 \otimes 3 \otimes 2 \otimes 1 \otimes 2 \otimes 1$};
  \node (node_1) at (162.0bp,155.0bp) [draw,draw=none] {$1 \otimes 3 \otimes 1 \otimes 1 \otimes 2 \otimes 1$};
  \node (node_0) at (110.0bp,229.0bp) [draw,draw=none] {$1 \otimes 2 \otimes 1 \otimes 1 \otimes 2 \otimes 1$};
  \draw [black,->] (node_3) ..controls (48.0bp,136.61bp) and (48.0bp,114.21bp)  .. (node_4);
  \definecolor{strokecol}{rgb}{0.0,0.0,0.0};
  \pgfsetstrokecolor{strokecol}
  \draw [black,->] (node_0) ..controls (96.003bp,217.17bp) and (88.742bp,210.52bp)  .. (83.0bp,204.0bp) .. controls (73.785bp,193.54bp) and (64.671bp,180.68bp)  .. (node_3);
  \draw [black,->] (node_4) ..controls (162.0bp,62.61bp) and (162.0bp,40.211bp)  .. (node_5);
  \draw [black,->] (node_1) ..controls (162.0bp,136.61bp) and (162.0bp,114.21bp)  .. (node_4);
  \draw [black,->] (node_0) ..controls (123.23bp,210.17bp) and (139.9bp,186.45bp)  .. (node_1);
\end{tikzpicture}}
\caption{\textbf{Left:} $\overline{G}(\mathcal{C})$. The graph $G(\mathcal{C})$ is obtained from $\overline{G}(\mathcal{C})$
by removing the labels. \textbf{Right:} $G(\mathcal{C}')$ for the crystals of 
Example~\ref{example.G counterexample}.
\label{figure.queer 12 graph}} 
\end{figure}

\begin{example}
\label{example.G(C)}
Let $\mathcal{C}$ be the connected component with highest weight element
$1 \otimes 1 \otimes 2 \otimes 2 \otimes 1 \otimes 3 \otimes 2 \otimes 1$ in the $\mathfrak{q}(4)$-crystal 
$\mathcal{B}^{\otimes 9}$. Then the graph $G(\mathcal{C})$ is given in Figure~\ref{figure.G(C)}.
One may easily check using Theorem~\ref{theorem.f-i} that all arrows in Figure~\ref{figure.G(C)} are given
by the application of $f_{-i}$ for some $i$ except for the arrows that by-pass other arrows and the arrow
to the lowest vertex, which is given by $f_{-2} f_3$ (which is also determined by Theorem~\ref{theorem.f-i}).
The result is shown in Figure~\ref{figure.G(C) labeled}.
\end{example}

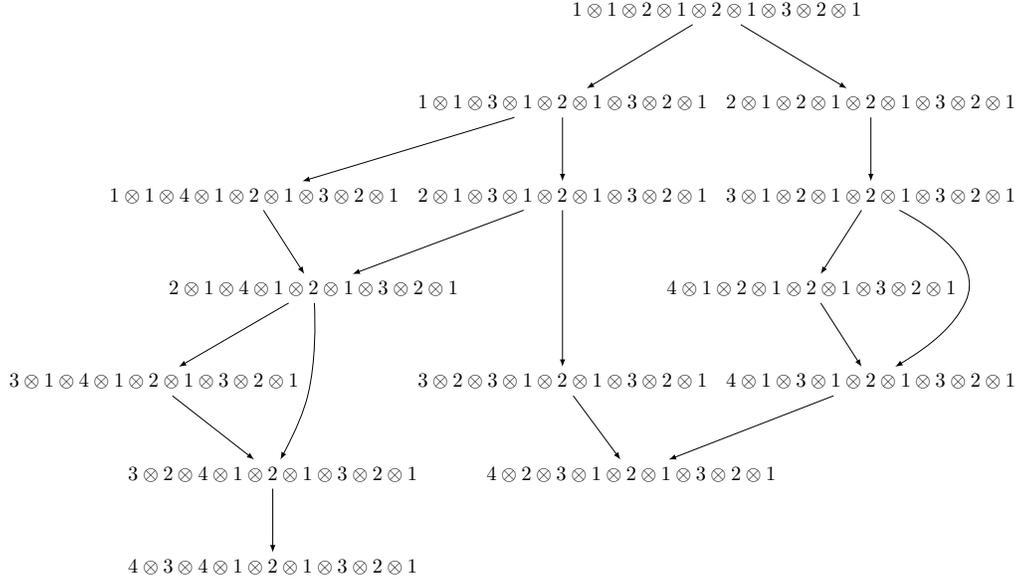
\begin{figure}
\scalebox{0.7}{
\begin{tikzpicture}[>=latex,line join=bevel,]
\node (node_13) at (138.0bp,7.0bp) [draw,draw=none] {$4 \otimes 3 \otimes 4 \otimes 1 \otimes 2 \otimes 1 \otimes 3 \otimes 2 \otimes 1$};
  \node (node_2) at (128.0bp,207.0bp) [draw,draw=none] {$1 \otimes 1 \otimes 4 \otimes 1 \otimes 2 \otimes 1 \otimes 3 \otimes 2 \otimes 1$};
  \node (node_9) at (138.0bp,57.0bp) [draw,draw=none] {$3 \otimes 2 \otimes 4 \otimes 1 \otimes 2 \otimes 1 \otimes 3 \otimes 2 \otimes 1$};
  \node (node_8) at (294.0bp,107.0bp) [draw,draw=none] {$3 \otimes 2 \otimes 3 \otimes 1 \otimes 2 \otimes 1 \otimes 3 \otimes 2 \otimes 1$};
  \node (node_7) at (74.0bp,107.0bp) [draw,draw=none] {$3 \otimes 1 \otimes 4 \otimes 1 \otimes 2 \otimes 1 \otimes 3 \otimes 2 \otimes 1$};
  \node (node_6) at (460.0bp,207.0bp) [draw,draw=none] {$3 \otimes 1 \otimes 2 \otimes 1 \otimes 2 \otimes 1 \otimes 3 \otimes 2 \otimes 1$};
  \node (node_5) at (160.0bp,157.0bp) [draw,draw=none] {$2 \otimes 1 \otimes 4 \otimes 1 \otimes 2 \otimes 1 \otimes 3 \otimes 2 \otimes 1$};
  \node (node_4) at (294.0bp,207.0bp) [draw,draw=none] {$2 \otimes 1 \otimes 3 \otimes 1 \otimes 2 \otimes 1 \otimes 3 \otimes 2 \otimes 1$};
  \node (node_3) at (460.0bp,257.0bp) [draw,draw=none] {$2 \otimes 1 \otimes 2 \otimes 1 \otimes 2 \otimes 1 \otimes 3 \otimes 2 \otimes 1$};
  \node (node_12) at (331.0bp,57.0bp) [draw,draw=none] {$4 \otimes 2 \otimes 3 \otimes 1 \otimes 2 \otimes 1 \otimes 3 \otimes 2 \otimes 1$};
  \node (node_1) at (294.0bp,257.0bp) [draw,draw=none] {$1 \otimes 1 \otimes 3 \otimes 1 \otimes 2 \otimes 1 \otimes 3 \otimes 2 \otimes 1$};
  \node (node_0) at (377.0bp,307.0bp) [draw,draw=none] {$1 \otimes 1 \otimes 2 \otimes 1 \otimes 2 \otimes 1 \otimes 3 \otimes 2 \otimes 1$};
  \node (node_11) at (460.0bp,107.0bp) [draw,draw=none] {$4 \otimes 1 \otimes 3 \otimes 1 \otimes 2 \otimes 1 \otimes 3 \otimes 2 \otimes 1$};
  \node (node_10) at (428.0bp,157.0bp) [draw,draw=none] {$4 \otimes 1 \otimes 2 \otimes 1 \otimes 2 \otimes 1 \otimes 3 \otimes 2 \otimes 1$};
  \draw [black,->] (node_0) ..controls (402.08bp,291.89bp) and (423.6bp,278.93bp)  .. (node_3);
  \draw [black,->] (node_4) ..controls (294.0bp,184.17bp) and (294.0bp,146.68bp)  .. (node_8);
  \draw [black,->] (node_4) ..controls (252.54bp,191.53bp) and (214.64bp,177.39bp)  .. (node_5);
  \draw [black,->] (node_8) ..controls (304.77bp,92.451bp) and (312.99bp,81.343bp)  .. (node_12);
  \draw [black,->] (node_6) ..controls (493.05bp,190.19bp) and (520.85bp,171.19bp)  .. (511.0bp,150.0bp) .. controls (505.05bp,137.2bp) and (493.21bp,126.84bp)  .. (node_11);
  \draw [black,->] (node_5) ..controls (134.01bp,141.89bp) and (111.72bp,128.93bp)  .. (node_7);
  \draw [black,->] (node_3) ..controls (460.0bp,242.97bp) and (460.0bp,232.94bp)  .. (node_6);
  \draw [black,->] (node_5) ..controls (160.94bp,138.95bp) and (161.19bp,117.54bp)  .. (157.0bp,100.0bp) .. controls (154.79bp,90.727bp) and (150.64bp,80.963bp)  .. (node_9);
  \draw [black,->] (node_1) ..controls (242.02bp,241.34bp) and (193.6bp,226.76bp)  .. (node_2);
  \draw [black,->] (node_1) ..controls (294.0bp,242.97bp) and (294.0bp,232.94bp)  .. (node_4);
  \draw [black,->] (node_11) ..controls (420.09bp,91.53bp) and (383.6bp,77.388bp)  .. (node_12);
  \draw [black,->] (node_6) ..controls (450.74bp,192.53bp) and (443.73bp,181.57bp)  .. (node_10);
  \draw [black,->] (node_9) ..controls (138.0bp,42.974bp) and (138.0bp,32.942bp)  .. (node_13);
  \draw [black,->] (node_7) ..controls (92.96bp,92.188bp) and (108.7bp,79.892bp)  .. (node_9);
  \draw [black,->] (node_10) ..controls (437.26bp,142.53bp) and (444.27bp,131.57bp)  .. (node_11);
  \draw [black,->] (node_2) ..controls (137.26bp,192.53bp) and (144.27bp,181.57bp)  .. (node_5);
  \draw [black,->] (node_0) ..controls (351.92bp,291.89bp) and (330.4bp,278.93bp)  .. (node_1);
\end{tikzpicture}}
\caption{The graph $G(\mathcal{C})$ for Example~\ref{example.G(C)}.
\label{figure.G(C)}}
\end{figure}

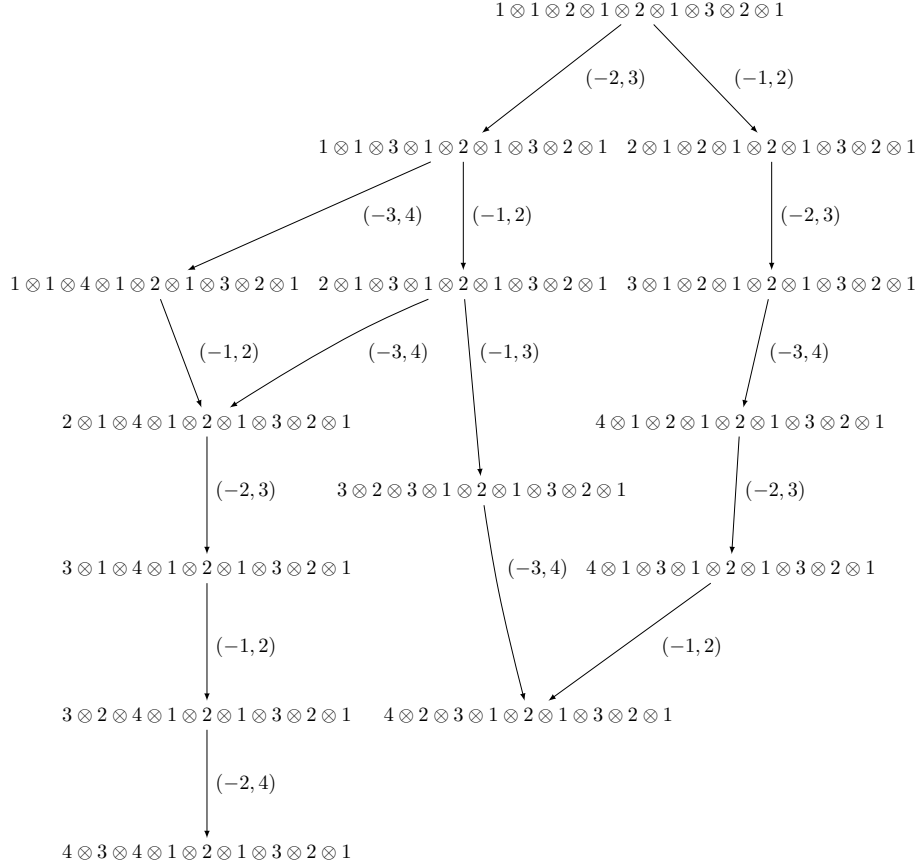
\begin{figure}
\scalebox{0.7}{
\begin{tikzpicture}[>=latex,line join=bevel,]
\node (node_13) at (102.0bp,7.0bp) [draw,draw=none] {$4 \otimes 3 \otimes 4 \otimes 1 \otimes 2 \otimes 1 \otimes 3 \otimes 2 \otimes 1$};
  \node (node_2) at (74.0bp,313.0bp) [draw,draw=none] {$1 \otimes 1 \otimes 4 \otimes 1 \otimes 2 \otimes 1 \otimes 3 \otimes 2 \otimes 1$};
  \node (node_9) at (102.0bp,81.0bp) [draw,draw=none] {$3 \otimes 2 \otimes 4 \otimes 1 \otimes 2 \otimes 1 \otimes 3 \otimes 2 \otimes 1$};
  \node (node_8) at (250.0bp,202.0bp) [draw,draw=none] {$3 \otimes 2 \otimes 3 \otimes 1 \otimes 2 \otimes 1 \otimes 3 \otimes 2 \otimes 1$};
  \node (node_7) at (102.0bp,160.0bp) [draw,draw=none] {$3 \otimes 1 \otimes 4 \otimes 1 \otimes 2 \otimes 1 \otimes 3 \otimes 2 \otimes 1$};
  \node (node_6) at (406.0bp,313.0bp) [draw,draw=none] {$3 \otimes 1 \otimes 2 \otimes 1 \otimes 2 \otimes 1 \otimes 3 \otimes 2 \otimes 1$};
  \node (node_5) at (102.0bp,239.0bp) [draw,draw=none] {$2 \otimes 1 \otimes 4 \otimes 1 \otimes 2 \otimes 1 \otimes 3 \otimes 2 \otimes 1$};
  \node (node_4) at (240.0bp,313.0bp) [draw,draw=none] {$2 \otimes 1 \otimes 3 \otimes 1 \otimes 2 \otimes 1 \otimes 3 \otimes 2 \otimes 1$};
  \node (node_3) at (406.0bp,387.0bp) [draw,draw=none] {$2 \otimes 1 \otimes 2 \otimes 1 \otimes 2 \otimes 1 \otimes 3 \otimes 2 \otimes 1$};
  \node (node_12) at (275.0bp,81.0bp) [draw,draw=none] {$4 \otimes 2 \otimes 3 \otimes 1 \otimes 2 \otimes 1 \otimes 3 \otimes 2 \otimes 1$};
  \node (node_1) at (240.0bp,387.0bp) [draw,draw=none] {$1 \otimes 1 \otimes 3 \otimes 1 \otimes 2 \otimes 1 \otimes 3 \otimes 2 \otimes 1$};
  \node (node_0) at (335.0bp,461.0bp) [draw,draw=none] {$1 \otimes 1 \otimes 2 \otimes 1 \otimes 2 \otimes 1 \otimes 3 \otimes 2 \otimes 1$};
  \node (node_11) at (384.0bp,160.0bp) [draw,draw=none] {$4 \otimes 1 \otimes 3 \otimes 1 \otimes 2 \otimes 1 \otimes 3 \otimes 2 \otimes 1$};
  \node (node_10) at (389.0bp,239.0bp) [draw,draw=none] {$4 \otimes 1 \otimes 2 \otimes 1 \otimes 2 \otimes 1 \otimes 3 \otimes 2 \otimes 1$};
  \draw [black,->] (node_0) ..controls (353.28bp,441.95bp) and (376.67bp,417.56bp)  .. (node_3);
  \definecolor{strokecol}{rgb}{0.0,0.0,0.0};
  \pgfsetstrokecolor{strokecol}
  \draw (402.0bp,424.0bp) node {$\left(-1, 2\right)$};
  \draw [black,->] (node_4) ..controls (242.22bp,288.31bp) and (246.22bp,243.92bp)  .. (node_8);
  \draw (265.0bp,276.0bp) node {$\left(-1, 3\right)$};
  \draw [black,->] (node_4) ..controls (211.63bp,301.22bp) and (196.75bp,294.63bp)  .. (184.0bp,288.0bp) .. controls (162.1bp,276.62bp) and (137.96bp,261.92bp)  .. (node_5);
  \draw (205.0bp,276.0bp) node {$\left(-3, 4\right)$};
  \draw [black,->] (node_5) ..controls (102.0bp,219.84bp) and (102.0bp,194.53bp)  .. (node_7);
  \draw (123.0bp,202.0bp) node {$\left(-2, 3\right)$};
  \draw [black,->] (node_3) ..controls (406.0bp,368.61bp) and (406.0bp,346.21bp)  .. (node_6);
  \draw (427.0bp,350.0bp) node {$\left(-2, 3\right)$};
  \draw [black,->] (node_8) ..controls (252.53bp,184.65bp) and (255.6bp,164.77bp)  .. (259.0bp,148.0bp) .. controls (262.49bp,130.78bp) and (267.24bp,111.27bp)  .. (node_12);
  \draw (280.0bp,160.0bp) node {$\left(-3, 4\right)$};
  \draw [black,->] (node_1) ..controls (195.28bp,367.06bp) and (134.28bp,339.87bp)  .. (node_2);
  \draw (202.0bp,350.0bp) node {$\left(-3, 4\right)$};
  \draw [black,->] (node_1) ..controls (240.0bp,368.61bp) and (240.0bp,346.21bp)  .. (node_4);
  \draw (261.0bp,350.0bp) node {$\left(-1, 2\right)$};
  \draw [black,->] (node_11) ..controls (356.11bp,139.79bp) and (316.46bp,111.05bp)  .. (node_12);
  \draw (363.0bp,118.0bp) node {$\left(-1, 2\right)$};
  \draw [black,->] (node_6) ..controls (401.78bp,294.61bp) and (396.63bp,272.21bp)  .. (node_10);
  \draw (421.0bp,276.0bp) node {$\left(-3, 4\right)$};
  \draw [black,->] (node_9) ..controls (102.0bp,62.61bp) and (102.0bp,40.211bp)  .. (node_13);
  \draw (123.0bp,44.0bp) node {$\left(-2, 4\right)$};
  \draw [black,->] (node_7) ..controls (102.0bp,140.84bp) and (102.0bp,115.53bp)  .. (node_9);
  \draw (123.0bp,118.0bp) node {$\left(-1, 2\right)$};
  \draw [black,->] (node_10) ..controls (387.79bp,219.84bp) and (386.19bp,194.53bp)  .. (node_11);
  \draw (408.0bp,202.0bp) node {$\left(-2, 3\right)$};
  \draw [black,->] (node_2) ..controls (81.0bp,294.5bp) and (89.599bp,271.78bp)  .. (node_5);
  \draw (114.0bp,276.0bp) node {$\left(-1, 2\right)$};
  \draw [black,->] (node_0) ..controls (310.11bp,441.61bp) and (277.46bp,416.18bp)  .. (node_1);
  \draw (322.0bp,424.0bp) node {$\left(-2, 3\right)$};
\end{tikzpicture}}
\caption{The graph $\overline{G}(\mathcal{C})$ of Figure~\ref{figure.G(C)} obtained from $G(\mathcal{C})$ by labeling
each edge (except for the by-pass edges) by $(-i,h)$ if $f_{(-i,h)}$ applies.
\label{figure.G(C) labeled}}
\end{figure}

\begin{figure}
\scalebox{0.7}{
\begin{tikzpicture}[>=latex,line join=bevel,]
\node (node_13) at (153.0bp,7.0bp) [draw,draw=none] {$4 \otimes 3 \otimes 4 \otimes 1 \otimes 2 \otimes 1 \otimes 3 \otimes 2 \otimes 1$};
  \node (node_2) at (101.0bp,299.0bp) [draw,draw=none] {$1 \otimes 1 \otimes 4 \otimes 1 \otimes 2 \otimes 1 \otimes 3 \otimes 2 \otimes 1$};
  \node (node_9) at (153.0bp,77.0bp) [draw,draw=none] {$3 \otimes 2 \otimes 4 \otimes 1 \otimes 2 \otimes 1 \otimes 3 \otimes 2 \otimes 1$};
  \node (node_8) at (267.0bp,226.0bp) [draw,draw=none] {$3 \otimes 2 \otimes 3 \otimes 1 \otimes 2 \otimes 1 \otimes 3 \otimes 2 \otimes 1$};
  \node (node_7) at (74.0bp,150.0bp) [draw,draw=none] {$3 \otimes 1 \otimes 4 \otimes 1 \otimes 2 \otimes 1 \otimes 3 \otimes 2 \otimes 1$};
  \node (node_6) at (433.0bp,299.0bp) [draw,draw=none] {$3 \otimes 1 \otimes 2 \otimes 1 \otimes 2 \otimes 1 \otimes 3 \otimes 2 \otimes 1$};
  \node (node_5) at (101.0bp,226.0bp) [draw,draw=none] {$2 \otimes 1 \otimes 4 \otimes 1 \otimes 2 \otimes 1 \otimes 3 \otimes 2 \otimes 1$};
  \node (node_4) at (267.0bp,299.0bp) [draw,draw=none] {$2 \otimes 1 \otimes 3 \otimes 1 \otimes 2 \otimes 1 \otimes 3 \otimes 2 \otimes 1$};
  \node (node_3) at (433.0bp,369.0bp) [draw,draw=none] {$2 \otimes 1 \otimes 2 \otimes 1 \otimes 2 \otimes 1 \otimes 3 \otimes 2 \otimes 1$};
  \node (node_12) at (326.0bp,77.0bp) [draw,draw=none] {$4 \otimes 2 \otimes 3 \otimes 1 \otimes 2 \otimes 1 \otimes 3 \otimes 2 \otimes 1$};
  \node (node_1) at (267.0bp,369.0bp) [draw,draw=none] {$1 \otimes 1 \otimes 3 \otimes 1 \otimes 2 \otimes 1 \otimes 3 \otimes 2 \otimes 1$};
  \node (node_0) at (364.0bp,439.0bp) [draw,draw=none] {$1 \otimes 1 \otimes 2 \otimes 1 \otimes 2 \otimes 1 \otimes 3 \otimes 2 \otimes 1$};
  \node (node_11) at (433.0bp,150.0bp) [draw,draw=none] {$4 \otimes 1 \otimes 3 \otimes 1 \otimes 2 \otimes 1 \otimes 3 \otimes 2 \otimes 1$};
  \node (node_10) at (433.0bp,226.0bp) [draw,draw=none] {$4 \otimes 1 \otimes 2 \otimes 1 \otimes 2 \otimes 1 \otimes 3 \otimes 2 \otimes 1$};
  \draw [blue,->] (node_0) ..controls (382.28bp,420.45bp) and (404.1bp,398.32bp)  .. (node_3);
  \definecolor{strokecol}{rgb}{0.0,0.0,0.0};
  \pgfsetstrokecolor{strokecol}
  \draw (418.5bp,404.0bp) node {$-1$};
  \draw [red,->] (node_4) ..controls (252.49bp,287.28bp) and (246.74bp,281.03bp)  .. (244.0bp,274.0bp) .. controls (240.77bp,265.72bp) and (241.0bp,262.37bp)  .. (244.0bp,254.0bp) .. controls (245.67bp,249.34bp) and (248.48bp,244.91bp)  .. (node_8);
  \draw (256.5bp,264.0bp) node {$-2$};
  \draw [blue,->] (node_4) ..controls (268.21bp,286.6bp) and (268.76bp,279.92bp)  .. (269.0bp,274.0bp) .. controls (269.35bp,265.12bp) and (269.32bp,262.88bp)  .. (269.0bp,254.0bp) .. controls (268.88bp,250.6bp) and (268.68bp,246.98bp)  .. (node_8);
  \draw (281.5bp,264.0bp) node {$-1$};
  \draw [green,->] (node_4) ..controls (221.82bp,279.13bp) and (161.03bp,252.4bp)  .. (node_5);
  \draw (220.5bp,264.0bp) node {$-3$};
  \draw [green,->] (node_8) ..controls (278.82bp,196.15bp) and (306.69bp,125.76bp)  .. (node_12);
  \draw (312.5bp,150.0bp) node {$-3$};
  \draw [green,->] (node_6) ..controls (469.83bp,282.51bp) and (502.25bp,263.71bp)  .. (516.0bp,236.0bp) .. controls (531.82bp,204.14bp) and (489.22bp,176.38bp)  .. (node_11);
  \draw (531.5bp,226.0bp) node {$-3$};
  \draw [red,->] (node_5) ..controls (94.35bp,207.28bp) and (85.881bp,183.44bp)  .. (node_7);
  \draw (103.5bp,188.0bp) node {$-2$};
  \draw [red,->] (node_3) ..controls (433.0bp,351.19bp) and (433.0bp,331.15bp)  .. (node_6);
  \draw (445.5bp,334.0bp) node {$-2$};
  \draw [red,->] (node_5) ..controls (123.78bp,207.71bp) and (147.68bp,185.39bp)  .. (157.0bp,160.0bp) .. controls (164.92bp,138.43bp) and (161.44bp,111.58bp)  .. (node_9);
  \draw (173.5bp,150.0bp) node {$-2$};
  \draw [green,->] (node_1) ..controls (221.21bp,349.69bp) and (161.76bp,324.62bp)  .. (node_2);
  \draw (215.5bp,334.0bp) node {$-3$};
  \draw [blue,->] (node_1) ..controls (267.0bp,351.19bp) and (267.0bp,331.15bp)  .. (node_4);
  \draw (279.5bp,334.0bp) node {$-1$};
  \draw [blue,->] (node_11) ..controls (404.23bp,130.37bp) and (367.19bp,105.1bp)  .. (node_12);
  \draw (401.5bp,112.0bp) node {$-1$};
  \draw [green,->] (node_6) ..controls (433.0bp,280.58bp) and (433.0bp,258.93bp)  .. (node_10);
  \draw (445.5bp,264.0bp) node {$-3$};
  \draw [green,->] (node_9) ..controls (138.49bp,65.283bp) and (132.74bp,59.033bp)  .. (130.0bp,52.0bp) .. controls (126.77bp,43.719bp) and (126.77bp,40.281bp)  .. (130.0bp,32.0bp) .. controls (131.46bp,28.264bp) and (133.76bp,24.748bp)  .. (node_13);
  \draw (142.5bp,42.0bp) node {$-3$};
  \draw [red,->] (node_9) ..controls (154.21bp,64.603bp) and (154.76bp,57.917bp)  .. (155.0bp,52.0bp) .. controls (155.35bp,43.118bp) and (155.35bp,40.882bp)  .. (155.0bp,32.0bp) .. controls (154.9bp,29.596bp) and (154.76bp,27.066bp)  .. (node_13);
  \draw (167.5bp,42.0bp) node {$-2$};
  \draw [blue,->] (node_7) ..controls (94.759bp,130.82bp) and (120.63bp,106.91bp)  .. (node_9);
  \draw (137.5bp,112.0bp) node {$-1$};
  \draw [red,->] (node_10) ..controls (433.0bp,207.39bp) and (433.0bp,183.9bp)  .. (node_11);
  \draw (445.5bp,188.0bp) node {$-2$};
  \draw [blue,->] (node_2) ..controls (101.0bp,280.58bp) and (101.0bp,258.93bp)  .. (node_5);
  \draw (113.5bp,264.0bp) node {$-1$};
  \draw [red,->] (node_0) ..controls (337.71bp,420.03bp) and (305.3bp,396.64bp)  .. (node_1);
  \draw (339.5bp,404.0bp) node {$-2$};
\end{tikzpicture}
}
\caption{The graph $\widetilde{G}(\mathcal{C})$ recovered from the graph $\overline{G}(C)$ of Figure~\ref{figure.G(C) labeled}.
\label{figure.G(C).recovered}}
\end{figure}
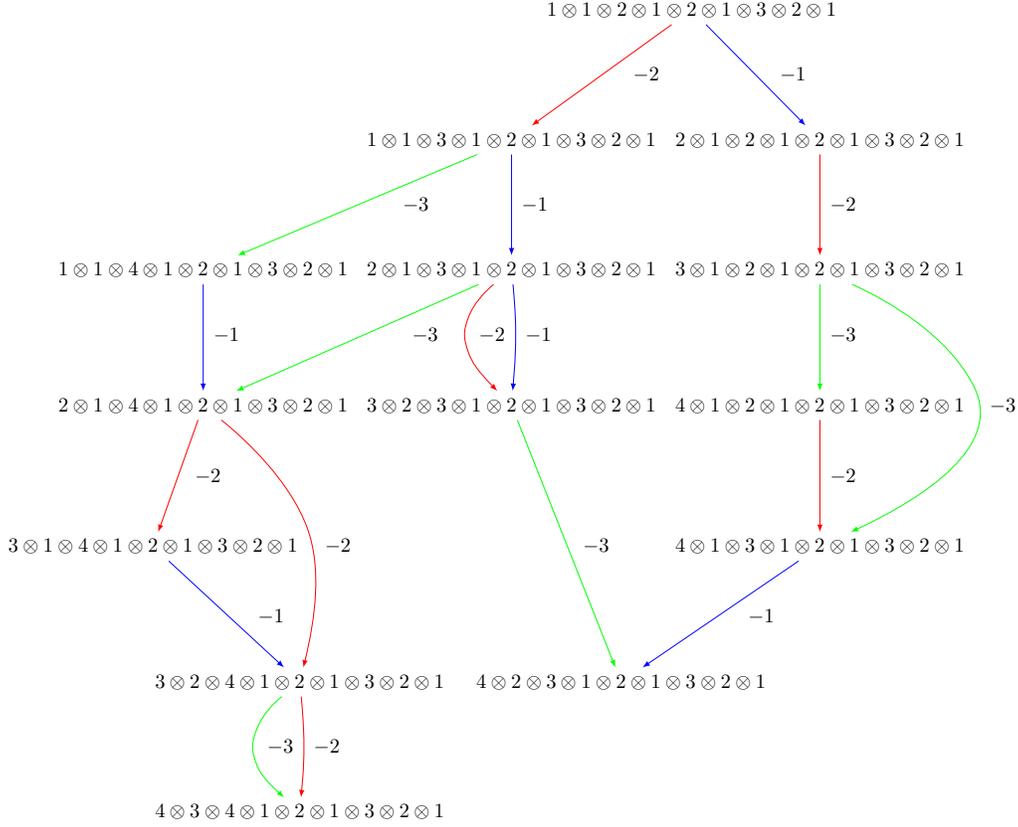

Next we introduce new axioms.

\begin{definition}[\defn{Connectivity axioms}]
\label{definition.connectivity axioms}
Let $\mathcal{C}$ be a connected crystal satisfying the local queer axioms of Definition~\ref{definition.local queer axioms}.
Let $v \in \mathcal{C}$ be an $I_0$-lowest weight element and $u = \uparrow v$. As in~\eqref{equation.precise almost lowest},
define $g_{j,k} := (e_1 \cdots e_j)(e_1 \cdots e_k) v$ for $1\leqslant j\leqslant k \leqslant n$. 
\begin{enumerate}
\item[{\bf C0.}] $\varphi_{-1}(g_{j,k})=0$ implies that $\varphi_{-1}(e_1\cdots e_k v)=0$.
\item[{\bf C1.}] Suppose that $G(\mathcal{C})$ contains an edge $u\to u'$ such that $\wt(u')$ is obtained from $\wt(u)$
by moving a box from row $n+1-k$ to row $n+1-h$ with $h<k$. For all $h<j \leqslant k$ such that $g_{j,k}\neq 0$, we require 
that $f_{-1} g_{j,k} \neq 0$ and 
\[
	f_{-1} g_{j,k} = (e_2 \cdots e_j) (e_1 \cdots e_h) v',
\]
where $v'$ is $I_0$-lowest weight with $\uparrow v' = u'$.
\item[{\bf C2.}] Suppose that either (a) $G(\mathcal{C})$ contains an edge $u\to u'$ such that $\wt(u')$ is obtained from 
$\wt(u)$ by moving a box from row $n+1-k$ to row $n+1-h$ with $h<k$ or (b) no such edge exists in $G(\mathcal{C})$.
For all $1\leqslant j \leqslant h$ in case (a) and all $1\leqslant j \leqslant k$ in case (b) such that 
$g_{j,k}\neq 0$ and $f_{-1} g_{j,k} \neq 0$, we require that 
\[
	f_{-1} g_{j,k} = (e_2 \cdots e_k) ( e_1 \cdots e_j) v.
\]
\end{enumerate}
\end{definition}

\begin{remark}
\label{remark.new C0}
Condition \textbf{C0} can be replaced by the following condition:
\[
	\text{\textbf{LQ7.}} \quad \text{If $\varepsilon_1(e_2(b))>\varepsilon_1(b)$ for $b\in \mathcal{C}$ with 
	$\varepsilon_2(b)>0$, then $\varphi_{-1}(b) \leqslant \varphi_{-1}(e_1e_2(b))$.}
\]
This condition indeed implies \textbf{C0}. Suppose $\varphi_{-1}(e_1\cdots e_k v)=1$. Then for
$b=(e_3 \cdots e_j) (e_1 \cdots e_k) v$, we have $\varphi_{-1}(b)=1$.  However, $b$ satisfies 
$\varepsilon_1(e_2(b))>\varepsilon_1(b)$, so the above condition implies that $\varphi_{-1}(e_1e_2(b))=1$ as well.  
But $e_1e_2(b)=g_{j,k}$. Hence $\varphi_{-1}(g_{j,k})=0$ implies that $\varphi_{-1}(e_1\cdots e_k v)=0$.

Moreover, in $\mathcal{B}^{\otimes \ell}$ the conditions in \textbf{LQ7} are satisfied. Namely, the condition 
$\varepsilon_1(e_2(b))>\varepsilon_1(b)$ implies that $e_2(b) \neq 0$ and $e_1e_2(b) \neq 0$. 
Moreover,  this condition implies that $e_1$ acts on $e_2(b)$ in a position weakly to the left of where $e_2$ acts on $b$.  
Thus if $\varphi_{-1}(b)=1$, it immediately follows that $\varphi_{-1}(e_1e_2(b))=1$ which proves the statement.
\end{remark}

\begin{theorem}
\label{theorem.connectivity axioms}
The $\mathfrak{q}(n+1)$-queer crystal $\mathcal{B}^{\otimes \ell}$ satisfies the axioms in 
Definition~\ref{definition.connectivity axioms}.
\end{theorem}

The proof of Theorem~\ref{theorem.connectivity axioms} is given in Appendix~\ref{appendix.proof connectivity}.

Next we show that the arrows in $G(\mathcal{C})$, where $\mathcal{C}$ is a connected component in 
$\mathcal{B}^{\otimes \ell}$, can be modeled by $e_{-i}$ on type $A$ highest weight elements.

\begin{proposition}
\label{proposition.e-i component}
Let $\mathcal{C}$ be a connected component in the $\mathfrak{q}(n+1)$-crystal $\mathcal{B}^{\otimes \ell}$.
Let $C_1$ and $C_2$ be two distinct type $A_n$ components in $\mathcal{C}$ and let $u_2$ be the $I_0$-highest weight
element in $C_2$. Then there is an edge from $C_1$ to $C_2$  in $G(\mathcal{C})$ if and only if
$e_{-i} u_2 \in C_1$ for some $i\in I_0$.
\end{proposition}

\begin{proof}
First note that there is an edge from $C_1$ to $C_2$ in $G(\mathcal{C})$ if there exists $b_1\in C_1$ and 
$b_2 \in C_2$ such that $e_{-1}b_2 = b_1$.
Recall that by~\eqref{equation.-i} we have $e_{-i} := s_{w_i^{-1}} e_{-1} s_{w_i}$.
Hence, if $e_{-i} u_2$ is defined and $e_{-i}u_2\in C_1$, then $b_2 := e_{-1} b_1$ is defined, where $b_1:= s_{w_i} u_1 
\in C_1$ and $b_2 \in C_2$. This proves that  there is an edge between $C_1$ and $C_2$ in $G(\mathcal{C})$.

Conversely assume that $b_1 = e_{-1}b_2$ for some $b_1\in C_1$ and $b_2 \in C_2$.
We want to show that then $e_{-i} u_2 \in C_1$ for some $i\in I_0$. 
By the discussion before Lemma~\ref{lemma.gjk}, we know that the $(-1)$-arrow on $b_1$ is induced (using the local 
queer axioms of Definition~\ref{definition.local queer axioms}) by the $(-1)$-arrow on 
$g_{j,k} = (e_1 \cdots e_j) (e_1 \cdots e_k) v_1$ for some $j\leqslant k$. By Theorem~\ref{theorem.connectivity axioms} and 
Condition \textbf{C1} of Definition~\ref{definition.connectivity axioms}, we must have
\[
	f_{-1} g_{j,k} = (e_2 \cdots e_j)(e_1 \cdots e_h) v_2 \qquad \text{for some $h<j\leqslant k$,}
\]
where $v_2$ is the $I_0$-lowest weight element in the component $C_2$. In particular, for the edge $u_1 \to u_2$ 
in $G(\mathcal{C})$, where $u_1$ is the $I_0$-highest weight element in the component $C_1$, the weight $\wt(u_2)$ 
differs from $\wt(u_1)$ by moving a box from row $n+1-k$ to row $n+1-h$ with $1\leqslant h<k\leqslant n$.
Furthermore, all $g_{j',k} \neq 0$ with $h<j'\leqslant k$ are mapped to component $C_2$ under $f_{-1}$. 

\smallskip

\noindent
\textbf{Claim:} 
Set $b:=s_{w_{n-h}} u_2$ and $b':=(e_2 \cdots e_{h+1}) ( e_1 \cdots e_h) v_2$.
If $\wt(b)_2>0$, there exist $j_1,\ldots, j_p \in I_0$ such that $b'= f_{j_1} \cdots f_{j_p} b$ and 
\begin{equation}
\label{equation.phi2 condition}
	\varphi_2(f_{j_a} \cdots f_{j_p} b)>0 \quad \text{if $j_a=2$.}
\end{equation}

\smallskip

The claim is a statement about type $A_n$ crystal operators, hence one may use the tableaux model
to verify it. It is straightforward to verify that every column of height $d>n-h$ in the insertion tableau of $b$ contains the 
letter $m$ in row $m$; the columns of height $n-h$ contain 1 in the first row and $m+1$ in row $m>1$; finally the columns 
of height $d<n-h$ contain the letter $m+2$ in row $m$. Hence $\wt(b)_2>0$ is only satisfied if there is at least 
one column of height $d>n-h$. Now we start acting with operators $f_j$ on $b$, where $j\in I_0\setminus \{2\}$, to make
$b$ into a $I_0\setminus \{2\}$-lowest weight element. This element differs from $v_2$ only in columns of height
$d\geqslant n-h$; columns of height $d>n-h$ contain 1 and 2 in rows 1 and 2, respectively, whereas columns of
height $d=n-h$ contain 2 in row 1. Suppose that there are $p$ columns whose height is less than $n+1$ and at least $n-h$.
Then we can apply $f_2^{p-1}$ without violating~\eqref{equation.phi2 condition} since each such column contains an 
unbracketed 2. Then apply again $f_j$ with $j\in I_0\setminus \{2\}$ to make the tableau into a $I_0\setminus \{2\}$-lowest 
weight element, followed by the maximal number of $f_2$ satisfying~\eqref{equation.phi2 condition}, followed by making the 
result $I_0\setminus \{2\}$-lowest weight. This tableau is exactly $(e_2 \cdots e_{h+1}) ( e_1 \cdots e_h) v_2$. 
This proves the claim.

Now since by assumption $\wt(u_2)$ differs from $\wt(u_1)$ by moving a box from row $n+1-k$ to row $n+1-h$,
as a tableau $s_{w_{n-h}}u_2$ indeed has a column of height $d>n-k$, so that $\wt(s_{w_{n-h}}u_2)_2>0$.
By condition~\eqref{equation.phi2 condition}, the $(-1)$-arrow coming into $s_{w_{n-h}}u_2$ is induced by the $(-1)$-arrow
coming into $(e_2 \cdots e_{h+1}) ( e_1 \cdots e_h) v_2$ by the local queer axioms of Definition~\ref{definition.local queer
axioms}. Hence $e_{-(n-h)} u_2 \in C_1$, which proves the proposition where $i=n-h$.
\end{proof}

\begin{example}
Let us illustrate the claim in the proof of Proposition~\ref{proposition.e-i component}.
Let $n=5,h=2$ and consider the type $A_5$ component $C_2$ of weight $(4,3,3,2,1)$.
Then
\[
	\raisebox{1cm}{$b=s_{w_3} u_2=$} \young(5,44,334,223,1113)
	\raisebox{1cm}{. This becomes} \;\; \young(6,56,456,235,1136)
\]
after making it $\{1,3,4,5\}$-lowest weight and applying $f_2^2$. Making this element $\{1,3,4,5\}$-lowest weight
again, no further $f_2$ are applicable and we obtain
\[
	\young(6,56,456,235,1246) \raisebox{1cm}{$=(e_2e_3)(e_1e_2)v_2$.}
\]
\end{example}

By Proposition~\ref{proposition.e-i component}, there is an edge from component $C_1$ to component $C_2$
in $G(\mathcal{C})$ if and only if $e_{-i} u_2 \in C_1$ for some $i\in I_0$, where $u_2$ is the $I_0$-highest weight
element of $C_2$. We call the arrow \defn{combinatorial} if $e_{-i} u_2$ is $\{1,2,\ldots,i\}$-highest weight.
Otherwise the arrow is called a \defn{by-pass arrow}.

Define $f_{(-i,h)}:=f_{-i} f_{i+1} f_{i+2} \cdots f_{h-1}$.

\begin{theorem}
\label{theorem.combinatorial G}
Let $\mathcal{C}$ be a connected component in $\mathcal{B}^{\otimes \ell}$. Then each by-pass arrow is the 
composition of combinatorial arrows. Furthermore, each combinatorial edge in $G(\mathcal{C})$ can be obtained by 
$f_{(-i,h)}$ for some $i\in I_0$ and $h>i$ minimal such that $f_{(-i,h)}$ applies.
\end{theorem}

\begin{proof}
Consider a combinatorial arrow from component $C_1$ to $C_2$. This means that $e_{-i} u_2$ is defined 
for some $i \in I_0$ and $e_{-i} u_2$ is $\{1,2,\ldots, i\}$-highest weight.
Then by Theorem~\ref{theorem.f-i} and Corollary~\ref{corollary.f-i} we have $f_{(-i,h)} u_1 = u_2$ for some $h>i$.

If the arrow is a by-pass arrow, then $e_{-i} u_2$ is not $\{1,2,\ldots, i\}$-highest weight.
By Proposition~\ref{proposition.by-pass} and induction, 
there exists a sequence of indices $1\leqslant i_1,\ldots,i_a<i$ such that
\[
	\uparrow e_{-i} u_2 = \uparrow e_{-i} \uparrow e_{-i_1} \cdots \uparrow e_{-i_a} u_2
\]
where each partial sequence $e_{-i_j} \uparrow e_{-i_{j+1}} \cdots \uparrow e_{-i_a} u_2$ is 
$\{1,2,\ldots, i_j\}$-highest weight. This means that each by-pass arrow is the composition of combinatorial
arrows.
\end{proof}

Theorem~\ref{theorem.combinatorial G} provides a combinatorial description of the graph $G(\mathcal{C})$.
Let $\overline{G}(\mathcal{C})$ be the graph $G(\mathcal{C})$ with all by-pass arrows removed and
each edge labeled by the tuple $(-i,h)$ for the combinatorial arrow $f_{(-i,h)} u_1 = u_2$, where 
$f_{-i}$ is given by the combinatorial rules stated in Theorem~\ref{theorem.f-i}.
Hence $\overline{G}(\mathcal{C})$ can be constructed from the $\mathfrak{q}(n+1)$-highest weight element $u$ by 
the application of combinatorial arrows, see for example Figure~\ref{figure.G(C) labeled}. In particular, the graph 
$G(\mathcal{C})$ and the graph $\overline{G}(\mathcal{C})$ have the same vertices.

Next we construct a graph $\widetilde{G}(\mathcal{C})$ from $\overline{G}(\mathcal{C})$ by applying $\uparrow e_{-i}$
to each vertex in the graph $\widetilde{G}(\mathcal{C})$ (if applicable). This will add additional labeled edges between
the vertices in the graph, see Figure~\ref{figure.G(C).recovered}. We would like to emphasize that the construction
of $\widetilde{G}(\mathcal{C})$ for a connected component $\mathcal{C}$ of $\mathcal{B}^{\otimes \ell}$ is purely 
combinatorial, starting with the highest weight element $u$ of a given weight $\lambda$, applying $f_{(-i,h)}$ of 
Theorem~\ref{theorem.f-i}, and then applying $\uparrow e_{-i}$ to all vertices using Theorem~\ref{theorem.e-i}.
This provides a combinatorial construction of $G(\mathcal{C})$ by dropping the labels in $\widetilde{G}(\mathcal{C})$.

\section{Characterization of queer crystals}
\label{section.characterization}

Our main theorem gives a characterization of the queer supercrystals.

\begin{theorem}
\label{theorem.main}
Let $\mathcal{C}$ be a connected component of a generic abstract queer crystal (see 
Definition~\ref{definition.abstract queer}). Suppose that $\mathcal{C}$ satisfies the following conditions:
\begin{enumerate}
\item $\mathcal{C}$ satisfies the local queer axioms of Definition~\ref{definition.local queer axioms}.
\item $\mathcal{C}$ satisfies the connectivity axioms of Definition~\ref{definition.connectivity axioms}.
\item $G(\mathcal{C})$ is isomorphic to $G(\mathcal{D})$, where $\mathcal{D}$ is some connected
component of $\mathcal{B}^{\otimes \ell}$.
\end{enumerate}
Then the queer supercrystals $\mathcal{C}$ and $\mathcal{D}$ are isomorphic.
\end{theorem}

Theorem~\ref{theorem.main} states that the local queer axioms, the connectivity axioms, and the component graph
uniquely characterize queer crystals. Before we give its proof, we need the following statement. Recall that
$g_{j,k} = (e_1 \cdots e_j)(e_1 \cdots e_k)v$ was defined in~\eqref{equation.precise almost lowest}, where $v$
is an $I_0$-lowest weight vector.

\begin{lemma}
\label{lemma.gjk condition}
In a crystal satisfying the local queer axioms of Definition~\ref{definition.local queer axioms}
and \textbf{C0} of Definition~\ref{definition.connectivity axioms}, we have for any
$g_{j,k} \neq 0$ with $1\leqslant j \leqslant k$
\[
	\varphi_{-1}(g_{j,k}) = 0 \quad \text{if and only if} \quad \varphi_{-1}(e_1 \cdots e_k v) =0.
\]
\end{lemma}

\begin{proof}
The condition \textbf{C0} requires that $\varphi_{-1}(g_{j,k})=0$ implies $\varphi_{-1}(e_1\cdots e_k v)=0$.

For the converse direction, note that $\wt(e_1 \cdots e_k v)_1>0$. Hence 
\[
	\varphi_{-1}(e_1 \cdots e_k v) =0 \quad \Leftrightarrow \quad 
	\varepsilon_{-1}(e_1 \cdots e_k v) =1.
\]
By the local queer axioms \textbf{LQ6} and \textbf{LQ5} of Definition~\ref{definition.local queer axioms}
(see also Figure~\ref{figure.local queer}), we have
\[
	\varepsilon_{-1}(e_1 \cdots e_k v) =1 \quad \Leftrightarrow \quad 
	\varepsilon_{-1}((e_3 \cdots e_j)(e_1 \cdots e_k) v) =1 \quad \Rightarrow \quad 
	\varepsilon_{-1}((e_2 \cdots e_j)(e_1 \cdots e_k) v) =1.
\]
It can be easily checked that $\varphi_1((e_2 \cdots e_j)(e_1 \cdots e_k)v)=1$ for $j\leqslant k$
(for example using the tableaux model for type $A_n$ crystals). Hence by the local queer axioms
\[
	\varepsilon_{-1}((e_2 \cdots e_j)(e_1 \cdots e_k) v) =1 \quad \Leftrightarrow \quad 
	\varepsilon_{-1}((e_1 \cdots e_j)(e_1 \cdots e_k) v) =1.
\]
This proves that $\varphi_{-1}(e_1 \cdots e_k v)=0$ implies $\varphi_{-1}(g_{j,k})=0$.
\end{proof}

\begin{proof}[Proof of Theorem~\ref{theorem.main}]
By Proposition~\ref{proposition.local queer axioms} and Theorem~\ref{theorem.connectivity axioms},
$\mathcal{D}$ satisfies the local queer axioms and the connectivity axioms and hence all conditions of
the theorem.

By \textbf{LQ1} of the local queer axioms of Definition~\ref{definition.local queer axioms}, each type $A_n$-component
of $\mathcal{C}$ is a Stembridge crystal and hence is uniquely characterized by~\cite{Stembridge.2003}.
By assumption $G(\mathcal{C}) \cong G(\mathcal{D})$. In particular, the vertices of $G(\mathcal{C})$ and
$G(\mathcal{D})$ agree. This proves that $\mathcal{C}$ and $\mathcal{D}$ are isomorphic as $A_n$ crystals.

Next we show that all $(-1)$-arrows also agree on $\mathcal{C}$ and $\mathcal{D}$. 
As discussed just before Lemma~\ref{lemma.gjk}, given the local queer axioms of Definition~\ref{definition.local queer axioms},
it suffices to show that $f_{-1}$ acts in the same way in $\mathcal{C}$ and $\mathcal{D}$ on the almost lowest elements 
satisfying~\eqref{equation.almost lowest} or equivalently by Lemma~\ref{lemma.gjk} on every
$g_{j,k}\neq 0$ with $1\leqslant j \leqslant k \leqslant n$. For the remainder of this proof, fix $g_{j,k} \neq 0$ in the 
$I_0$-component $u$.

Let us first assume that $G(\mathcal{C})$ contains an edge $u\to u'$ such that $\wt(u')$ is obtained from $\wt(u)$
by moving a box from row $n+1-k$ to row $n+1-h$ for some $h<k$. If $h<j\leqslant k$, then $f_{-1} g_{j,k}$ is determined
by \textbf{C1} of Definition~\ref{definition.connectivity axioms}. If $j\leqslant h$, pick $h<j'\leqslant k$ such that
$g_{j',k} \neq 0$. Such a $j'$ must exist since there is an edge $u\to u'$ in $G(\mathcal{C})$. By \textbf{C1},
we have $\varphi_{-1}(g_{j',k})=1$ and hence by Lemma~\ref{lemma.gjk condition} also $\varphi_{-1}(g_{j,k})=1$.
Hence $f_{-1} g_{j,k}$ is determined by \textbf{C2}(a).

Next assume that $G(\mathcal{C})$ does not contain an edge $u\to u'$ such that $\wt(u')$ is obtained from $\wt(u)$
by moving a box from row $n+1-k$.

\smallskip

\noindent \textbf{Claim}: If $g_{k,k} \neq 0$, then $f_{-1} g_{j,k}=0$.

\begin{proof}
Suppose $f_{-1} g_{k,k} \neq 0$. By \textbf{C2}(b), we have $f_{-1} g_{k,k} = (e_2 \cdots e_k) ( e_1 \cdots e_k) v
= f_1 g_{k,k}$. But this contradicts the local queer axioms of Definition~\ref{definition.local queer axioms} since
$\varphi_1(g_{k,k})>1$. Hence $\varphi_{-1}(g_{k,k})=0$ and by Lemma~\ref{lemma.gjk condition}
also $\varphi_{-1}(g_{j,k})=0$, which proves the claim.
\end{proof}

If $g_{k,k}=0$, we have $j<k$ since by assumption $g_{j,k} \neq 0$.

\smallskip

\noindent \textbf{Claim}: Suppose $g_{k,k}=0$.
\begin{enumerate}
\item Suppose there is an edge $\overline{u} \to u$ in $G(\mathcal{C})$ such that $\wt(u)$ is obtained from 
$\wt(\overline{u})$ by moving a box from row $n+1-\overline{k}$ to row $n+1-\overline{h}$ such that 
$\overline{h}<k\leqslant \overline{k}$. Then $f_{-1} g_{j,k} =0$.
\item Suppose $G(\mathcal{C})$ does not contain an edge as in (1). Then
$f_{-1} g_{j,k} = (e_2 \cdots e_k) ( e_1 \cdots e_j) v$.
\end{enumerate}

\begin{proof}
Suppose that the conditions in (1) are satisfied. Then by \textbf{C1} there must exist
\[
	\overline{g}_{\overline{j},\overline{k}} := (e_1 \cdots e_{\overline{j}}) (e_1 \cdots e_{\overline{k}}) \overline{v} \neq 0,
\]
where $\overline{h} < \overline{j} \leqslant \overline{k}$ and $\overline{v}$ is the $I_0$-lowest weight element in the 
component of $\overline{u}$, such that
\begin{equation}
\label{equation.overline g}
	f_{-1} \overline{g}_{\overline{j},\overline{k}} = (e_2 \cdots e_{\overline{j}})(e_1 \cdots e_{\overline{h}}) v.
\end{equation}
Since $g_{j,k}\neq 0$, we have in particular that $(e_1 \cdots e_k) v \neq 0$. Since $\wt(u)$ is obtained from 
$\wt(\overline{u})$ by moving a box from row $n+1-\overline{k}$ to row $n+1-\overline{h}$, this hence also implies that
$\overline{g}_{k,\overline{k}} = (e_1 \cdots e_k)(e_1 \cdots e_{\overline{k}})\overline{v} \neq 0$. Hence by \textbf{C1}
Equation~\eqref{equation.overline g} holds for $\overline{j}=k$.

If $f_{-1}g_{\overline{h},k}=0$, we also have $f_{-1} g_{j,k}=0$ by Lemma~\ref{lemma.gjk condition} as claimed.
Hence we may assume that $f_{-1}g_{\overline{h},k}\neq 0$. Then by \textbf{C2}(b) we have
\[
	f_{-1} g_{\overline{h},k} = (e_2 \cdots e_k)(e_1 \cdots e_{\overline{h}}) v.
\]
But then $f_{-1} \overline{g}_{k,\overline{k}} = f_{-1} g_{\overline{h},k} = (e_2 \cdots e_k)(e_1 \cdots e_{\overline{h}}) v$,
which contradicts the fact that the crystal operator $f_{-1}$ has a partial inverse since 
$\overline{g}_{k,\overline{k}} \neq g_{\overline{h},k}$. This proves (1).

Now suppose that the conditions in (2) are satisfied. Recall that by assumption $g_{j,k} \neq 0$ with $j<k$. This implies
that $y:=(e_2 \cdots e_k) (e_1 \cdots e_j)v \neq 0$, $\varphi_i(y)=0$ for $i \in I_0 \setminus \{2\}$ and
$\varphi_2(y)=1$. By the local queer axioms of Definition~\ref{definition.local queer axioms}, this implies that
$x:= e_{-1} y \neq 0$ with $\varphi_1(x) \in \{1,2\}$ and $\varphi_i(x)=0$ for $i\in I_0 \setminus \{1\}$.
Thus we may write $x= (e_1 \cdots e_s)(e_1 \cdots e_t)\overline{v}$, where $0 \leqslant s \leqslant t$ and
$\overline{v} \in \mathcal{C}$ is some $I_0$-lowest weight vector. This yields the equality
\[
	f_{-1} (e_1 \cdots e_s)(e_1 \cdots e_t) \overline{v} = (e_2 \cdots e_k)(e_1\cdots e_j)v.
\]
If $\overline{v}\neq v$, then by the connectivity axioms of Definition~\ref{definition.connectivity axioms} this means that 
$j<k=s\leqslant t$ and there is an edge in $G(\mathcal{C})$ from $\uparrow \overline{v}$ to $u=\uparrow v$, moving a 
box from row $n+1-t$ to row $n+1-j$. This contradicts the assumptions of (2). Hence we must have $\overline{v}=v$.
By \textbf{C2}(b) we have $f_{-1} g_{s,t} = (e_2 \cdots e_t)(e_1 \cdots e_s)v$, so that $k=t$ and $j=s$.
This implies $f_{-1} g_{j,k} = (e_2 \cdots e_k) ( e_1 \cdots e_j) v$, proving the claim.
\end{proof}

We have now shown that $f_{-1} g_{j,k}$ is determined in all cases, which proves the theorem.
\end{proof}

\begin{remark}
\label{remark.GC example}
Consider the $\mathfrak{q}(4)$-queer crystal $\mathcal{B}^{\otimes 4}$. The elements $4114$ and $4113$ both lie
in the same $\{1,2,3\}$-component of highest weight $(3,1)$. The highest (resp. lowest) weight element in this component 
is $u=2111$ (resp. $v=4344$). Both $4114$ and $4113$ satisfy~\eqref{equation.almost lowest}. In fact,
$4114 = (e_1 e_2) (e_1 e_2 e_3) v = g_{2,3}$ and $4113 = (e_1 e_2 e_3) (e_1 e_2 e_3) v = g_{3,3}$.
In the component of $u$ there is no sequence of crystal operators that would induce the action of $f_{-1}$ on $4114$
from the action of $f_{-1}$ on $4113$ using the local queer axioms of Definition~\ref{definition.local queer axioms}.

This suggests that the connectivity axioms of Definition~\ref{definition.connectivity axioms} are indeed necessary.
However, in this example the graph $G(\mathcal{C})$, where $\mathcal{C}$ is the connected component in 
$\mathcal{B}^{\otimes 4}$ containing $2111$, is linear and hence forces $4114$ and $4113$ to be mapped to the
same $\{1,2,3\}$-component by $f_{-1}$, see Figure~\ref{figure.GC example}.
\end{remark}

\begin{figure}[t]
\begin{tikzpicture}[>=latex,line join=bevel,]
\node (node_3) at (31.0bp,7.0bp) [draw,draw=none] {$4 \otimes 3 \otimes 2 \otimes 1$};
  \node (node_2) at (31.0bp,57.0bp) [draw,draw=none] {$3 \otimes 2 \otimes 1 \otimes 1$};
  \node (node_1) at (31.0bp,107.0bp) [draw,draw=none] {$2 \otimes 1 \otimes 1 \otimes 1$};
  \node (node_0) at (31.0bp,157.0bp) [draw,draw=none] {$1 \otimes 1 \otimes 1 \otimes 1$};
  \draw [black,->] (node_2) ..controls (31.0bp,42.974bp) and (31.0bp,32.942bp)  .. (node_3);
  \draw [black,->] (node_0) ..controls (31.0bp,142.97bp) and (31.0bp,132.94bp)  .. (node_1);
  \draw [black,->] (node_1) ..controls (31.0bp,92.974bp) and (31.0bp,82.942bp)  .. (node_2);
\end{tikzpicture}
\caption{The graph $G(\mathcal{C})$ for the example in Remark~\ref{remark.GC example}.
\label{figure.GC example}}
\end{figure}
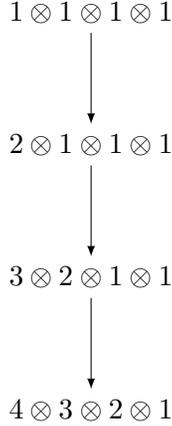

\begin{remark}
Consider the connected component $\mathcal{C}$ of $111212121$ in the $\mathfrak{q}(6)$-queer crystal 
$\mathcal{B}^{\otimes 9}$. The $\{1,2,3,4,5\}$-component containing $321312121$ is connected to the components 
$421312121$, $431312121$, and $432312121$ in $G(\mathcal{C})$.
The elements $g_{4,5}=651615464$ and $g_{3,5}=651615465$ in the component of $321312121$
are mapped to the same component $432312121$ by \textbf{C1} of Definition~\ref{definition.connectivity axioms}.
However, the element $g_{4,5}$ is connected to $431413131$ in the crystal using only arrows that commute with $f_{-1}$
and the element $g_{3,5}$ is connected to $431413143$ in the crystal using only arrows that commute with $f_{-1}$.
However, these two components (containing $431413131$ resp. $431413143$ using only crystal operators $f_i$ and $e_i$
with $i\in I_0$ that commute with $f_{-1}$) are disjoint.
This suggests that \textbf{C1} of Definition~\ref{definition.connectivity axioms} is necessary for uniqueness.
\end{remark}

\appendix
\section{Proof of Theorem~\ref{theorem.connectivity axioms}}
\label{appendix.proof connectivity}

In this appendix we prove Theorem~\ref{theorem.connectivity axioms}. We use the shorthand notation
$e_1^k := e_1\cdots e_k$, $e_{\bar{1}}^k := e_{-1} e_2 \cdots e_k$, $f_k^1 := f_k \cdots f_1$, and 
$f_k^{\bar{1}} := f_k \cdots f_2 f_{-1}$.

\begin{lemma}
In $\mathcal{B}^{\otimes \ell}$, condition \textbf{C0} of Definition~\ref{definition.connectivity axioms} holds.
\end{lemma}

\begin{proof}
This follows from Remark~\ref{remark.new C0}.
\end{proof}

The connectivity axioms \textbf{C1} and \textbf{C2} of Definition~\ref{definition.connectivity axioms} are implied by the 
following conditions. Here $v$ is an $I_0$-lowest weight vector in $\mathcal{C}$:
\begin{enumerate}
\item[\textbf{C1'.}] If $h<k$ and there exists some $j \in(h,k]$ such that $f_h^1f_j^{\bar{1}}e_1^je_1^k (v)$ is $I_0$-lowest 
weight, then for any $j' \in (h,k]$ with $e_1^{j'}e_1^k (v)\neq0$ we have $f_{j'}^{\bar{1}}e_1^{j'}e_1^k(v) = 
f_j^{\bar{1}}e_1^je_1^k(v)$.
\item[\textbf{C2'.}] If $j \leqslant k$ and $f_{-1}e_1^je_1^k (v)\neq 0$, then either:
\begin{enumerate}
\item $j \neq k$ and $f_j^1f_k^{\bar{1}}e_1^je_1^k(v)=v$, or
\item $f_h^1f_j^{\bar{1}}e_1^je_1^k(v)$ is $I_0$-lowest weight for some $h<j$.
\end{enumerate}
\end{enumerate}

\begin{proposition}
\label{proposition.ins}
In $\mathcal{B}^{\otimes \ell}$, condition \textbf{C2'} holds.
\end{proposition}

\noindent
The proof of Proposition~\ref{proposition.ins} is given in Section~\ref{section.proof C2'}.

\begin{proposition} 
\label{proposition.prin}
In $\mathcal{B}^{\otimes \ell}$, condition \textbf{C1'} holds.
\end{proposition}

We will prove a seemingly weaker statement:
\begin{lemma} 
\label{lemma.clean}
In $\mathcal{B}^{\otimes \ell}$, condition \textbf{C1'} holds for $j=n-1$, $j'=k=n$ and for $j=k=n$, $j'=n-1$.
\end{lemma}

\noindent
The proof of Lemma~\ref{lemma.clean} is given in Sections~\ref{section.proof clean 1} and~\ref{section.proof clean 2}.

\begin{proposition}
Lemma~\ref{lemma.clean} implies Proposition~\ref{proposition.prin}.
\end{proposition}

\begin{proof}
We first assume that $h<j<j'\leqslant k$ and the assumptions in \textbf{C1'} hold. Then we have
\begin{equation*}
\begin{split}
	f_h^1 f_j^{\bar{1}} e_1^j e_1^k(v) 
	&= f_h^1 f_j^{\bar{1}} (f_{j'} \cdots f_{j+2})(e_{j+2} \cdots e_{j'}) e_1^j e_1^k(v)\\
	&= (f_{j'} \cdots f_{j+2}) f_h^1 f_j^{\bar{1}} e_1^j (e_{j+2} \cdots e_{j'}) e_1^k(v)\\
	&= (f_{j'} \cdots f_{j+2}) f_h^1 f_j^{\bar{1}} e_1^j e_1^{j+1} (v'),
\end{split}
\end{equation*}
where $v' = (e_{j+2} \cdots e_{j'})(e_{j+2} \cdots e_k)(v)$. Here we have used Stembridge relations to commute
crystal operators and in the last step also that the operators are acting on an $I_0$-lowest weight element.
Note that $v'$ is $\{1, \ldots, j+1\}$-lowest weight.  Moreover, 
$f_h^1 f_j^{\bar{1}} e_1^j e_1^{j+1} (v')$ is $\{1, \ldots, j+1\}$-lowest weight.  
Since  $e_1^{j+1} e_1^{j+1} (v') = e_1^{j'} e_1^k(v)\neq 0$, we may apply Lemma~\ref{lemma.clean} with $n=j+1$.
This implies
\begin{equation*}
\begin{split}
	(f_{j'} \cdots f_{j+2}) f_h^1 f_j^{\bar{1}} e_1^j e_1^{j+1} (v')
	&= (f_{j'} \cdots f_{j+2}) f_h^1 f_{j+1}^{\bar{1}} e_1^{j+1} e_1^{j+1} (v')\\
	&= f_h^1 f_{j'}^{\bar{1}} e_1^{j+1} e_1^{j+1} (e_{j+2} \cdots e_{j'}) (e_{j+2} \cdots e_k) (v)\\
	&= f_h^1 f_{j'}^{\bar{1}} e_1^{j'} e_1^k(v),
\end{split}
\end{equation*}
which proves the claim.

Next assume that $h<j'<j\leqslant k$. Then
\[
	f_h^1f_j^{\bar{1}}e_1^je_1^k(v)
	=f_h^1f_j^{\bar{1}}e_1^{j'+1}e_1^{j'+1}(e_{j'+2} \cdots e_j)(e_{j'+2} \cdots e_k)(v)
	=(f_j \cdots f_{j'+2})f_h^1 f_{j'+1}^{\bar{1}}e_1^{j'+1}e_1^{j'+1}(v'),
\]
where $v' = (e_{j'+2} \cdots e_j)(e_{j'+2} \cdots e_k)(v)$.
In this case, both $v'$ and $f_h^1 f_{j'+1}^{\bar{1}}e_1^{j'+1}e_1^{j'+1}(v')$ are $\{1, \ldots, j'+1\}$-lowest weight.
Since  $e_1^{j'} e_1^{j'+1} (v') \neq 0$, we may apply Lemma~\ref{lemma.clean} with $n=j'+1$ to obtain
\[
	f_h^1f_j^{\bar{1}}e_1^je_1^k(v)
	=(f_j \cdots f_{j'+2})f_h^1 f_{j'}^{\bar{1}}e_1^{j'}e_1^{j'+1}(v') = f_h^1 f_{j'}^{\bar{1}} e_1^{j'} e_1^k(v),
\]
proving the claim.
\end{proof}

\subsection{Proof of Proposition~\ref{proposition.ins}}
\label{section.proof C2'}

Given a word $w=w_1 \cdots w_\ell$ in the letters $\{1, \ldots, n+1\}$ we write $w^{\#}=\overline{w_\ell} \cdots \overline{w_1}$, 
where $\overline{w_i}=n+2-w_i$. Suppose that $x=g_{j,k} = e_1^j e_1^k(v) \in \mathcal{B}^{\otimes \ell}$, where $v$ is
$I_0$-lowest weight and $1\leqslant j \leqslant k \leqslant n$, so that by Lemma~\ref{lemma.gjk} we have $\varphi_1(x)=2$ 
and  $\varphi_i(x)=0$ for all $i>1$.  
The RSK insertion tableau for $x^{\#}$, denoted by $P(x^{\#})$, can be constructed as follows:  Construct the 
semistandard Young tableau with weight and shape equal to the weight of $v^{\#}$.  Change the rightmost $n+1-k$ in 
row $n+1-k$ and the rightmost $n+1-j$ in row $n+1-j$ to $n+1$. 

For instance, suppose $n=8$ and $x=198199887766$.  Then $x=e_1^6e_1^8(v)$, where $v=998799887766$ is 
$I_0$-lowest weight and $v^{\#}=443322113211$ has weight $(4,3,3,2)$. Hence the tableau $P(x^{\#})$ is obtained from 
the tableau of shape and weight equal to $(4,3,3,2)$ by changing the rightmost $1$ in row $1$ to $9$ and the rightmost 
$3$ in row $3$ to $9$:
\[
	\young(44,333,222,1111) \quad \raisebox{0.7cm}{$\longrightarrow$} \quad \young(44,339,222,1119)\;.
\]

Below, we consider the entries of a tableau to be linearly ordered in the row reading order.
If $f_{-1}(x)\neq 0$ there are two possibilities:

\begin{enumerate}
\item The recording tableau of $x^{\#}$ is the same as the recording tableau of $(f_{-1}(x))^{\#}$. This implies that during 
the insertion of $x^{\#}$, the final  two $(n+1)$'s to be inserted are at no point in the same row.  (Note that this is clearly
impossible if $j=k$.) This means, that after the insertion of the final two $(n+1)$'s, the rightmost $n+1$ is 
never inserted into another row containing an $n+1$, and, moreover, there is never an $n$ being inserted into the row 
containing the rightmost $n+1$ (since after the insertion of the final two $(n+1)$'s, the rightmost $n$ or $n+1$ is always $n+1$). 
In this case, $P((f_{-1}(x)^{\#})$ is obtained from $P(x^{\#})$ by changing the $n+1$ in row $n+1-k$ into an $n$.  Since
$x^{\#}$ and $(f_{-1}(x))^{\#}$ have the same recording tableau, $x$ and $f_{-1}(x)$ are in the same connected component.  
Since it is evident from $P((f_{-1}(x)^{\#})$ that $f_j^1   f_k \cdots f_2 (f_{-1}(x))$ must be $I_0$-lowest weight, it follows that 
$v=f_j^1 f_k ^{\bar{1}}e_1^j e_1^k(v)$.  This is precisely what happens in the example above; $P((f_{-1}(x)^{\#})$ is 
obtained from $P(x^{\#})$ by:
\[
	\young(44,339,222,1119)  \quad \raisebox{0.7cm}{$\longrightarrow$} \quad
	\young(44,339,222,1118)\;.
\]
Hence \textbf{C2'}(a) holds.
\item The recording tableau of $x^{\#}$ differs from the recording tableau of $(f_{-1}(x))^{\#}$. This implies that during 
the insertion of  $x^{\#}$, there is some point at which the final  two $(n+1)$'s to be inserted are in the same row. Call
 this row $r$ and suppose that this occurs during the insertion of the $i^{th}$ letter of  $x^{\#}$.  Let $P_i$ be the tableau 
 obtained from inserting the first $i$ letters of  $x^{\#}$ and let $P_i '$ be the tableau obtained from inserting the first $i$
 letters of $(f_{-1}(x))^{\#}$.  Then $P_i'$ is obtained from $P_i$ by changing the second to rightmost $n+1$ to $n$ and 
 moving the rightmost $n+1$ from row $r$ to some row $s>r$.

  Now continue with the insertion of the $(i+1)^{st}$ letter in each case. Since the $(n,n+1)$-subword of $x^{\#}$ ends 
  with two $(n+1)$'s, and these are the only $(n,n+1)$-unbracketed $(n+1)$'s in this subword, the same is true of the 
  $(n,n+1)$-subword of each of $P_i, P_{i+1}, \ldots, P_\ell$.  This implies that  at no point in the rest of the insertion of 
  $x^{\#}$ is the second to rightmost $n+1$ inserted into a row containing another $n+1$, and moreover at no point 
  is an $n$ inserted into the row containing the second to rightmost $n+1$ (since after the insertion of the final two 
  $(n+1)$'s, the two rightmost entries which are either $n$ or $n+1$ must both be $n+1$).

It follows that, if we ignore, the rightmost $n+1$ in $P((f_{-1}(x)^{\#})$ and $P(x^{\#})$, then they have the same shape, 
and the second differs from the first only by changing its rightmost $n$ to $n+1$.  Adding back the rightmost $n+1$ to 
$P(x^{\#})$, we see that it must go somewhere to the right of this position (by definition), and  adding back the rightmost 
$n+1$ to $P(f_{-1}(x^{\#}))$, we see that it must go somewhere to the left of this position (otherwise $P((f_{-1}(x)^{\#})$ 
would have an $(n,n+1)$-unbracketed $n+1$.) 

It follows that $P((f_{-1}(x)^{\#})$ is obtained from $P(x^{\#})$ by eliminating the (rightmost) $n+1$ in row $n-k+1$, 
changing the (leftmost) $n+1$ in row $n-j+1$ to $n$ and adding an $n+1$ to some row $n-h+1$ for $h<j$.
It follows that $v'=f_h^1f_j^{\bar{1}}e_1^je_1^k(v)$ and $v$ are both (distinct) $I_0$-lowest weight elements.
Hence \textbf{C2'}(b) holds.

To see an example of the second case, let $v=99889$.  Then $v^{\#}=12211$, $(e_1^7e_1^8(v))^{\#}=29911$,  
$(f_{-1}e_1^7e_1^8(v))^{\#}=29811$, and $(f_6^1f_7^{\bar{1}}e_1^7e_1^8(v))^{\#}=23211$ have the following 
insertion tableaux:
\[
	\young(22,111) \quad \raisebox{0.4cm}{$\longrightarrow$} \quad \young(29,119)
	\quad \raisebox{0.4cm}{$\longrightarrow$} \quad \young(9,28,11)
	\quad \raisebox{0.4cm}{$\longrightarrow$} \quad \young(3,22,11) \;.
\]
\end{enumerate}

\subsection{Proof of Lemma~\ref{lemma.clean} for $j=n-1$ and $j'=n$}
\label{section.proof clean 1}

Define $X=(e_1 \cdots e_n)v$.  
For $1 \leqslant i \leqslant n+1$, set $A_i=(e_i \cdots e_n)X$ and  $B_i=(e_i \cdots e_{n-1})X$. For $2 \leqslant i \leqslant n+1$,
set $A_{-i}=( f_{(i-1)}\cdots f_2f_{-1})A_1$ and  $B_{-i}=( f_{(i-1)}\cdots f_2f_{-1})B_1$.  (So $A_1=A_{-1}$ and 
$B_1=B_{-1}$.  Moreover, $B_{n+1}=B_n$.)  By assumption $(f_h \cdots f_1)(B_{-n})$ is $I_0$-lowest weight, so 
$f_{n}(f_h \cdots f_1)(B_{-n})=0$ and hence $B_{-(n+1)}=0$.

Let $x_i$ be the integer which represents the position where $A_{i+1}$ and $A_i$ differ, and $y_i$ be the integer which 
represents the position where $B_{i+1}$ and $B_i$ differ.  Also, let $x_{-i}$ be the integer which represents the position where
$A_{-i}$ and $A_{-(i+1)}$ differ, and let $y_{-i}$ be the integer which represents the position where $B_{-i}$ and $B_{-(i+1)}$ 
differ. Note that $y_n$ and $y_{-n}$ are undefined.

Recall that $v\in \mathcal{B}^{\otimes \ell}$. Suppose $W$ is any word of length $\ell$ in the letters $\{1, \ldots, n+1$\}. 
If $1 \leqslant p \leqslant \ell$, we define $W(p)$ to be the $p^{th}$ entry of $W$. If $ 1 \leqslant p \leqslant q \leqslant \ell$ 
are integers, then the notation $W(p:q)$ will be used to refer to the word $W(p) W(p+1) \ldots W(q-1) W(q)$.

If $1 \leqslant i \leqslant n$, we define the $i/(i+1)$-subword of $W$ to be the word composed of the symbols $\{i,i+1,\_\}$ 
which is obtained from $W$ by changing each entry that is neither $i$ nor $i+1$ to the symbol $\_$.  For instance the 
$2/3$-subword of $241432143$ is $2\,\_\,\_\,\_\,32\,\_\,\_\,3$.  When we speak of erasing an $i$ or $i+1$, we mean changing 
that entry to $\_$ ; similarly, when we speak of adding an $i$ or $i+1$, we mean changing some $\_$ to $i$ or $i+1$.  Moving an 
$i$ or $i+1$ from $p$ to $q$ means erasing an $i$ or $i+1$ from position $p$ and adding an $i$ or $i+1$ to position $q$. 
The notation $W(p:q)$ is used in the same way for subwords as it is for words.  For instance,
if W=$3\,\_\,\_\,\_\,32\,\_\,\_\,3$ then $W(3:7)=\,\_\,\_\,32\,\_$.

\begin{claim}
For $2 \leqslant i\leqslant n$, we have $x_i\geqslant x_{i-1}$.   
For $2 \leqslant i \leqslant n-1$, we have $y_i\geqslant y_{i-1}$.
\end{claim}

\begin{proof}
If $x_i<x_{i-1}$, then it follows that $f_i A_{i-1} \neq 0$.  But this is the statement that
\[
	f_i(e_{i-1}e_i \cdots e_n)(e_1 \cdots e_n)v \neq 0
\]
for some  integer $2 \leqslant i \leqslant n$, which is absurd since $v$ is $I_0$-lowest weight. 
If $y_i<y_{i-1}$, then it follows that $f_i B_{i-1} \neq 0$. But this is the statement that
\[
	f_i(e_{i-1}e_i \cdots e_{n-1})(e_1 \cdots e_n)v \neq 0
\]
for some  integer $2 \leqslant i \leqslant n-1$, which is also absurd. 
\end{proof}

\begin{claim}
We have $x_1>x_{-1}$ and $y_1>y_{-1}$.  (In particular, $f_{-1}(A_1) \neq 0$, so $x_{-1}$ is well-defined.)
\end{claim}

\begin{proof}
By the definition of the operator $f_{-1}$ we have $y_1 \geqslant y_{-1}$.  Since $v$ and $v^*:=f_h^1 f_{n-1}^{\bar{1}}
e_1^{n-1} e_1^n v$ are both $I_0$-lowest weight and have different weights,
we cannot have $y_1=y_{-1}$.  Thus $y_1>y_{-1}$.  Now $B_{n}(1:y_{-1})=B_{1}(1:y_{-1})$.  
Therefore, there are no $1$'s or $2$'s in $B_{n}(1:y_{-1}-1)$ and we have $B_{n}(y_{-1})=1$ since these statements must be 
true of $B_1$.  If $x_1>y_{-1}$, then $A_1(1:y_{-1})=B_1(1:y_{-1})$ and so $A_{-2} \neq 0$ with $x_{-1}=y_{-1}$. If 
$x_1< y_{-1}$, then $A_1(1:x_1-1)=B_{n}(1:x_1-1)$ contains no $1$'s or $2$'s and $A_1(x_1)=1$.  Thus  $A_{-2} \neq 0$ 
with $x_{-1}=x_{1}$.  It is clearly impossible for $x_{1}=y_{-1}$.  Therefore, we have established that 
$A_{-2}=f_{-1}(A_1) \neq 0$.  In the notation of Proposition~\ref{proposition.ins}, we have for $j=k=n$, that 
$f_{-1}e_1^je_1^k(v) \neq 0$.  Hence we must be in case \textbf{C2'}(b) from which we deduce that 
$f_{-1}(A_1)$ lies in a different $I_0$-connected component than $A_{1}$.  From this it follows that $x_1>x_{-1}$.
\end{proof}

\begin{claim}
For $2 \leqslant i \leqslant n$, we have $x_{-(i-1)} \leqslant x_{-i}$.  
For $2 \leqslant i \leqslant n$, we have $y_{-(i-1)} \leqslant y_{-i}$.  
(In particular, $A_{-3}, \ldots, A_{-(n+1)}$ are nonzero, so $x_{-2}, \ldots, x_{-n}$ are well-defined.)
\end{claim}

\begin{proof}
Again, case \textbf{C2'}(b) applies to $f_{-1}(A_1)$ and so the parenthetical statement is immediate.  
First, it is clear from the definitions of the $f_{-1}$ and $f_2$ operators that $x_{-1}\leqslant x_{-2}$ and that 
$y_{-1} \leqslant y_{-2}$.  If $x_{-(i-1)}>x_{-i}$ for $i>2$, then it follows that $f_i A_{-(i-1)} \neq 0$.  But this is the statement 
that $f_i(e_{i-1}e_i \cdots e_n)(e_1 \cdots e_g) \hat{v} \neq 0$ for some $I_0$-lowest weight element $\hat{v}$ and integers 
$3 \leqslant i \leqslant n$ and $0 \leqslant g <n$ which is absurd.  If $y_{-(i-1)}>y_{-i}$ for $i>2$, then it follows that 
$f_i(B_{-(i-1)})\neq 0$. But this is the statement that $f_i(e_{i-1}e_i \cdots e_{n-1})(e_1 \cdots e_g)v^* \neq 0$ for 
some  integers $3 \leqslant i \leqslant n$ and $0 \leqslant g <n$ which is equally absurd.  
\end{proof}

So far, we have the following situation:
\begin{equation*}
\begin{split}
 x_n & \geqslant \cdots \geqslant x_2 \geqslant x_1 > x_{-1} \leqslant x_{-2} \leqslant \cdots \leqslant x_{-n}
 \qquad \text{and}\\
 y_{n-1} & \geqslant \cdots \geqslant y_2 \geqslant y_1 > y_{-1} \leqslant y_{-2} \leqslant \cdots  \leqslant y_{-(n-1)}.
\end{split}
\end{equation*}

\begin{claim}
We have $x_{-1}=y_{-1}$.
\end{claim}

\begin{proof}
Since $x_1 = y_{-1}$ is impossible and since $x_1<y_{-1}$ would imply that $x_{-1}=x_1$, which contradicts $x_1>x_{-1}$, 
we may assume $x_1>y_{-1}$.  However, in this case we have $A_1(1:y_{-1})=B_1(1:y_{-1})$.  Since $f_{-1}$ acts on $B_1$ 
in position $y_{-1}$, it follows that $f_{-1}$ acts on $A_1$ in position $y_{-1}$ as well.  This implies $x_{-1}=y_{-1}$.
\end{proof}

\begin{claim}
For $1 \leqslant i \leqslant n-1$, we have $x_i \leqslant y_i$.
\end{claim}

\begin{proof}
First we show that $x_{n-1} \leqslant y_{n-1}$. Now $y_{n-1}$ represents the position of the leftmost $(n-1,n)$-unbracketed 
$n$ in $B_n$. This $n$ is also unbracketed in $A_n$ because the $(n-1)/n$-subword of $A_n$ is obtained from the
$(n-1)/n$-subword of $B_n$ by inserting an $n$. Hence the leftmost $(n-1,n)$-unbracketed $n$ in $A_n$ is weakly to the 
left of position $y_{n-1}$, so $x_{n-1} \leqslant y_{n-1}$. Next, suppose that $x_{i+1} \leqslant y_{i+1}$ but $x_i > y_i$.   
The $i/(i+1)$-subword of $A_{i+1}$ only differs from the $i/(i+1)$-subword of $B_{i+1}$ by moving an $i+1$ to the left from 
$y_{i+1}$ to $x_{i+1}$.  Since $y_i<x_{i+1}$ by assumption, the $i+1$ which appears in $B_{i+1}(y_i)$ still appears in 
$A_{i+1}(y_i)$ and is $(i,i+1)$-unbracketed.  This implies $x_i \leqslant y_i$.  Induction completes the proof.
\end{proof}

\begin{claim}
For $1 \leqslant i \leqslant n$, we have $x_i \geqslant x_{-i}$. For $1 \leqslant i \leqslant n-1$, we have $y_i \geqslant y_{-i}$.
\end{claim}

\begin{proof}
We already know that $x_1 \geqslant x_{-1}$. So assume that $x_{i-1} \geqslant x_{-(i-1)}$ but $x_i<x_{-i}$.
The $i/(i+1)$-subword of $A_{i}$ is obtained from the $i/(i+1)$-subword of $A_{-i}$ by moving an $i$ to the right from 
$x_{-(i-1)}$ to $x_{i-1}$.  Since $A_{-i}(x_{-i})$ contains an $(i,i+1)$-unbracketed $i$ and $x_{i-1}<x_{-i}$, we see that 
$A_{i}(x_{-i})$ still contains an $(i,i+1)$-unbracketed $i$. This implies that $x_i \geqslant x_{-i}$. Induction completes the proof.
The second statement is proved in the same way.
\end{proof}

From the previous result, we have the following situation:
\[
\begin{array}{ccccccccccccccc}
 	\cdots & \geqslant & x_3 & \geqslant & x_2  & \geqslant & x_1  &> & x_{-1} & \leqslant & x_{-2} & \leqslant 
 	& x_{-3} & \leqslant & \cdots \\
  	&& \mathbin{\rotatebox[origin=c]{270}{$\leqslant$}} &  & \mathbin{\rotatebox[origin=c]{270}{$\leqslant$}} &  & 
	\mathbin{\rotatebox[origin=c]{270}{$\leqslant$}} &  & || &  &  & &  &&  \\
	\cdots &  \geqslant & y_3 & \geqslant & y_2  & \geqslant & y_1 & > & y_{-1} & \leqslant & y_{-2}  & \leqslant & 
	y_{-3} & \leqslant & \cdots  
\end{array}
\]
where every entry on the left side of the array is $\geqslant$ to its mirror image on the right side of the array.
From now on, let $j$ be minimal such that $x_{j}<y_{j}$; if no such $j$  exists, set $j=n$.

\begin{claim}
We have $x_i = y_i$ for all $i <j$ and $x_{i+1}  < y_i$ for all $j \leqslant i <n$.
\end{claim}  

\begin{proof}
The first claim is immediate.  Next we note that $x_i<y_i$ for all $i \geqslant j$.  (Otherwise $x_i  = y_i$ for some $i \geqslant j$.
This implies that $x_k=y_k$ for all $k \leqslant i$, and, in particular, $x_j=y_j$.) By definition, we have $B_{i+1}(y_i)=i+1$ 
and $A_{i+2}(x_{i+1})=i+2$.  From the latter, it follows that $B_{i+2}(x_{i+1})\geqslant i+2$ and, since $y_{i+1}>x_{i+1}$ 
(or $y_{i+1}$ is undefined) that $B_{i+1}(x_{i+1}) \geqslant i+2$.  Therefore, we have $x_{i+1} \neq y_i$.  If $x_{i+1}>y_i$,
we must have $x_i<x_{i+1}$ and $y_i<y_{i+1}$ from which it follows that $A_{i+1}(1:y_i)=B_{i+1}(1:y_i)$. But this makes 
$x_i<y_i$ impossible.  By contradiction, we conclude that $x_{i+1} < y_i$.
\end{proof}

\begin{claim} \label{mir}
For $i <j$ we have $x_{-i}=y_{-i}$.  Also, $x_j>x_{j-1}$.
\end{claim}

\begin{proof}
Since the restrictions of $A_{j-1}$ and  $B_{j-1}$ to the alphabet $\{1,2,\ldots,j-1\}$ are identical, and since the operators 
$e_{j-2},\ldots, e_1,f_{-1},f_2,\ldots, f_{j-2}$ only depend on and effect these letters, it follows that for $i\leqslant j-2$ we 
have $x_{-i}=y_{-i}$. Now we must show $x_{-(j-1)}=y_{-(j-1)}$. We have $A_{j+1}(x_j)=j+1$ and thus 
$B_{j+1}(x_j)\geqslant j+1$, and hence by $x_j<y_j$, $B_{j}(x_j)\geqslant j+1$. Since $B_j(y_{j-1})=j$, this yields
$x_j \neq y_{j-1}$.  In light of $x_{j-1}=y_{j-1}$ this gives $x_j \neq x_{j-1}$.  From this it follows that 
$A_{j} (1:x_{j-1})=B_j(1:x_{j-1})$. By the minimality of $j$ and by the result for $i \leqslant j-2$ this implies that 
$A_{-(j-1)}(1:x_{j-1})=B_{-(j-1)}(1:x_{j-1})$.  Since we have both $x_{-(j-1)} \leqslant x_{j-1}$ and $y_{-(j-1)} \leqslant y_{j-1}$, 
the previous equality implies that $x_{-(j-1)}=y_{-(j-1)}$.
\end{proof}

If $1<i<n$, let $\#(A_{-i}(p:q))$ denote the number of $i$'s minus the number of $(i+1)$'s which appear in $A_{-i}(p:q)$. 
Define $\#(B_{-i}(p:q))$ analogously.  Set $AB_i(p:q)=\#(A_{-i}(p:q))-\#(B_{-i}(p:q))$.

\begin{claim} 
\label{ineq}
Suppose $1<i<n$.
\begin{enumerate}
\item{If $x_{-i}<y_{-i}$, then $AB_i(1:x_{-i})>0$.}
\item{ If $x_{-i}>y_{-i}$, then $AB_i(1:y_{-i})<0$.}
\item{If $x_{-i}<y_{-i}$, then $AB_i(x_{-i}+1:y_{-i})<0$.}
\item{If $x_{-i}<y_{-i}$, $x_{-i}=x_{i}$, $x_i \neq x_{i+1}$, and $x_{i} \neq y_i$, then $AB_i(x_{-i}+1:y_i)<-1$.}
\end{enumerate}
\end{claim}

\begin{proof}
 Once again, \textbf{C2'}(b) applies to $f_{-1}(A_1)$ and so we may write $A_{-i}=e_i \cdots e_n e_1^{h'}(v')$ 
 for some $I_0$-lowest weight element $v'$ and some $h'<n$.  It follows that $A_{-i}$ has exactly one $(i,i+1)$-unbracketed 
 $i$ and it occurs in $x_{-i}$.  In addition, case \textbf{C2'}(b) applies to $f_{-1}(B_1)$ by assumption, so 
 $B_{-i}=e_i \cdots e_{n-1}e_1^h(v^*)$ for an $I_0$-lowest weight element $v^*$. Hence $B_{-i}$ has exactly one 
 $(i,i+1)$-unbracketed $i$ and it occurs in $y_{-i}$.  Thus we have $\#(A_{-i}(1:x_{-i}))>0$ and $\#(B_{-i}(1:y_{-i}))>0$. 
 If $x_{-i}<y_{-i}$ then  $\#(B_{-i}(1:x_{-i})) \leqslant 0$, while if $x_{-i}>y_{-i}$ then $\#(A_{-i}(1:y_{-i})) \leqslant 0$.  
 Together this proves the first two statements.
 For the third statement we have $\#(A_{-i}(x_{-i}+1:y_{-i})) \leqslant 0$ and $\#(B_{-i}(x_{-i}+1:y_{-i}))>0$. 
 For the fourth statement, again, we have $\#(A_{-i}(x_{-i}+1:y_{-i})) \leqslant 0$, but now note that $A_{i+1}(x_{i})=i+1$.  
 Since $x_{i} \neq x_{i+1}$, also, $A_{i+2}(x_{i})=i+1$, whence $B_{i+1}(x_i)=i+1$, and, by, $x_i \neq y_i$, we have
 $B_{i}(x_i)=i+1$.  This now implies that $B_{-i}(x_i)=i+1$ or $B_{-i}(x_{-i})=i+1$. Since the $i$ in $B_{-i}(y_i)$ must be 
 $(i,i+1)$-unbracketed this implies that $\#(B_{-i}(x_{-i}+1:y_{-i}))>1$.
\end{proof}

\begin{claim}
\label{func}
Fix an interval $[p,q]$.  We define the function $[t]$ by $[t]=1$ if $t \in [p,q]$ and $[t]=0$ otherwise.  
With this notation, we have that 
\[
 	AB_i(p:q)=[x_{-(i-1)}]-[x_{i-1}]+2[x_i]-[x_{i+1}]+[y_{i+1}]-2[y_i]+[y_{i-1}]-[y_{-(i-1)}].
\]
\end{claim}

\begin{proof}
This is a straightforward computation.
\end{proof}

\begin{claim}
\label{samesies}
Suppose $j<n$. If either $x_j>x_{-j}$ or $y_j>y_{-j}$, then both $x_j>x_{-j}$ and $y_j>y_{-j}$.  
In this case we have $x_{-j}=y_{-j}$.
\end{claim}

\begin{proof}
If $j=1$, the conclusions of the claim have already been proven in previous claims. Thus assume $j>1$.  First note that, 
since $x_{-(j-1)}=y_{-(j-1)}$ and $x_{j-1}=y_{j-1}$, we have $AB_j(p:q)=2[x_j]-[x_{j+1}]+[y_{j+1}]-2[y_j]$. To prove the first 
statement, we will show that both (1) $x_j>x_{-j}$ and $y_j=y_{-j}$ and (2) $y_j>y_{-j}$ and $x_j=x_{-j}$ are impossible.

First suppose that $x_j>x_{-j}$ and that $y_j=y_{-j}$. Since $x_{-j}<x_j<y_j=y_{-j}$, we have by Claim~\ref{ineq} that 
$AB_j(1:x_{-j})>0$. However, $x_j,x_{j+1},y_{j+1},y_j$ are each $>x_{-j}$ so by Claim~\ref{func} we have $AB_j(1:x_{-j})=0$. 
Hence,  $x_j>x_{-j}$ and  $y_j=y_{-j}$ is impossible.

Now suppose that $y_j>y_{-j}$ and that $x_j=x_{-j}$.

\smallskip \noindent
\textbf{Case 1:} $y_{-j}<x_{-j}$.  Since $y_{-j}<x_{-j}$ we have by Claim~\ref{ineq} that $AB_j(1:y_{-j})<0$. However, 
$x_j,x_{j+1},y_{j+1},y_j$ are each $>y_{-j}$ so by Claim~\ref{func} we have $AB_j(1:y_{-j})=0$.

\smallskip \noindent
\textbf{Case 2:} $y_{-j}=x_{-j}$.  We have $A_{j+1}(x_j)=j+1$ and so $B_{j+1}(x_j)\geqslant j+1$. Hence by $x_j<y_j$ 
we have $B_{j}(x_j)\geqslant j+1$ which gives $B_{-j}(x_j)\geqslant j+1$. However, by definition $B_{-j}(y_{-j})=j$ so this 
makes $x_{-j}=y_{-j}$ impossible in light of  $x_j=x_{-j}$.

\smallskip \noindent
\textbf{Case 3a:} $y_{-j}>x_{-j}$ and $x_j=x_{j+1}$. Since $y_{-j}>x_{-j}$ we have by Claim~\ref{ineq} that 
$AB_j(x_{-j}+1:y_{-j})<0$. However, $x_j, x_{j+1}$ are each $<x_{-j}+1$ and $y_j, y_{j+1}$ are each $>y_{-j}$ so by 
Claim~\ref{func} we have $AB_j(1:y_{-j})=0$.

\smallskip \noindent
\textbf{Case 3b:} $y_{-j}>x_{-j}$ and $x_j<x_{j+1}$. Since $y_{-j}>x_{-j}=x_{j}$, $x_j \neq x_{j+1}$, and 
$x_{j} \neq y_j$, we have by Claim~\ref{ineq} that $AB_j(x_{-j}+1:y_{-j})<-1$. However, $x_j<x_{-j}+1$ and 
$y_j, y_{j+1}$ are each $>y_{-j}$ so by Claim~\ref{func} we have $AB_j(x_{-j}+1:y_{-j}) \in \{-1,0\}$.

\smallskip

Hence  $y_j>y_{-j}$ and  $x_j=x_{-j}$ is impossible.  This establishes that if either $x_j>x_{-j}$ or $y_j>y_{-j}$, then both 
$x_j>x_{-j}$ and $y_j>y_{-j}$.  

Now assume that both  $x_j>x_{-j}$ and $y_j>y_{-j}$.  If $x_{-j}<y_{-j}$, we have by Claim~\ref{ineq} that 
$\#_j(A_{-j}(1:x_{-j}))>0$. However, $x_j,x_{j+1},y_{j+1},y_j$ are each $>x_{-j}$ so by Claim~\ref{func} we have 
$\#_j(A_{-j}(1:x_{-j}))=0$.  If $x_{-j}>y_{-j}$, we have by Claim~\ref{ineq} that $\#_j(A_{-j}(1:y_{-j}))<0$. However, 
$x_j,x_{j+1},y_{j+1},y_j$ are each $>x_{-j}$ so by Claim~\ref{func} we have $\#_j(A_{-j}(1:y_{-j}))=0$. Hence $x_{-j}=y_{-j}$.
\end{proof}

\begin{claim}
If $x_j<x_{-j}$ or $y_{j}<y_{-j}$, then for $j \leqslant i <n$ we have $y_{-i}<y_i$ and $y_{-i} \leqslant x_{-i}$.
\end{claim}

\begin{proof}
We proceed by induction.  By the first statement of Claim~\ref{samesies}, we can be sure that $y_{-j}<y_j$.  By the second 
statement of Claim~\ref{samesies} we can be sure that $y_{-j}=x_{-j}$, so in particular, $y_{-j} \leqslant x_{-j}$. Therefore the 
claim holds for $i=j$.  Now let $i>j$ and suppose that the claim holds for $i-1$ so that $y_{-(i-1)}<y_{i-1}$ and 
$y_{-(i-1)} \leqslant x_{-(i-1)}$.  We will show that under this assumption, each of (1) $y_{-i}=y_i$ and $y_{-i}>x_{-i}$, 
(2)  $y_{-i}<y_i$ and $y_{-i}>x_{-i}$, and (3) $y_{-i}=y_i$ and $y_{-i}\leqslant x_{-i}$ is impossible.

First suppose that $y_{-i}=y_i$ and that $y_{-i}>x_{-i}$.  

\smallskip \noindent
\textbf{Case 1:} $x_{-i}<x_{i}$.  Since $y_{-i}>x_{-i}$ by Claim~\ref{ineq} we have $AB_i(1:x_{-i})>0$.  However, by 
assumption $x_i,x_{i+1},y_{i+1},y_i,y_{i-1}$ are each $>x_{-i}$  and $x_{-(i-1)}=y_{-(i-1)}$ so the only possible relevant 
change is at $x_{i-1}$.  Thus by Claim~\ref{func} we have $AB_i(1:y_{-i})\in \{-1,0\}$.

\smallskip \noindent
\textbf{Case 2a:} $x_{-i}=x_i$ and $x_i=x_{i+1}$. Since $y_{-i}>x_{-i}$ by Claim~\ref{ineq} we have $AB_i(1:x_{-i})>0$.  
By assumptions, each of  $x_{-(i-1)},x_{i-1},x_i,x_{i+1},y_{-(i-1)}$ are $< x_{-i}+1$.  Clearly $y_i=y_{-i} \in [x_{-i}+1:y_{-i}]$.  
Moreover, $y_{i-1}\leqslant y_i =y_{-i}$ and $y_{i-1}>x_i=x_{-i}$, so $y_{i-1}\in [x_{-i}+1:y_{-i}]$.  Without computing the value 
of $[y_{i+1}]$ we may conclude by Claim~\ref{func} that $AB_i(1:y_{-j})\in \{-1,0\} $.

\smallskip \noindent
\textbf{Case 2b:} $x_{-i}=x_i$ and $x_i<x_{i+1}$. Since $y_{-i}>x_{-i}$, $x_{-i}=x_{i}$, $x_i \neq x_{i+1}$, and 
$x_{i} \neq y_i$ we have by Claim~\ref{ineq} that $AB_i(x_{-i}+1:y_{-i})<-1$. By assumptions, each of
$x_{-(i-1)},x_{i-1},x_i,y_{-(i-1)}$ are $< x_{-i}+1$.  Again, we know that $y_i,y_{i-1}\in [x_{-i}+1:y_{-i}]$. Without computing 
 the value of $[y_{i+1}]$ and $[x_{i+1}]$ we may compute by Claim~\ref{func} that $AB_i(x_{-i}+1:y_{-i}) \in \{-1,0,1\}$.
 
\smallskip
Hence it is impossible that $y_{-i}=y_i$ and that $y_{-i}>x_{-i}$.  Now suppose that $y_{-i}<y_i$ and that $y_{-i}>x_{-i}$.

\smallskip \noindent
\textbf{Case 1a:} $x_{-i}<x_i$ and $x_i \leqslant y_{-i}$. 
Since $y_{-i}>x_{-i}$, we have by Claim~\ref{ineq} that $AB_i(x_{-i}+1:y_{-i})<0$. We have that $x_{-(i-1)}, y_{-(i-1)}$ are 
both $< x_{-i}+1$, that $x_i \in [x_{-i}+1:y_{-i}]$ and that $y_i, y_{i+1}$ are both $>y_{-i}$.  Without computing $[x_{i-1}], 
[x_{i+1}], [y_{i-1}]$  we may determine by Claim~\ref{func} that $AB_i(x_{-i}+1:y_{-i}) \in \{0,1,2,3\}$.

\smallskip \noindent
\textbf{Case 1bi:} $x_{-i}<x_i$, $x_i > y_{-i}$, and $x_{i-1} \leqslant x_{-i}$.  Since $y_{-i}>x_{-i}$, we have by Claim~\ref{ineq} 
that $AB_i(x_{-i}+1:y_{-i})<0$.  By assumption each of $x_{-(i-1)},x_{i-1},y_{-(i-1)}$ are $< x_{-i}+1$ and $x_{i+1}, x_i, y_i, 
y_{i+1}$ are $> y_{-i}$. Without computing  $[y_{i-1}]$ we may determine by Claim~\ref{func} that $AB_i(x_{-i}+1:y_{-i}) 
\in \{0,1\}$.

\smallskip \noindent
\textbf{Case 1bii:} $x_{-i}<x_i$, $x_i > y_{-i}$, and $x_{i-1} > x_{-i}$. Since $y_{-i}>x_{-i}$, we have by Claim~\ref{ineq} that
$AB_i(1:x_{-i})<0$. By assumption $x_{-(i-1)}, y_{-(i-1)}$ are $\leqslant x_{-i}$ whereas each of 
$x_{i-1},x_i,x_{i+1},y_{i-1},y_i,y_{i+1}$ are $> x_{-i}$.  Thus  by Claim~\ref{func}, we have  $AB_i(1:x_{-i})=0$.

\smallskip \noindent
\textbf{Case 2a:} $x_{-i}=x_i$ and $x_i=x_{i+1}$.  Since $y_{-i}>x_{-i}$ we have by Claim~\ref{ineq} that 
$AB_i(x_{-i}+1:y_{-i})<0$.  By assumption each of $x_{-(i-1)},x_{i-1},x_i,x_{i+1},y_{-(i-1)}$ are $< x_{-i}+1$ and $y_i, y_{i+1}$ 
are $> y_{-i}$. Without computing  $[y_{i-1}]$ we may determine by Claim~\ref{func} that $AB_i(x_{-i}+1:y_{-i}) \in \{0,1\}$.

\smallskip \noindent
\textbf{Case 2b:} $x_{-i}=x_i$ and $x_i<x_{i+1}$. Since $y_{-i}>x_{-i}$, $x_{-i}=x_{i}$, $x_i \neq x_{i+1}$, and $x_{i} \neq y_i$ 
we have by Claim~\ref{ineq} that $AB_i(x_{-i}+1:y_{-i})<-1$.   By assumption each of $x_{-(i-1)},x_{i-1},x_i,y_{-(i-1)}$ are
$< x_{-i}+1$ and $y_i, y_{i+1}$ are $> y_{-i}$. Without computing  $[y_{i-1}]$ and $[x_{i-1}]$ we may determine by 
Claim~\ref{func} that $AB_i(x_{-i}+1:y_{-i}) \in \{-1,0,1\}$.

\smallskip
Hence $y_{-i}<y_i$ and $y_{-i}>x_{-i}$ is impossible.  Now suppose  $y_{-i}=y_i$ and $y_{-i}\leqslant x_{-i}$.  This would 
imply $y_{i}=y_{-i} \leqslant x_{-i} \leqslant x_i < y_i$ which is absurd. The three possibilities listed in the beginning of the 
proof are thus impossible, and the only remaining one is $y_{-i}<y_i$ and $y_{-i}\leqslant x_{-i}$.
\end{proof}

Supposing $j=3$, and $n=5$, and $x_j>x_{-j}$ our situation would look as follows:
\[
\begin{array}{cccccccccccccccc}
	x_5 & \geqslant & x_4 & \geqslant  \mathbf{x_3}  > & x_2  & \geqslant & x_1  &> & x_{-1} & \leqslant & x_{-2} & 
	\leqslant & x_{-3} &  \leqslant & x_{-4} & \geqslant x_{-5} \\
	\wedge &  &\wedge &  & || &  & || &  &|| &  & || &   & || & & \mathbin{\rotatebox[origin=c]{270}{$\leqslant$}} & \\
 	\mathbf{y_4}  & \geqslant & \mathbf{y_3} &  \geqslant & y_2  & \geqslant & y_1 & > & y_{-1} & \leqslant & y_{-2}  
	& \leqslant & y_{-3}  & \leqslant & y_{-4}&
\end{array}
\]
where again every entry on the left side of the array is $\geqslant$ its mirror image on the right side of the array, and the 
bold entries are bigger than their mirror image.

\begin{claim}
\label{claim.A=B}
If $x_j=x_{-j}$, then $A_{-(n+1)}=B_{-n}$.
\end{claim}

\begin{proof}
We have for all $i<j$ that $x_i=y_i$ and $x_{-i}=y_{-i}$. Since by assumption $x_j=x_{-j}$, we have for all $i \geqslant j $, 
$x_i=x_{-i}$.  Moreover, if $j<n$  then by Claim~\ref{samesies} $y_j=y_{-j}$ and for all $i \geqslant j$, we have $y_i=y_{-i}$. 
If $\ell$ is the length of the word $v$ and $1 \leqslant p \leqslant \ell$, define the vector $\vec{p}$ to be the vector of length 
$\ell$, which has a $1$ in position $p$ and $0$'s elsewhere. Then recalling that $A_{n+1}=X=B_n$, we have the equalities:
\begin{multline*}
A_{-(n+1)}=X-\sum_{i=1}^n \vec{x}_i + \sum_{i=1}^n \vec{x}_{-i}
=X-\sum_{i=1}^{j-1} \vec{x}_i + \sum_{i=1}^{j-1} \vec{x}_{-i}
=X-\sum_{i=1}^{j-1} \vec{y}_i + \sum_{i=1}^{j-1} \vec{y}_{-i}\\
=X-\sum_{i=1}^{n-1} \vec{y}_i + \sum_{i=1}^{n-1} \vec{y}_{-i} =B_{-n}.
\end{multline*}
\end{proof}

\begin{claim}
We have $x_j=x_{-j}$.
\end{claim}

\begin{proof}
Suppose $x_j>x_{-j}$.

\smallskip \noindent
\textbf{Case 1:} $j=n$. By the definition of $j$, we have $x_{n-1}=y_{n-1}$ and by Claim~\ref{mir} we have 
$x_{-(n-1)}=y_{-(n-1)}$. Since $x_{-n}<x_n$, this implies $A_{-n}(1:x_{-n})=B_{-n}(1:x_{-n})$.  Since $A_{-n}$ contains an 
$(n,n+1)$-unbracketed $n$ in position $x_{-n}$, so does $B_{-n}$. Therefore, $f_n(B_{-n})\neq 0$ which contradicts 
$B_{-(n+1)}=0$.

\smallskip \noindent
\textbf{Case 2a:} $j<n$ and $x_{n-1}=x_{-(n-1)}$. We have $y_{-(n-1)} \leqslant x_{-(n-1)} \leqslant x_n$. Since $x_n<y_{n-1}$ 
this means that we cannot have $y_{-(n-1)} =x_n$, so we must have $y_{-(n-1)}<x_n$.  Since $x_{n-1}=x_{-(n-1)}$ and 
$y_{n-1}>x_{n}$, the $n/(n+1)$-subword of $B_{-n}(1:x_{n})$ is obtained from the $n/(n+1)$-subword of $A_{n}(1:x_{-n})$
 by:
\begin{enumerate}
\item{Erasing an $n$ from $x_{n}$ and adding an $n$ in $y_{-(n-1)}$. (Note $y_{-(n-1)}<x_n$.)}
\item{Adding an $n+1$ to $x_n$. }
\end{enumerate}
Therefore, since the  $n/(n+1)$-subword of $A_{-n}(1:x_{n})$ contains an $(n,n+1)$-unbracketed $n$ and each one of 
these two steps does not change that property, the $n/(n+1)$-subword of $B_{-n}(1:x_{n})$ also does. This implies $
f_n(B_{-n})\neq 0$ which contradicts $B_{-(n+1)}=0$.

\smallskip \noindent
\textbf{Case 2b:} $j<n$ and $x_{n-1}>x_{-(n-1)}$. Since, $x_{n-1}, y_{n-1} \in [1:x_{n-1}]$ and $x_{n-1}, x_n \in 
[x_{n-1}+1: x_n]$ and $y_{n-1}>x_n$, the $n/(n+1)$-subword of $B_{-n}(1:x_{n})$ is obtained from the $n/(n+1)$-subword 
of $A_{-n}(1:x_{n})$ by:
\begin{enumerate}
\item{Erasing an $n$ from $x_{-(n-1)}$ and adding an $n$ in $y_{-(n-1)}$. (Note $y_{-(n-1)} \leqslant x_{-(n-1)}$).}
\item{Adding an $n$ to $x_{n-1}$ and erasing an $n$ from $x_n$. (Note $x_{n-1} \leqslant x_n$).}
\item{Adding an $n+1$ to $x_n$. }
\end{enumerate}
Therefore, since the  $n/(n+1)$-subword of $A_{-n}(1:x_{n})$ contains an $(n,n+1)$-unbracketed $n$ and each one of
these three steps does not change that property, so does the $n/(n+1)$-subword of $B_{-n}(1:x_{n})$. This implies 
$f_n(B_{-n})\neq 0$ which contradicts $B_{-(n+1)}=0$.
\end{proof}

Since, indeed $x_j=x_{-j}$, we have $A_{-(n+1)}=B_{-n}$ by Claim~\ref{claim.A=B}, 
which completes the proof of Lemma~\ref{lemma.clean}.

\subsection{Proof of Lemma~\ref{lemma.clean} for $j=n$ and $j'=n-1$}
\label{section.proof clean 2}

\begin{lemma} \label{clean2}
Suppose $v$ is $I_0$-lowest weight and $h<n-1$.  Suppose that $(e_2 \cdots e_{n-1})e_1^h(v)\neq0$ and 
$e_2 \cdots e_n e_1^h(v)\neq0$. If $f_n^1f_{n}^{{1}}e_{\bar{1}}^{n}e_1^n(v)$ is $I_0$-lowest weight, then 
$f_n^1f_{n-1}^{{1}}e_{\bar{1}}^{n-1}e_1^n(v)$ is $I_0$-lowest weight.
\end{lemma}

\begin{proof}[Proof of Lemma~\ref{clean2}]
Suppose $v$ and $v'=f_n^1f_{n}^{{1}}e_{\bar{1}}^{n}e_1^h(v)$ are $I_0$-lowest weight and 
$(e_2 \cdots e_{n-1})e_1^h(v)\neq0$.  We must show that $f_n^1f_{n-1}^{{1}}e_{\bar{1}}^{n-1}e_1^h(v)$ is 
$I_0$-lowest weight.

\begin{claim}
\label{FR}
Given a word $W$, define $L(W)$ to be the length of the longest weakly increasing subsequence of $W$. If $V$ is 
$I_0$-lowest weight, and $W$ and $V$ are in the same $I_0$-connected component, then the number of $(n+1)$'s 
which appear in $V$ is equal to $L(W)$.
\end{claim}

\begin{proof}
This easily follows from analyzing the RSK insertion tableaux of the words.
\end{proof}

\begin{claim}
 \label{Is}
We have  $L(e_{\bar{1}}^{n-1}e_1^h(v)) \geqslant L(e_{\bar{1}}^{n}e_1^h(v))$.
\end{claim}

\begin{proof}
Since $Y=e_2 \cdots e_{n-1} e_1^h(v) \neq 0$, by inspection of the insertion tableaux of $v$ and $Y$ we observe that 
$\varphi_1(Y)=0$, $\varphi_2(Y)=1$, and $\varphi_k(Y)=0$ for all $k>2$.  This implies that $Y$ contains a letter $2$ 
which precedes all letters $1$.  Hence $e_{\bar{1}}^{n-1}e_1^h(v)=e_{-1}(Y) \neq 0$, so the statement 
$L(e_{\bar{1}}^{n-1}e_1^h(v)) \geqslant L(e_{\bar{1}}^{n}e_1^h(v))$ is well-defined.

We will now recycle notation from the proof of Section~\ref{section.proof clean 1} with slight changes.  
Let $X=e_1^h(v)$. For $2 \leqslant i \leqslant n+1$, set $A_i=(e_i \cdots e_n)(X)$ and  $B_i=(e_i \cdots e_{n-1})(X)$.  
Set $A_1=e_{-1}(A_2)$ and $B_1=e_{-1}(B_2)$.
Let $x_i$ be the integer which represents the position, where $A_{i+1}$ and $A_i$ differ and $y_i$ be the integer which 
represents the position where $B_{i+1}$ and $B_i$ differ. 

Suppose that $v$ contains $r$ letters $(n+1)$. It follows from weight considerations that $v'$ contains $(r+1)$ letters $(n+1)$.
This implies that $L(e_{\bar{1}}^{n}e_1^h(v))=r+1$ whereas $L(e_2 \cdots e_n e_1^h(v))=r$. This is to say $L(A_1)=r+1$ and
 $L(A_2)=r$. So $A_1$ contains a weakly increasing subsequence of length $r+1$, specified by the indices 
 $i_1^0, \ldots, i_1^{r}$.  We must have that $i_1^0 =x_1$ and that $A_1(i_1^1)=1$, otherwise the same indices 
 would specify a weakly increasing subsequence of $A_2$ of length $r+1$. It follows that $A_2$ has a weakly 
 increasing subsequence given by the indices $i_2^1, \ldots, i_2^{r}$ where $A_2(i_2^1)=1$.  Now suppose 
 $2 \leqslant k \leqslant n$ and $A_k$ has a weakly increasing subsequence given by the indices $i_k^1, \ldots, i_k^{r}$,
 where $A_k(i_k^1)=1$. If $x_k \notin \{ i_k^1, \ldots, i_k^{r}\}$, then $A_{k+1}$ has such a subsequence specified 
 by the same indices.  

Now suppose that  $x_k \in \{ i_k^1, \ldots, i_k^{r}\}$.  Create a list of indices as follows:
\begin{enumerate}
\item{If $i_k^j \leqslant x_k$ or $A_k(i_k^j) \neq k$, then $i_{k+1}^j=i_k^{j}$.}
\item{If $i_k^j> x_k$ and $A_k(i_k^j) = k$, then $A_k(i_k^j)$ is $(k,k+1)$-bracketed with some $k+1$ in a position 
between $x_k$ and $i_k^j$.  Let $i_{k+1}^j$ denote this position.}
\end{enumerate}

This creates a set $\{ i_{k+1}^1, \ldots, i_{k+1}^{r}\}$, which, after a possible reordering into increasing order, specifies 
a weakly increasing subsequence of $A_{k+1}$ with $A_{k+1}(i_{k+1}^1)=1$.   

By induction $B_n=A_{n+1}=X$ has a weakly increasing subsequence  specified by the indices $\{ {i'}_{n}^1, \ldots, {i'}_{n}^{r}\}$,
with $B_{n}({i'}_{n}^1)=1$.  Let $k>1$ and assume  $B_{k+1}$ has a weakly increasing subsequence  specified by the 
indices $\{ {i'}_{k+1}^1, \ldots, {i'}_{k+1}^{r}\}$, with $B_{k+1}({i'}_{k+1}^1)=1$.   If $y_k < {i'}_{k+1}^1$, then the same is 
true of $B_k$ with the same indices.  If $y_k > {i'}_{k+1}^1$ then $B_k=e_k(B_{k+1})=[B_{k+1}(1:{i'}_{k+1}^1) \,\,\, 
e_k(B_{k+1}({i'}_{k+1}^1+1:\ell))]$.  Since $B_{k+1}({i'}_{k+1}^1+1:\ell)$ has a weakly increasing subsequence of length $r-1$, 
$e_k(B_{k+1}({i'}_{k+1}^1+1:\ell))$ does as well.  Thus $B_k=[B_{k+1}(1:{i'}_{k+1}^1) \,\,\, e_k(B_{k+1}({i'}_{k+1}^1+1:\ell))]$
has a weakly increasing subsequence of length $r$ specified by some indices $\{ {i'}_{k}^1, \ldots, {i'}_{k}^{r}\}$, with 
$B_{k}({i'}_{k}^1)=1$ (where ${i'}_{k}^1={i'}_{k+1}^1$).  By induction this is true for $k=2$. Since $e_{-1}(B_2)=B_1$ is 
defined and since $B_2({i'}_{2}^1)=1$, we have $y_1<{i'}_{2}^1$ and so $\{y_1, {i'}_{2}^1, \ldots, {i'}_{2}^{r}\}$ is a list of 
indices  which give a weakly increasing subsequence of length $r+1$ in $B_1$.
\end{proof}

We want to show that $f_n^1f_{n-1}^{{1}}e_{\bar{1}}^{n-1}e_1^h(v)$ is $I_0$-lowest weight. Now $e_{-1}(Y)$ is obtained from 
$Y=e_2 \cdots e_{n-1} e_1^h(v)$ by changing its first $2$ to $1$.  As a result $\varphi_1(e_{-1}(Y)) \in \{1,2\}$ and 
$\varphi_k(e_{-1}(Y))=0$ for all $k>1$. Therefore, we may write $e_{-1}(Y)=e_1^s e_1^t(v^*)$ for some $I_0$-lowest weight 
element $v^*$, and $s \geqslant 0$ and $t>0$ with $t \geqslant s$ (using Lemma~\ref{lemma.gjk} when 
$\varphi_1(e_{-1}(Y))=2$). This gives $v^*=f_t^1f_s^1e_{\bar{1}}^{n-1}e_1^h(v)$.
Since $v'$ contains one more $n+1$ than $v$, it follows from Claims~\ref{FR} and~\ref{Is} that $v^*$ contains at least 
one more $n+1$ than $v$, which means we must have $t=n$.  This also means that $v$ and $v^*$ are not in the same 
connected $I_0$-component. But if $v=f_h^1f_{n-1}^{\bar{1}}e_{1}^{s}e_1^n(v^*)$ is in a different connected $
I_0$-component than $v^*$, then \textbf{C2'}(b) applies which forces $s=n-1$. Thus  
$v^*=f_n^1f_{n-1}^1e_{\bar{1}}^{n-1}e_1^h(v)$. 

This concludes the proof of Lemma~\ref{clean2}.
\end{proof}

\begin{proposition}
Lemma~\ref{lemma.clean} with $j=n-1$ and $j'=n$ and Lemma~\ref{clean2} imply Lemma~\ref{lemma.clean}.
\end{proposition}

\begin{proof}
We need to show that if $v$ is $I_0$-lowest weight, $e_1^{n-1}e_1^n(v)\neq0$, $e_1^ne_1^n(v)\neq0$, and 
$v^*=f_h^1f_{n}^{\bar{1}}e_1^{n}e_1^n(v)$ is $I_0$-lowest weight, then $f_{n-1}^{\bar{1}}e_1^{n-1}e_1^n(v)
=f_{n}^{\bar{1}}e_1^{n}e_1^n(v)$. Now $v=f_n^1f_n^1e_{\bar{1}}^ne_1^h(v^*)$ is $I_0$-lowest weight (in particular, 
$e_2 \cdots e_n e_1^h(v^*)\neq0$).  Now we show that $e_2 \cdots e_{n-1} e_1^h(v^*)\neq0$.  By definition, 
$e_1^h(v^*)\neq0$. Either $v^*$ has more $n$'s than $(n-1)$'s so that $e_2 \cdots e_{n-1} e_1^h(v^*)\neq0$, 
or else $v^*$ has the same number of $n$'s as $(n-1)$'s and $h=n-2$ in which case also 
$e_2 \cdots e_{n-1} e_1^h(v^*)\neq0$.  Therefore, by Lemma~\ref{clean2} $v'=f_n^1f_{n-1}^1e_{\bar{1}}^{n-1}e_1^h(v^*)$
is $I_0$-lowest weight. Rewriting this as $v^*=f_h^1f_{n-1}^{\bar{1}}e_{1}^{n-1}e_1^n(v')$ and noting that $\wt(v)=\wt(v')$ 
implies $e_{1}^{n}e_1^n(v') \neq 0$ Lemma~\ref{lemma.clean} with $j=n-1$ and $j'=n$ gives 
$v^*=f_h^1f_{n}^{\bar{1}}e_{1}^{n}e_1^n(v')$.  This implies that $v=v'$ and that hence that 
$f_{n-1}^{\bar{1}}e_1^{n-1}e_1^n(v)=f_{n}^{\bar{1}}e_1^{n}e_1^n(v)$.
\end{proof}

\bibliographystyle{alpha}
\bibliography{main}{}

\end{document}